\documentclass[10pt]{article}

\usepackage{times}
\usepackage{epsfig}
\usepackage{graphicx}
\usepackage{amsmath}
\usepackage{amssymb,amsthm}
\usepackage{mathtools}
\usepackage{dsfont}
\usepackage{xspace}
\usepackage{mathrsfs}
\usepackage[linesnumbered]{algorithm2e}
\usepackage{subfigure}
\usepackage{tikz}
\usepackage{pgfplots}
\usepackage{enumitem}
\usepackage{color}
\usepackage{wrapfig}  
\usepackage{marginnote}
\usepackage{graphpap}
\usepackage[left=2.5cm,right=2.5cm,top=2cm,bottom=3cm]{geometry}
\usepackage{hyperref}

\newtheorem{theorem}{Theorem}[section]
\newtheorem{lemma}[theorem]{Lemma}

\theoremstyle{definition}
\newtheorem{definition}[theorem]{Definition}

\theoremstyle{remark}
\newtheorem{remark}[theorem]{Remark}
\newtheorem{example}[theorem]{Example}

\newcommand{\R}{\mathbb{R}}
\newcommand{\N}{\mathbb{N}}
\newcommand{\Q}{\mathbb{Q}}
\newcommand{\Z}{\mathbb{Z}}

\DeclareMathOperator*{\argmin}{argmin}

\renewcommand{\d}{{\,\mathrm{d}}} 

\newcommand{\dist}{\mathrm{dist}}

\newcommand{\V}{\mathbf{V}}
\newcommand{\Y}{\mathbf{Y}}

\newcommand{\energy}{\mathcal{W}} 
\newcommand{\Qenergy}{\mathcal{Q}}

\newcommand{\pathenergy}{\mathcal{E}}
\newcommand{\Pathenergy}{\mathbf{E}}

\newcommand{\splineenergy}{\mathcal{F}}
\newcommand{\Splineenergy}{\mathbf{F}}

\newcommand{\constraintEnergy}{\mathcal{C}}

\newcommand{\metric}{g}

\newcommand{\manifold}{\mathcal{M}}
\newcommand{\x}{u}
\newcommand{\y}{y}

\newcommand{\cov}{{\tfrac{D}{\d t}}} 

\newcommand{\missing}[1]{{\color{magenta} {\bf [missing:}  {#1} {\bf ]}}}
\newcommand{\comment}[1]{{\color{cyan} {\bf [comment:}  {#1} {\bf ]}}}

\newcommand{\notinclude}[1]{}

\newcommand{\beq}{\begin{equation*}}
\newcommand{\eeq}{\end{equation*}}
\newcommand{\beqn}{\begin{equation}}
\newcommand{\eeqn}{\end{equation}}
\newcommand{\beqa}{\begin{eqnarray*}}
\newcommand{\eeqa}{\end{eqnarray*}}
\newcommand{\beqan}{\begin{align}}
\newcommand{\eeqan}{\end{align}}

\newcommand{\tgl}{{\mathrm{tgl}}}

\makeatletter
\DeclareRobustCommand\onedot{\futurelet\@let@token\@onedot}
\def\@onedot{\ifx\@let@token.\else.\null\fi\xspace}

\def\eg{\emph{e.g}\onedot} 
\def\ie{\emph{i.e}\onedot}

\def\etal{\emph{et al}\onedot}

\makeatother

\makeatletter
\def\namedlabel#1#2{\begingroup#2\def\@currentlabel{#2}\phantomsection\label{#1}\endgroup}
\makeatother

\definecolor{uniblau}{HTML}{004291}
\definecolor{bigsblau}{HTML}{365079}
\definecolor{bigsblau50}{HTML}{9CA7BC}
\definecolor{bigsblau25}{HTML}{CDD3DD}
\definecolor{uniorangedark}{HTML}{E6B400}
\definecolor{uniorange}{HTML}{FFCB0E}
\definecolor{uniorange!50}{HTML}{FFE586}
\definecolor{uniwhite}{HTML}{FFF2C2}
\definecolor{hcmgruen}{HTML}{567877}
\definecolor{hcmgruen50}{HTML}{AFBDBE}
\definecolor{hcmgruen25}{HTML}{D7DEDE}
\definecolor{himgrau}{HTML}{626566}
\definecolor{himgrau75}{HTML}{949592}
\definecolor{himgrau50}{HTML}{C5C4BE}
\definecolor{himgrau25}{HTML}{F5F5F5}
\definecolor{textgrau}{HTML}{000000}
\definecolor{black}{HTML}{000000}
\definecolor{white}{HTML}{FFFFFF}
\definecolor{sfbrot}{HTML}{CB4154}
\definecolor{lightgray}{gray}{0.9}
\colorlet{hcmrot}{sfbrot}
\colorlet{red}{sfbrot}
\colorlet{blue}{uniblau}
\colorlet{greyblue}{bigsblau}
\colorlet{hcmgelb}{uniorange}
\colorlet{yellow}{uniorange}
\colorlet{tafelgruen}{hcmgruen}
\colorlet{green}{hcmgruen}

\newcommand{\ym}{{\y_{k-1}}}
\newcommand{\yp}{{\y_{k+1}}}
\newcommand{\yt}{{\widetilde{\y}_{k}}}
\DeclareMathOperator*{\Gammalim}{\Gamma - \lim}

\DeclareMathOperator*{\Gammalimsup}{\Gamma - \limsup}

\newcommand{\revision}[1]{{#1}}

\newcommand{\changed}[1]{\textcolor{red}{#1}}

\definecolor{myOrange}{RGB}{255, 169, 87 }
\definecolor{myGreen}{RGB}{180, 255, 162  }

\begin{document}


\title{Variational time discretization of Riemannian splines}
\author{Behrend Heeren, Martin Rumpf, Benedikt Wirth}

\maketitle

\begin{abstract}
We investigate a generalization of cubic splines to Riemannian manifolds.
Spline curves are defined as minimizers of the spline energy---a combination of the Riemannian path energy and 
the time integral of the squared covariant derivative of the path velocity---under suitable interpolation conditions.
A variational time discretization for the spline energy leads to a constrained optimization problem over discrete paths on the manifold.
Existence of continuous and discrete spline curves is established using the direct method in the calculus of variations. 
Furthermore, the convergence of discrete spline paths to a continuous spline curve
follows from the $\Gamma$-convergence of the discrete to the continuous spline energy.
Finally, selected example settings are discussed, including splines on embedded finite-dimensional manifolds,
on a high-dimensional manifold of discrete shells with applications in surface processing, 
and on the infinite-dimensional shape manifold of viscous rods. 

\end{abstract}


\section{Introduction}\label{sec:intro}
In this paper we investigate a variational time discretization of spline curves on Riemannian manifolds.
To this end we extend the time-discrete geodesic calculus developed in~\cite{RuWi12b}.
The approach has already been presented without analysis in the context of the space of discrete shells in~\cite{HeRuSc16}, and
related classical interpolation and approximation tools in this shell space were discussed in \cite{HuPeRu17}.
Here, we investigate existence, regularity, and convergence properties of the approach under general assumptions, which are slightly stronger than 
those required in \cite{RuWi12b}. We further show examples of discrete spline curves 
on embedded two-dimensional manifolds in $\R^3$, on the space of discrete shells, and on an infinite-dimensional space 
of viscous rods.\medskip

Recently, Riemannian calculus on shape spaces has attracted a lot of attention.  
It in particular allows to transfer many important concepts from classical geometry to these usually high or even infinite-dimensional spaces.
Prominent examples with a full-fledged geometric theory are spaces of planar curves with a curvature-based metric \cite{MiMu04}, 
an elastic metric \cite{SrJaJo06} or Sobolev-type metrics \cite{ChKePo05,MiMu07,SuYeMe07}.
Geodesic paths in shape space can be considered as geometrically or physically natural
morphs from one shape into another.
Geodesic paths can be computed in closed form only for few nontrivial application-oriented Riemannian spaces (\eg \cite{YoMiSh08,SuMeSo11}).

In image processing the large deformation diffeomorphic metric mapping (LDDMM) framework proved to be a powerful tool 
underpinned with the rigorous mathematical theory
of diffeomorphic flows. In fact, Dupuis \etal \cite{DuGrMi98} showed that the resulting flow is actually a flow of diffeomorphisms.
In \cite{HaZaNi09}, Hart \etal exploited the optimal control perspective to the LDDMM model with the motion field as the underlying control.
Vialard \etal \cite{ViRiRu12a,ViRiRu12} studied methods from optimal control theory
to accurately estimate this initial momentum and to relate it to the Hamiltonian formulation of the geodesic 

In the context of geometry processing Kilian et al.~\cite{KiMiPo07}
studied geodesics in the space of triangular surfaces to interpolate between two
given poses.  The underlying metric is derived from the in-plane membrane
distortion.  Since this pioneering paper a variety of other Riemannian metrics
on the space of surfaces have been investigated
\cite{LiShDi10,KuKlDi10,BaBr11}.
In~\cite{HeRuWa12,HeRuSc14} a metric was proposed that takes the full elastic
responses including bending distortion into account.  
Brandt et al.~\cite{BrTyHi16} proposed an accelerated
scheme for the computation of geodesic paths in this shell space. 
\medskip
 
In Euclidean space cubic splines are known as minimizers of the total squared acceleration,  
due to a classical result by de Boor \cite{Bo63}
Analogously, \emph{Riemannian cubic splines} were introduced by Noakes et
al.~\cite{NoHePa89} as stationary paths of the integrated squared covariant derivative
of the velocity.  Subsequently, Camarinha et al.~\cite{CaSiCr01} proved a local
optimality condition and Giambo and Giannoni~\cite{GiGi02} established a 
global existence results. More recently, Trouv\'{e} and Vialard~\cite{TrVi12} studied 
spline interpolation on Riemannian manifolds with applications to 
time-indexed sequences of landmarks of 2D or 3D shapes.  Hinkle et al.~\cite{HiMuFl12}
investigated higher order Riemannian polynomials to perform polynomial
regression on Riemannian manifolds. 

Minimizing curve energies in an ambient space subject to the 
restriction of the curve to the manifold led to extrinsic variational formulations of splines. 
Wallner~\cite{Wa04} showed existence of minimizers in this setup for
finite dimensional manifolds, and Pottmann and Hofer~\cite{PoHo05} proved that
these minimizers are $C^2$.  
As an alternative to variational formulations 
interpolatory \emph{subdivision schemes} have been investigated on manifolds
exploiting the fact that subdivision schemes in Euclidean space are mostly based
on repeated local averages \cite{Dy02}. Wallner and
Dyn~\cite{WaDy05} showed that the resulting Riemannian cubic subdivision scheme
yields $C^1$ curves.  Recently, Wallner~\cite{Wa14} improved the regularity results for 
linear four-point scheme and other
univariate interpolatory schemes.  
\medskip

Here, we consider the variational definition and variational time discretization of Riemannian splines.
Let us briefly describe the main ideas and contributions of this work.
To define spline curves on a Riemannian manifold $\manifold$ in a variational way we follow the approach by Noakes \etal \cite{NoHePa89}. 
Splines are introduced as smooth curves $\y : [0, 1] \rightarrow \manifold$  that minimize the \emph{spline energy}
$$\splineenergy[\y] = \int_0^1   g_{\y(t)}\left(\cov\dot \y,\cov\dot \y\right) \d t$$ subject to some interpolation conditions and suitable 
boundary conditions (see Definitions\,\ref{def:energies} \& \ref{def:splineInterpolation}). Here $g$ denotes the Riemannian metric and $\cov v$ the covariant derivative of a vector field $v$ along the curve $\y$. 
We will show (Theorem\,\ref{thm:nonexistence}) that minimizers might not exist for this spline energy.
To overcome this problem we regularize the spline energy by combining it with the path energy and define a regularized spline curve as the minimizer of this new energy.
Those regularized spline curves can now be shown to exist via the direct method (Theorem\,\ref{thm:ExistenceContinuous}),
where the involved lower semi-continuity and the coercivity of the functional are rather intricate to show.
In fact, the lower semi-continuity imposes a quite strong requirement on the underlying Riemannian metric $g$, namely that any of its spatially nonconstant and nonquadratic components be continuous under weak convergence (see Definition\,\ref{def:admissibleMetric}).
In the same theorem we also show interior $C^{2,\frac12}$-regularity of spline interpolations, adapting classical elliptic regularity theory.

The discretization of (regularized) Riemannian spline curves will again be based on a (now discrete) variational model.
As a motivation consider the situation in Euclidean space,
in which the covariant derivative of the velocity field $v=\dot \y$ of a curve $\y: [0,1] \to \R^d$ 
coincides with $\ddot\y$. 
Given a uniform sampling $\y_k = \y(t^k)$ for $t^k = k\tau$, $k = 0,\ldots,K$, and time step $\tau = 1/K$,
the acceleration $\ddot \y(t^k)$ can be approximated by central second order difference quotients,
$\ddot \y(t^k)  \approx  \frac{2\y(t^k) - \y(t^{k-1}) - \y(t^{k+1})}{\tau^2}$.
Introducing the function $\energy[\y, \widetilde{\y}] = \|\y - \widetilde{\y}\|^2$ we thus obtain 
\begin{equation*}\label{eq:secOrderDiffQuotient}
  \left\| \ddot \y(t^k) \right\|^2 
  \approx 4K^4\, \left\| \y_k -\frac{ \y_{k-1} + \y_{k+1}}{2} \right\|^2 = 4K^4\, \energy[\y_k,\widetilde\y_k]\,,
\end{equation*}
where $\widetilde\y_k=\tfrac12 (\y_{k+1}+\y_{k-1})$ is the midpoint between $\y_{k-1}$ and $\y_{k+1}$
and one easily verifies that $\widetilde\y_k = \argmin_{\y}  \energy[\y_{k-1},\y] +  \energy[\y,\y_{k+1}]$.
The simple numerical quadrature $\int_0^1 f(t) \d t \approx  \tau \,\sum_{k=1}^{K-1} f(t^k)$ now yields
a straightforward approximation of the elastic functional in Euclidean space,
\begin{equation*}
\splineenergy[\y]
=\int_0^1 \|\ddot \y(t)\|^2 \d t 
\approx 4K^3 \sum_{k=1}^{K-1}  \energy[\y_k,\widetilde\y_k] \,.
\end{equation*}
Therefore, we define a discrete splines  as a $(K+1)$-tuples $(\y_0, \ldots, \y_K)\in\manifold^{K+1}$ that minimize the \emph{discrete spline energy}
\begin{equation*}\label{eq:dSpline}
F^K[\y_0, \ldots, \y_K] = 4K^3\, \sum_{k=1}^{K-1}  \energy[\y_k, \widetilde{\y}_k]
\qquad\text{with }\widetilde{\y}_k = \arg \min_\y \Big(  \energy[\y_{k-1},\y] +  \energy[\y,\y_{k+1}] \Big)
\end{equation*}
subject to appropriate interpolation and boundary conditions (see Definitions\,\ref{def:discreteEnergies} \& \ref{def:discreteSplineInterpolation}).
On Riemanian manifolds we consider $\energy[\y,\widetilde\y]$ to be (an approximation of) the squared Riemannian distance between $\y$ and $\widetilde\y$.
Then $F^K[\y_0, \ldots, \y_K]$ turns into a functional which variationally describes discrete Riemannian splines, where 
$\widetilde\y_k=\argmin_\y\energy[\y_{k-1},\y] +  \energy[\y,\y_{k+1}]$ is the (approximate) geodesic midpoint of $\y_{k-1}$ and $\y_{k+1}$.
As in the continuous case, the discrete spline energy is regularized by adding a small amount of a discrete path energy.
Mimicking the approach and the conditions of the continuous setting, we prove existence of minimizers of this regularized discrete spline energy (Theorem\,\ref{thm:ExistenceDiscrete}).

Finally we prove that the (regularized) discrete spline energy is indeed an approximation of the (regularized) continuous spline energy in the sense of $\Gamma$-convergence (Theorem\,\ref{thm:GammaConvergence}).
Here the major effort lies in deriving the consistency of the discrete spline energy with minimal regularity requirement.
Ultimately, as a consequence we also prove that discrete spline interpolations converge against continuous spline interpolations (Theorem\,\ref{thm:equicoercivity}).
\medskip

The paper is  organized as follows.
At first we recall in Section\,\ref{sec:splines}
the Riemanian path energy and define a Riemannian spline energy under suitable assumptions on the 
Riemannian manifolds, which also admit a class of infinite-dimensional Hilbert manifolds.
Existence of energy minimizers is also investigated.
In Section\,\ref{sec:splinesdiscrete} we develop the variational time discretization for Riemannian splines
and prove existence of discrete spline curves.
The approach we follow here and expand towards Riemannian splines 
is based on the concept and analysis of the discrete Riemannian calculus in \cite{RuWi12b}.
In Section\,\ref{sec:gamma} the convergence of discrete splines to continuous splines is justified via $\Gamma$-convergence.
Finally, in Section\,\ref{sec:numerics} we discuss some applications and show numerical results in the case of embedded finite-dimensional manifolds, 
a high-dimensional manifold of discrete shells, and the infinite-dimensional manifold of viscous rods.  

\section{Geodesics and splines on a Riemannian manifold}\label{sec:splines}

In this section we briefly recapitulate the notion of geodesics on (possibly infinite-dimensional) Riemannian manifolds
and introduce a definition of cubic splines on those manifolds as curves minimizing a particular spline energy.
We then analyse the well-posedness of this definition using variational techniques.

To be mathematically precise, we first fix an abstract setting in which we shall work, and review a few differential geometric concepts needed to state the spline energy. The following definition of an admissible metric coincides with that in \cite{RuWi12b} except that we additionally require a particular splitting of $g$ into a compact and a quadratic component.

\begin{definition}[Admissible manifold]\label{def:admissibleManifold}
Let $\V$ be a separable Hilbert space that is compactly embedded in a real Banach space $\Y$, \revision{\ie the identity mapping $\mathrm{id}: \V \to \Y$ is a compact linear operator}.  
Then we call $\manifold\equiv\V$ an \emph{admissible} (Hilbert) manifold, thus
\begin{equation*}
\manifold=\V\hookrightarrow\Y\,.
\end{equation*}
\end{definition}

\begin{definition}[Admissible metric]\label{def:admissibleMetric}
A Riemannian metric $g:\bigcup_{y \in \manifold} \left(\{y\} \times T_\y\manifold\times T_\y\manifold\right)\to\R$ on the admissible manifold $\manifold$ shall be called \emph{admissible}
if it can be extended to a function $g:\Y\times\V\times\V\to\R$ of the form
$$
g_\y(v,w) =  g^c_\y(v,w) + \Qenergy(v,w)
$$
for some compact part $g^c$, which depends on the position $\y$, and a quadratic part $\Qenergy$,
where the following hypotheses shall be satisfied for all $v\in\V$.
\begin{enumerate}[label=(\roman*)]
\item\label{enm:boundedness_gc} $g^c$ is symmetric in the last two arguments and $g^c_\y(v,v) \leq C^\ast \|v\|^2_\Y$ for some constant $C^\ast$.
\item\label{enm:compactness} $g^c$ is twice differentiable with bounded derivatives as a function $g^c:\Y\to\Y'\otimes\Y'$.
\item\label{enm:boundedness_Q} $\Qenergy$ is symmetric positive semi-definite and bilinear on $\V\times \V$ with $\Qenergy(v,v) \leq C^{\ast\ast} \|v\|^2_\V$ for some constant $C^{\ast\ast}$\,.
\item\label{enm:coercivity} $g$ is uniformly coercive with respect to the $\V$ norm, \ie $ c^\ast \|v\|^2_\V \leq  g_\y(v,v)$ for some $c^\ast>0$.
\end{enumerate}
Here, $\V'$ and $\Y'$ denote the dual spaces to $\V$ and $\Y$, and $\V'\otimes\V'$ and $\Y'\otimes\Y'$ are equipped with the topology induced by the injective cross norm.
\end{definition}
\begin{remark}[Weaker differentiability condition]
Condition \ref{enm:compactness} can be relaxed to only require bounded first derivatives of $g^c$ as a function $g^c:\Y\to\Y'\otimes\Y'$ and bounded second derivatives of $g^c$ as a function $g^c:\Y\to\V'\otimes\V'$.
All arguments remain valid in this case without modifications.
\end{remark}
\begin{remark}[Modulus of continuity]
As a direct consequence of \ref{enm:compactness} there exists a strictly increasing
continuous function $\beta$ with $\beta(0)=0$ such that
\begin{equation*}
| g^c_\y(v,v) - g^c_{\widetilde{\y}}(v,v)| \leq \beta(\| \y - \widetilde{\y}\|_\Y) \|v\|^2_\Y\text{ for all }\y,\, \widetilde{\y}  \in \Y\text{ and all }v \in \V.
\end{equation*}
\end{remark}
\begin{remark}[Finite-dimensional admissible Riemannian manifolds]
As the simplest example, $d$-dimensional manifolds $\manifold$ parameterized over $\R^d$ with a smooth metric $g:\R^d\times\R^d\times\R^d$ are admissible
with $\Y=\V=\R^d$ and $g^c=g$ as well as $\Qenergy=0$.
\end{remark}
\begin{remark}[Generalized admissible manifolds]\label{rem:manifoldBoundary}
Our admissible manifolds have no boundary and necessarily have the same topology as the underlying Hilbert space $H$.
However, some results can be carried over to manifolds with different topology or with boundary,
essentially by reduction to the case of admissible manifolds via local charts.
We will make corresponding comments when discussing our shape space examples in Section\,\ref{sec:numerics}.
\end{remark}
Given an admissible Riemannian manifold $(\manifold,g)$, let us next introduce the Christoffel operator $\Gamma_\y: T_\y\manifold\times T_\y\manifold \to T_\y\manifold$ by
\beqn\label{eq:Gamma}
2 g_\y(\Gamma_\y(v,w),z) = \left(D_\y g_\y\right)(w)(v,z) - \left(D_\y g_\y\right)(z)(v,w) + \left(D_\y g_\y\right)(v)(w,z) \,,
\eeqn
where $D_\y$ denotes the derivative with respect to $\y$ and thus $D_\y g_\y = D_\y g_\y^c$ under the above assumptions.
Symmetry of the metric $g$ implies symmetry of the Christoffel operator.
\begin{remark}[Implications of metric decomposition]
The decomposition of $g$ into $g^c$ and $\Qenergy$ will be necessary to establish weak continuity of the associated Christoffel operator (see Lemma \ref{thm:weakcontGamma}), which is required for an existence result of Riemannian splines.
For a general $g_\y$, which is uniformly bounded and positive definite on $\V\times \V$ for all $\y \in \Y$ but nonlinear in $\y$,
the right-hand side of \eqref{eq:Gamma} will in general not be weakly continuous jointly in $v$ and $w$ on infinite-dimensional spaces $\V$.
\end{remark}
The covariant derivative $\cov w$ of a tangential vector field $w:[0,1]\to\V$ along a path $\y:[0,1]\to\manifold$ with $w(t)\in T_{\y(t)}\manifold$ now can be defined via
\beqn\label{eq:cov}
g_{\y(t)}\left(\cov w(t), z\right) = g_{\y(t)}\left(\dot w(t), z\right)+ g_{\y(t)}(\Gamma_{\y(t)}(w(t),\dot \y(t)),z)
\qquad\text{for all }z\in T_{\y(t)}\manifold\,,
\eeqn
where a dot denotes differentiation with respect to time $t$.
Note that the covariant derivative along curves allows to define a connection $\nabla$ on the manifold.
Indeed, for $v\in T_\y\manifold$ and a tangent vector field $\tilde w:\manifold\to\V$ with $\tilde w(\widetilde{\y})\in T_{\widetilde{\y}}\manifold$ for all $\widetilde{\y}\in\manifold$ we define
\beqn\label{eq:connection}
\nabla_v\tilde w = \cov(\tilde w\circ\y)(0)\,,
\eeqn
where $t\mapsto \y(t)$ is any path with $\y(0)=\y$ and $\dot \y(0) = v\in T_\y\manifold$.

Now we are in a position to define the central energies and to state the spline interpolation problems.
\begin{definition}[Path and spline energy]\label{def:energies}
The \emph{path energy} $\pathenergy$ and the \emph{spline energy} $\splineenergy$ are given as
\begin{align}
\pathenergy[\y] &= \int_0^1 \metric_{\y(t)}\left( \dot \y(t), \dot \y(t) \right) \d t\,, \label{eq:pathenergy}\\
\splineenergy[\y] &= \int_0^1   g_{\y(t)}\left(\cov\dot \y(t),\cov\dot \y(t)\right) \d t\,,\label{eq:splineenergy}
\end{align}
both defined for curves $\y:[0,1]\to\manifold$.
The \emph{(regularized) spline energy} is defined as $\splineenergy^\sigma=\splineenergy+\sigma\pathenergy$ for some $\sigma\geq0$.
\end{definition}

Minimizers of the path energy with fixed end points $\y(0)$ and $\y(1)$ are called geodesics, and they can be shown to exist for any given $\y(0),\y(1)\in\manifold$ \cite{RuWi12b}.
Our interest here is not the interpolation between only two points $\y(0),\y(1)\in\manifold$, though.
Rather we aim to find curves $\y:[0,1]\to\manifold$ minimizing the path or spline energy under a set of $I>2$ interpolation constraints
\begin{equation}\label{eq:IC}
\y(t_i) = \bar \y_i\, , \quad i= 1, \ldots, I,
\end{equation}
for fixed and pairwise different 
$t_i\in[0,1]$ and $\bar\y_i\in\manifold$, $i=1,\ldots,I$, with $t_1<\ldots<t_I$.
In addition we may impose one of the following boundary conditions,
\begin{align}
 &\text{natural b.\,c.:}&& \cov\dot \y(0) = \cov\dot \y(1) = 0,\label{eqn:naturalBC}\\
 &\text{Hermite b.\,c.:}&& \dot \y(0) = v_0, \quad \dot \y(1) = v_1 \quad \text{ for given }v_0\in T_{\y(0)}\manifold,v_1 \in T_{\y(1)}\manifold\,,\label{eqn:HermiteBC}\\
 &\text{periodic b.\,c.:}&& \y(0) = \y(1), \quad \dot \y(0) = \dot \y(1)\,.\label{eqn:periodicBC}
\end{align}
In case of the Hermite boundary condition we in addition assume that $t_0=0$ and $t_I=1$ so that $\y(0)$ and $\y(1)$ are also prescribed.
In case of the periodic boundary condition, on the other hand, we impose in addition $t_1\neq0$ or $t_I\neq1$ since otherwise the interpolation task would be overdetermined.
\begin{definition}[Geodesic and spline interpolation]\label{def:splineInterpolation}
For given data points $t_i\in[0,1]$ and $\bar\y_i\in\manifold$, $i=1,\ldots,I$,
a \emph{piecewise geodesic interpolation} $\y$ is defined as a minimizer of the path energy under the interpolation constraints \eqref{eq:IC},
\begin{equation*}
\y\in\argmin\{\pathenergy[\widetilde{\y}]\,|\,\widetilde{\y}:[0,1]\to\manifold \text{ with } \widetilde \y \in W^{1,2}((0,1);\V)  \text{ and \eqref{eq:IC}}\}\,,
\end{equation*}
while we define a \emph{spline interpolation} $\y$ as a minimizer of the (regularized) spline energy under \eqref{eq:IC} and an additional boundary condition,
\begin{equation*}
\y\in\argmin\{\splineenergy^\sigma[\widetilde{\y}]\,|\,\widetilde{\y}:[0,1]\to\manifold \text{ with } \widetilde \y \in W^{2,2}((0,1);\V)  \text{ as well as \eqref{eq:IC} and one of \eqref{eqn:naturalBC}-\eqref{eqn:periodicBC}}\}\,.
\end{equation*}
\end{definition}

Before we analyse the well-posedness of the above definitions, a few remarks are in order.

\begin{remark}[Euclidean space]
In the Euclidean case $\manifold=\V=\Y=\R^m$, the metric represents the Euclidean inner product $g_\y(v,w)=v\cdot w$, and the covariant derivative simplifies to $\cov\dot \y(t)=\ddot\y(t)$.
In that case it is easy to see that piecewise geodesic equals piecewise linear interpolation
and that spline interpolation with $\sigma=0$ is unique (due to the result by de Boor \cite{Bo63}) and coincides with standard cubic spline interpolation.
\end{remark}

\begin{remark}[Well-posedness]
The existence of geodesics between two points by \cite{RuWi12b} immediately implies existence of a piecewise geodesic interpolation.
Likewise, spline interpolation with $\sigma=0$ is known to have a unique solution in linear spaces $\manifold$ \cite{Bo63}.
However, below we shall see that on nonlinear Riemannian manifolds $\manifold$, existence of a spline interpolation can in general only be guaranteed for $\sigma>0$.
\end{remark}

\begin{remark}[Relation between geodesic and spline]
The Euler--Lagrange equation of a geodesic interpolation $\y$ reads
\begin{equation} \label{eq:ELpathenergy}
0
= \partial_\y \pathenergy[\y](\vartheta)
= \int_0^1 (D_\y g_\y)(\vartheta) (\dot \y,\dot \y) + 2 g_\y(\dot \y, \dot \vartheta) \d t
= \int_0^1 (D_\y g_\y)(\vartheta) (\dot \y,\dot \y) - 2 (D_\y g_\y)(\dot \y)(\dot \y, \vartheta) - 2 g_\y (\ddot \y,  \vartheta)\d t
\end{equation}
for all smooth tangent vector fields $t\mapsto\vartheta(t)\in T_{\y(t)}\manifold$ with $\vartheta(t_i)=0$, $i=1,\ldots,I$, where in the last step we integrated by parts.
By the fundamental lemma of the calculus of variations we thus arrive at
\begin{equation*}
g_\y(\ddot \y, \vartheta)
=  -(D_\y g_\y)(\dot \y)(\dot \y,  \vartheta)  + \tfrac12(D_\y g_\y)(\vartheta) (\dot \y,\dot \y)
= g_\y(-\Gamma_\y(\dot \y, \dot \y),\vartheta)
\end{equation*}
for all $\vartheta$.
Thus, the necessary condition for $\y$ to be piecewise geodesic is that the velocity field $\dot \y$ is parallel along $\y$, that is,
\beq
\cov \dot \y = 0\,.
\eeq
In physical terms $\cov \dot \y$ is the acceleration along the path.  
The spline energy therefore penalizes any deviation from zero acceleration.
As a consequence, geodesic and spline interpolation coincide for just two interpolation points $\y(0),\y(1)\in\manifold$.
\end{remark}

\begin{remark}[Natural boundary conditions] 
Since a spline interpolation $\y$ may initially only be expected to have Sobolev-regularity $W^{2,2}$ in time, the natural boundary condition \eqref{eqn:naturalBC} does not make sense a priori.
However, as the name implies, it is the natural condition on $\y$ that arises if we impose no boundary condition at all (which is how we shall interpret \eqref{eqn:naturalBC} in our analysis).
Indeed, let $\y:[0,1]\to\manifold$ be a spline interpolation without any boundary condition and $t\mapsto\vartheta(t)\in T_{\y(t)}\manifold$ be a smooth perturbation of $\y$ which is only nonzero near $t=0$ and $t=1$.
The spline energy can be rewritten as 
\begin{equation}
\label{eq:spline2}
\splineenergy[\y] = \int_0^1  g_\y(\ddot \y, \ddot \y) + 2 g_\y(\ddot \y, \Gamma_\y(\dot \y,\dot \y)) +g_\y( \Gamma_\y(\dot \y,\dot \y), \Gamma_\y(\dot \y,\dot \y))\d t\,.
\end{equation}
Hence, the Euler--Lagrange equation for a regularized spline curve takes the form
\begin{alignat}{2}
0 = \partial_\y\splineenergy^\sigma[\y](\vartheta) &= \sigma\partial_\y\pathenergy[\y](\vartheta) + 
\int_0^1 && 2 g_\y(\ddot \vartheta, \ddot \y + \Gamma_\y(\dot \y,\dot \y))+ 2 g_\y(\ddot \y,2\Gamma_\y(\dot \vartheta,\dot \y)+(D_y \Gamma_\y)(\vartheta)(\dot \y,\dot \y))\nonumber\\
&&& + (D_\y g_\y)(\vartheta) (\ddot \y, \ddot \y + 2 \Gamma_\y(\dot \y,\dot \y)) \nonumber\\
&&& + 2g_\y(\Gamma_\y(\dot \y,\dot \y),2\Gamma_\y(\dot \vartheta,\dot \y)+(D_y \Gamma_\y)(\vartheta)(\dot \y,\dot \y)) \nonumber\\
&&&+ (D_\y g_\y)(\vartheta)( \Gamma_\y(\dot \y,\dot \y), \Gamma_\y(\dot \y,\dot \y))\d t \nonumber\\
&= \sigma\partial_\y\pathenergy[\y](\vartheta) + 
\int_0^1 && 2 g_\y(\ddot \vartheta,\ddot \y + \Gamma_\y(\dot \y,\dot \y)) + h(\y,\dot \y, \ddot \y, \vartheta, \dot \vartheta) \d t \label{eq:ELSpline}
\end{alignat}
for some function $h:\Y \times\V\times\V \times\Y \times \V\to\R$ and for all test functions $\vartheta$.
Above, $(D_\y \Gamma_\y)$ denotes the variation with respect to the (implicit) argument $\y$, and we used that $\Gamma_\y$ is bilinear in its arguments.
Applying integration by parts we obtain
\begin{multline*}
0= \sigma\partial_\y\pathenergy[\y](\vartheta) +\int_0^1 h(\y,\dot \y, \ddot \y, \vartheta, \dot \vartheta)-2g_\y(\dot\vartheta,\dddot\y+(D_\y\Gamma_\y)(\dot\y)(\dot\y,\dot\y)+2\Gamma_\y(\ddot\y,\dot\y))-2(D_\y g_\y)(\dot\y)(\dot\vartheta,\cov\dot\y) \d t \\
+  \left[2g_{y(t)}(\dot\vartheta(t),\cov\dot\y(t))\right]_{t=0}^{t=1}\,,
\end{multline*}
where we used the relation $\frac\d{\d t}\cov\dot\y=\dddot\y+(D_\y\Gamma_\y)(\dot\y)(\dot\y,\dot\y)+2\Gamma_\y(\ddot\y,\dot\y)$ which is obtained from differentiating the definition of $\cov\dot\y$.
Since $\vartheta$ is arbitrary and in particular $\dot\vartheta(0),\dot\vartheta(1)$ may take any value in $T_{\y(0)}\manifold$ and $T_{\y(1)}\manifold$, respectively,
we must have $\cov\dot\y=0$ at $t=0$ and $t=1$.
\end{remark}

\begin{remark}[Periodic boundary conditions]
Imposing periodic boundary conditions is equivalent to defining the curve $\y$ and its (regularized) spline energy $\splineenergy^\sigma$ over $S^1$ instead of the interval $(0,1)$,
in which case time differentiation $\dot\y$ just has to be reinterpreted as differentiation with respect to the angle variable, scaled by $2\pi$.
Indeed, under the identification of $(0,1)$ with $S^1$ via the mapping $t\mapsto(\cos(2\pi t),\sin(2\pi t))$,
the set $\{\y\in W^{2,2}((0,1);\V)\,|\,\y(0)=\y(1),\dot\y(0)=\dot\y(1)\}$ coincides with the function space $W^{2,2}(S^1;\V)$
so that the domain of the regularized spline energy on $(0,1)$ with periodic boundary conditions and the domain of the regularized spline energy on $S^1$ coincide.
\end{remark}

We next show that spline interpolations with $\sigma=0$ do not exist in general, but that spline interpolation with positive $\sigma$ is indeed well-posed.

\begin{lemma}[Nonexistence of Riemannian splines]\label{thm:nonexistence} Let $\manifold$ be any manifold with a closed geodesic curve $C$ and a point $\bar\y_1\in C$ such that any locally geodesic curve connecting $\bar\y_1$ with itself lies inside $C$.
Then, minimizers of $\splineenergy^0[\y]=\splineenergy[\y]$ under the interpolation contraints \eqref{eq:IC} do not exist in general.
\end{lemma} 
\begin{proof}
It suffices to provide a counterexample.
Initially, consider $\manifold$ to be a cylinder in $\R^3$ of infinite length and perimeter $1$.
(Note that this manifold actually is not admissible in the sense of Definition\,\ref{def:admissibleManifold},
but the spline energy is nevertheless defined on curves $\y:[0,1]\to\manifold$, and the example of a cylinder only serves to illustrate the general argument.
Furthermore, Remark\,\ref{rem:problematicManifolds} will give an example of a manifold admissible in the sense of Definition\,\ref{def:admissibleManifold} and satisfying the conditions of this lemma.)
Let $t_1=0$, $t_2=r\in(0,1)\setminus\Q$, $t_3=1$, and choose an arbitrary point $\bar\y_1=\bar\y_3\in\manifold$. 
Let $\bar\y_2\in\manifold$ be the point opposite $\bar\y_1$.

Now we have $\inf_\y\splineenergy[\y]=0$.
Indeed, let us arclength-parameterize the circle through $\bar\y_1,\bar\y_2,\bar\y_3$ by a $1$-periodic function $\xi:\R\to\manifold$ with $\xi(0)=\bar\y_1$
and consider the Euclidean cubic spline $x:[0,1]\to\R$ with $x(t_1)=x_1=0$, $x(t_2)=x_2=m+\frac12$, and $x(t_3)=x_3=n$ for some $m,n\in\N$.
Obviously, $\y=\xi\circ x$ is a curve on $\manifold$ satisfying \eqref{eq:IC}, and its spline energy equals
\begin{equation*}
\splineenergy[\y]
=\int_0^1|\ddot x(t)|^2\d t
=\frac{3((x_2-x_1)(t_3-t_1)-(x_3-x_1)(t_2-t_1))^2}{(t_3-t_2)^2(t_3-t_1)(t_2-t_1)^2}
=\frac{3(m+\frac12-rn)^2}{(1-r)^2r^2}\,,
\end{equation*}
where we used that the energy of a Euclidean cubic spline can be explicitly computed.
By Dirichlet's approximation theorem there exist $m,n\in\Z$ that make the above arbitrarily small.

However, there is no curve $\y$ with $\splineenergy[\y]=0$.
Indeed, such a curve would satisfy $\cov\dot\y=0$, which on the cylinder results in a regular helix with constant speed.
Since $\y(0)=\y(1)$, the helix is degenerate and winds round the circle at constant speed so that necessarily $\y(t)=\xi(mt)$ for some $m\in\Z$.
However, the preimage of $\bar\y_2$ under $\y$ does not contain $r$ so that \eqref{eq:IC} is violated, $r\notin\y^{-1}(\bar\y_2)=\{\frac1{2m},\frac3{2m},\ldots,\frac{2m-1}{2m}\}$.

This construction can easily be transferred onto a general manifold with a closed geodesic, where the circle is replaced by the closed geodesic and
the interpolation conditions are chosen correspondingly.
Indeed, the infimum of the spline energy on an analogous sequence of curves, now mapping onto the closed geodesic, vanishes.
Hence, any minimizer, if it exists, must be a (local) geodesic characterized by $\cov \dot \y = 0$ and fulfilling the interpolation conditions. 
The above argument shows that this is impossible.
\end{proof}

\begin{remark}[Examples of admissible manifolds with the property from Lemma\,\ref{thm:nonexistence}]\label{rem:problematicManifolds}
The above-used class of manifolds with closed geodesics also contains manifolds which are \revision{infinitely smooth and globally homeomorphic to Euclidean space, and whose distance metric is equivalent to the Euclidean metric in the sense that one can be bounded by the other up to a constant factor}.
Indeed, we can cut a cylinder segment out of an infinite cylinder and close one end smoothly with something like a spherical cap, while the other one is smoothly blended into the plane $\R^2$.
\end{remark}

Before, we prove existence of minimizers in the case $\sigma>0$, let us establish a weak continuity result for the Christoffel operator, which will imply weak lower semi-continuity of the (regularized) spline energy.
\begin{lemma}[Weak continuity of the Christoffel operator]\label{thm:weakcontGamma}
On an admissible Riemannian manifold $(\manifold,g)$ the Christoffel operator is weakly continuous in the sense
\begin{equation*}
\Gamma_{\y_k}(v_k,w_k) \to \Gamma_\y(v,w)\text{ in }\V\text{ as }k\to\infty
\end{equation*}
for $\y_k \to \y$ strongly in $\Y$ and $(v_k,w_k) \rightharpoonup (v,w)$ weakly in $\V\times \V$.
In more detail,
\begin{multline}\label{eqn:ChristoffelConvergence}
\|\Gamma_{\y_k}(v_k,w_k) - \Gamma_\y(v,w)\|_\V\\
\leq C\left(\|w_k\|_\Y\|v_k-v\|_\Y+\|w_k-w\|_\Y\|v\|_\Y+\|\y_k-\y\|_\Y\|w\|_\V\|v\|_\V+\|\y_k-\y\|_\Y\|\Gamma_{\y_k}(v_k,w_k)\|_\Y\right)
\end{multline}
for some constant $C>0$ only depending on $c^\ast$ and the derivative bounds on $g^c$ from Definition\,\ref{def:admissibleMetric}.
\end{lemma} 
\begin{proof}
For $z\in\V$ let us define 
\begin{align*}
l_k(z) &= \left(D_\y g^c_{\y_k}\right)(w_k)(v_k,z) - \left(D_\y g^c_{\y_k}\right)(z)(v_k,w_k) + \left(D_\y g^c_{\y_k}\right)(v_k)(w_k,z)\\
l(z) &= \left(D_\y g^c_\y\right)(w)(v,z) - \left(D_\y g^c_\y\right)(z)(v,w) + \left(D_\y g^c_\y\right)(v)(w,z)
\end{align*}
and note that
\begin{equation*}
\|l_k-l\|_{\V'}
\leq\|D_\y g^c\|\left(\|w_k\|_\Y\|v_k-v\|_\Y+\|w_k-w\|_\Y\|v\|_\Y\right)+\|D^2_\y g^c\|\|\y_k-\y\|_\Y\|w\|_\V\|v\|_\V\,,
\end{equation*}
where $\|D_\y g^c\|$ and $\|D_\y^2 g^c\|$ shall denote the bounds from Definition\,\ref{def:admissibleMetric}\ref{enm:compactness} on the first and second derivative of $g^c$.
Indeed, we have
\begin{align*}
&\left|\left(D_\y g^c_{\y_k}\right)(w_k)(v_k,z)-\left(D_\y g^c_\y\right)(w)(v,z)\right|\\
&\leq\left|\left(D_\y g^c_{\y_k}\right)(w_k)(v_k-v,z)\right|+\left|\left(D_\y g^c_{\y_k}\right)(w_k-w)(v,z)\right|+\left|\left(D_\y g^c_{\y_k}-D_\y g^c_\y\right)(w)(v,z)\right|\\
&\leq\|D_\y g^c\|\|w_k\|_\Y\|v_k-v\|_\Y\|z\|_\Y+\|D_\y g^c\|\|w_k-w\|_\Y\|v\|_\Y\|z\|_\Y+\|D_\y^2 g^c\|\|\y_k-\y\|_\Y\|w\|_\Y\|v\|_\V\|z\|_\V\,,
\end{align*}
and analogous estimates are readily derived for the other terms in $l_k$, which results in the desired bound for $\|l_k-l\|_{\V'}$.

By definition of the Christoffel operator and Definition\,\ref{def:admissibleMetric} we have
$2g_{\y_k}(r_k,z)= l_k(z)$ and $2g_\y(r,z)= l(z)$ for $r_k=\Gamma_{\y_k}(v_k,w_k)$ and $r=\Gamma_{\y}(v,w)$
so that with Definition\,\ref{def:admissibleMetric}\ref{enm:coercivity} we obtain
\begin{align*}
\|r_k-r\|_\V^2 &\leq \frac1{c^\ast} g_\y(r_k-r,r_k-r) \\
& = \frac1{c^\ast} \left(g_{\y_k}(r_k, r_k-r)-g_\y(r,r_k-r) + \left(g_\y-g_{\y_k}\right) (r_k,r_k-r)\right)\\
&\leq \frac1{c^\ast} \left(\tfrac12(l_k-l)(r_k-r) + \|D_\y g^c\| \|\y_k-\y\|_\Y \|r_k\|_\Y \|r_k-r\|_\Y\right)\,,
\end{align*}
which implies $\|r_k-r\|_\V \leq \frac1{c^\ast} (\tfrac12\|l_k-l\|_{\V'} + \|D_\y g^c\|\|\y_k-\y\|_\Y\|r_k\|_\V)$.
Inserting the previous bound on $\|l_k-l\|_{\V'}$ yields \eqref{eqn:ChristoffelConvergence}.
If we can show uniform boundedness of $\|r_k\|_\Y$,
then the desired convergence $r_k\to r$ follows from \eqref{eqn:ChristoffelConvergence} and the strong convergence $(v_k,w_k)\to(v,w)$ in $\Y\times\Y$ due to the compactness of the embedding $\V\hookrightarrow\Y$.
However, inequality \eqref{eqn:ChristoffelConvergence} together with the inverse triangle inequality implies
$\|r_k\|_\V-\|r\|_\V\leq(1+\|r_k\|_\Y)o(1)\leq(1+\|r_k\|_\V)o(1)$,
where $o(1)$ indicates a term converging to zero.
Consequently we must have $\limsup_{k\to\infty}\|r_k\|_\V\leq\|r\|_\V$ and thus $\limsup_{k\to\infty} \|r_k-r\|_\V =0$ , as required.
\end{proof}

\begin{lemma}[Continuity properties of spline energy]\label{thm:energyLSC}
For $(\manifold,g)$ admissible, the regularized spline energy $\splineenergy^\sigma$ is lower semi-continuous under weak and continuous under strong convergence in $W^{2,2}((0,1);\V)$.
\end{lemma}
\begin{proof}
First recall that Sobolev space $W^{2,2}\equiv W^{2,2}((0,1);\V)$ embeds continuously into $C^{1,\frac12}([0,1];\V)$ and compactly into $C^{1,\alpha}([0,1];\Y)$ for any $\alpha\in(0,\frac12)$ by the Arzel\`a--Ascoli Theorem.
Furthermore, we notice that the Christoffel operator is bilinear and bounded, $\|\Gamma_\y(v,w)\|_\V\leq\frac32\frac{\|D_\y g^c\|}{c^\ast}\|v\|_\Y\|w\|_\Y$, where $\|D_\y g^c\|$ is the bound from Definition\,\ref{def:admissibleMetric}\ref{enm:compactness} on the first derivative of $g^c$ (just insert $z=\Gamma_\y(v,w)$ in \eqref{eq:Gamma})
so that $\|\Gamma_\y(\dot\y,\dot\y)\|_\V\leq\frac32\frac{\|D_\y g^c\|}{c^\ast}\|\dot\y\|_\Y^2$.

Next consider a weakly converging sequence $\y_k\rightharpoonup\y$ in $W^{2,2}$.
By the above remarks we may assume in addition that  $\y_k\to\y$ strongly in $C^{1,\alpha}([0,1];\Y)$.
We now show $\liminf_{k\to\infty}\splineenergy[\y_k]+\sigma\pathenergy[\y_k]\geq\splineenergy[\y]+\sigma\pathenergy[\y]$.
Indeed, the lower semi-continuity of $\pathenergy$ is shown in \cite[Thm.\,4.1]{RuWi12b}.
The lower semi-continuity of $\splineenergy$ can be seen as follows.
Abbreviating $\Gamma_k = \Gamma_{\y_k}(\dot\y_k,\dot\y_k)$ as well as $\Gamma =\Gamma_{\y}(\dot\y,\dot\y)$,
estimate \eqref{eqn:ChristoffelConvergence} turns into
\begin{equation*}
\|\Gamma_k-\Gamma\|_\V
\leq C\left(\|\dot\y_k\|_\Y\|\dot\y_k-\dot\y\|_\Y+\|\dot\y_k-\dot\y\|_\Y\|\dot\y\|_\Y+\|\y_k-\y\|_\Y\|\dot\y\|_\V^2+\|\y_k-\y\|_\Y\|\Gamma_k\|_\Y\right)\,.
\end{equation*}
The uniform convergence $\y_k\to\y$ and $\dot\y_k\to\dot\y$ in $\Y$ as well as the boundedness of $\dot\y$ in $\V$ due to $\y\in C^{1,\frac12}([0,1];\V)$ and the uniform boundedness of $\|\Gamma_k\|_\Y\leq\frac32\frac{\|D_\y g^c\|}{c^\ast}\|\dot\y_k\|_\Y^2$
now imply the strong convergence $\Gamma_k\to\Gamma$ in $L^\infty((0,1);\V)$.
Now we consider the splitting 
\begin{equation*}
\splineenergy[\y_k]=\int_0^1\metric_{\y_k}(\ddot \y_k+\Gamma_k,\ddot\y_k+\Gamma_k)\d t = \int_0^1\metric_{\y}(\ddot \y_k+\Gamma_k,\ddot\y_k+\Gamma_k)\d t +  \int_0^1(\metric_{\y_k}-\metric_{\y})(\ddot \y_k+\Gamma_k,\ddot\y_k+\Gamma_k)\d t\,.
\end{equation*}
The weak convergence of $\ddot \y_k+\Gamma_k$ to $\ddot \y+\Gamma$ in $L^2\equiv L^2((0,1);\V)$ and the fact that 
$g_\y$ is a positive definite quadratic form on $\V$ implies the lower semi-continuity of the first integral on the right-hand side via Mazur's lemma.
The second integral vanishes in the limit due to the strong convergence $\y_k\to\y$ in $C^0([0,1];\Y)$ and the boundedness of $\ddot \y_k+\Gamma_k$ in $L^2$.
From this the requested lower semi-continuity $\liminf_{k\to\infty}\splineenergy[\y_k]\geq\splineenergy[\y]$ follows.

If instead $\y_k\to\y$ strongly in $W^{2,2}$, then the continuity of $\pathenergy$ along this sequence follows immediately from the strong convergence in $C^{1,\frac12}([0,1];\V)$.
Furthermore, following the same argument as above we obtain $\ddot\y_k+\Gamma_k\to\ddot\y+\Gamma$ strongly in $L^2$. Thus
\begin{multline*}
\left|\splineenergy[\y_k]-\splineenergy[\y]\right|
=\left|\int_0^1\metric_{\y_k}((\ddot \y_k+\Gamma_k)+(\ddot\y+\Gamma),(\ddot\y_k+\Gamma_k)-(\ddot\y+\Gamma))\d t\right|\\
\leq(CC^\ast+C^{\ast\ast})\|(\ddot \y_k+\Gamma_k)+(\ddot\y+\Gamma)\|_{L^2}\|(\ddot \y_k+\Gamma_k)-(\ddot\y+\Gamma)\|_{L^2}\to0
\quad\text{as }k\to\infty\,,
\end{multline*}
where we used the bound on the metric $g$ from Definition\,\ref{def:admissibleMetric} and $C$ is the embedding constant of $\V\hookrightarrow\Y$.
\end{proof}

Now we are in the position to prove existence and regularity of spline interpolations.
\begin{theorem}[Existence of spline interpolations]\label{thm:ExistenceContinuous}
For $\sigma>0$ and $(\manifold,g)$ admissible
there exists a spline interpolation $\y$ of \eqref{eq:IC} under natural, Hermite, or periodic boundary conditions in the Sobolev space $W^{2,2}((0,1);\V)$.
Furthermore,
$\y$ is in $C^{2,\frac12}([\delta, 1-\delta];\V)$ for every $\delta \in (0,\tfrac12)$.
\end{theorem}
\begin{proof} 
Let us first collect properties of the involved function spaces.
As a Hilbert space, $\V$ is reflexive.
The same therefore holds true for $W^{2,2}\equiv W^{2,2}((0,1);\V)$.
Furthermore, $W^{2,2}$ embeds continuously into $L^p\equiv L^p((0,1);\V)$ and $W^{1,p}\equiv W^{1,p}((0,1);\V)$ for any $p\in[1,\infty)$.
Likewise, $W^{2,2}\hookrightarrow C^{1,\frac12}([0,1];\V)$ continuously, which proves differentiability of the interpolation, if it exists.
In addition, this embedding shows that pointwise evaluation of $\y\in W^{2,2}$ and its derivative at any $t\in[0,1]$ is a bounded linear functional on $W^{2,2}$.

Next we show that the regularized spline energy $\splineenergy^\sigma$ with condition \eqref{eq:IC} is coercive in $W^{2,2}$.
Indeed, let $\splineenergy[\y]+\sigma\pathenergy[\y]<M$ for some $M\in\R$, then $\|\dot\y\|_{L^2}^2\leq\frac{\pathenergy[\y]}{c^\ast}\leq\frac{M}{\sigma c^\ast}$,
and by \eqref{eq:IC} and Poincar\'e's inequality it follows that $\|\y\|_{W^{1,2}}^2\leq C(M)$.
Furthermore, using the reverse triangle inequality and Young's inequality we have (abbreviating $\Gamma=\Gamma_\y(\dot\y,\dot\y)$)
\begin{multline*}
\tfrac M{c^\ast}
\geq\tfrac1{c^\ast}\splineenergy[\y]
\geq\int_0^1\|\cov\dot\y\|_\V^2\d t
\geq\int_0^1\|\ddot\y\|_\V^2-2\|\ddot\y\|_\V\|\Gamma\|_\V+\|\Gamma\|_\V^2\d t\\
\geq\int_0^1\tfrac12\|\ddot\y\|_\V^2-\|\Gamma\|_\V^2\d t
\geq\tfrac12\|\ddot\y\|_{L^2}^2-\left(\tfrac32\tfrac{\|D_\y g^c\|}{c^\ast}\right)^2\|\dot\y\|_{L^4}^4\,,
\end{multline*}
where in the last inequality we used the estimate $\|\Gamma_\y(\dot\y,\dot\y)\|_\V\leq\frac32\frac{\|D_\y g^c\|}{c^\ast}\|\dot\y\|_\Y^2$ from the previous proof.
From this we deduce that $\|\ddot\y\|_{L^2}$ is bounded by a constant depending solely on $M$.
This can be shown via an argument by contradiction.
If $\|\ddot\y\|_{L^2}$ cannot be bounded in terms of $M$,
there must be a positive constant $K(M)$ and a sequence $(\y_k)_{k=1,\ldots}$ with 
$\|\y_k\|_{W^{1,2}}^2\leq C(M)$, but $\|\dot\y_k\|_{L^4}^4\geq K(M)\|\ddot\y_k\|_{L^2}^2\to\infty$.
Let $f_k:(0,1)\to[0,\infty)$ be the decreasing rearrangement of $\|\dot\y_k\|_\V$,
then by the properties of the decreasing rearrangement (in particular the Szeg\"o inequality)
\begin{equation*}
\|f_k\|_{L^4}^4=\|\dot\y_k\|_{L^4}^4\to\infty,\,
\|\dot f_k\|_{L^2}^2\leq\|\ddot\y_k\|_{L^2}^2,\,
\|f_k\|_{L^2}^2\leq\|\y_k\|_{W^{1,2}}^2\leq C(M)\,.
\end{equation*}
Due to $\|f_k\|_{L^4}^4\leq\|f_k\|_{L^2}^2\|f_k\|_{L^\infty}^2$ we see that $M_k:=f_k(0) = \|f_k\|_{L^\infty}\to\infty$.
Now $|f_k(t)|^2 t \leq \|f_k\|_{L^2}^2\leq C(M)$ implies $f_k(t)\leq\min(M_k,\sqrt{C(M)}/\sqrt t)$.
Next we estimate $\|\dot f_k\|_{L^2}^2$ from below. Let $g_k$ be the unique affine function with $g_k(0) = M_k$ that touches the function
$\sqrt{C(M)}/\sqrt t$ in some $t=t^k>0$. It is straightforward to check $t^k = \frac94 M_k^{-2} C(M)$ as well as $g_k'=-\frac{4 M_k^3}{27 C(M)}$ and 
$\|\dot g_k\|^2_{L^2((0,t^k))} = \frac{4 M_k^4}{81 \,C(M)}$. Then we obtain by Jensen's inequality
$$
\|\dot g_k\|^2_{L^2((0,t^k))} \leq  \frac{|f_k(t^k)-f_k(0)|^2}{t^k} = \frac1{t^k} \left( \int_0^{t^k} \dot f_k(s) \d s \right)^2 \leq \int_0^{t^k} |\dot f_k(s)|^2 \d s\,.
$$
Thus, $\frac{4 M_k^4}{81 \,C(M)}$ is a lower bound for $\|\dot f_k\|_{L^2}^2$ and we achieve the following chain of inequalities
$$
\frac{4 M_k^4 K(M)}{81 \,C(M)} \leq K(M) \|\dot f_k\|_{L^2}^2 \leq K(M) \|\ddot\y_k\|_{L^2}^2 \leq  \|\dot\y_k\|_{L^4}^4 = \|f_k\|_{L^4}^4 
\leq \|f_k\|_{L^2}^2\|f_k\|_{L^\infty}^2 \leq C(M) M_k^2
$$
which leads to a contradiction for large $k$. Hence,  $\|\ddot\y_k\|_{L^2}$ is bounded by a constant depending solely on $M$, which implies the coercivity of the energy
$\splineenergy+\sigma\pathenergy$.

As a consequence of the coercivity, a minimizing sequence $(\y_k)_{k=1,\ldots}$ is uniformly bounded in $W^{2,2}$, and by the reflexivity of the space $W^{2,2}$ we obtain a weakly converging subsequence,
again denoted by $(\y_k)_{k=1,\ldots}$, which converges to some $\y$ in $W^{2,2}$.
Finally, the weak lower semi-continuity of $\splineenergy^\sigma$ by Lemma\,\ref{thm:energyLSC} implies
$$\splineenergy^\sigma[\y]\leq\liminf_{k\to \infty}\splineenergy^\sigma[\y_k]\,.$$
%
Furthermore, the weak limit $\y$ satisfies \eqref{eq:IC} and the boundary conditions since they are continuous with respect to weak convergence in $W^{2,2}((0,1);\V)$.
Thus, $\y$ is the sought spline interpolation.

\medskip

\newcommand{\dep}{\partial_\epsilon^+}
\newcommand{\dem}{\partial_\epsilon^-}
\newcommand{\norm}[1]{\|{#1}\|}
\newcommand{\znorm}[1]{\|{#1}\|_{L^2}}
\newcommand{\gy}{g_\y}

Next we consider the regularity of the minimizer $\y$.
To prove higher regularity we apply Friedrichs' regularity theory to the Euler--Lagrange equation \eqref{eq:ELSpline}. 
Note that $\vartheta = \dep (\eta^4 \dem \y)$ is an admissible test function for any  smooth scalar function $\eta$ 
with compact support on $(0,1)$ for sufficiently small $\epsilon$. Here,  $\partial_\epsilon^\pm$ denotes the forward/backward difference operator defined by
$\partial_\epsilon^\pm \phi(t) = \pm\tfrac1\epsilon (\phi(t \pm  \epsilon) - \phi(t))$.  
Using $\dot \vartheta = \dep (\eta^4 \dem \dot \y + 4\dot \eta \eta^3 \dem \y)$ and 
$\ddot \vartheta = \dep (\eta^4 \dem \ddot \y + 8 \dot \eta \eta^3  \dem \dot \y + 4(\ddot \eta \eta + 3 \dot \eta^2) \eta^2 \dem \y)$ and 
applying a discrete integration by parts formula 
$$
\int_0^1 (\dep \beta) \gamma \d t
= \frac1\epsilon \int_0^1  \beta(t+\epsilon)\gamma(t) - \beta(t)\gamma(t) \d t
= \frac1\epsilon \int_0^1  \beta(t)\gamma(t-\epsilon) - \beta(t)\gamma(t) \d t
= - \int_0^1  \beta (\dem \gamma) \d t\,,
$$
which holds for  compactly supported $\beta$ and $\epsilon$ sufficiently small, and the product rule for difference quotions 
$\partial_\epsilon^\pm (\beta \gamma)=\beta (\partial_\epsilon^\pm \gamma)+ (\partial_\epsilon^\pm \beta) \gamma(\cdot \pm \epsilon)$, one obtains 
\begin{alignat}{2}\label{eq:leadingterm}
\int_0^1 \!\!\!g_\y(\ddot \vartheta, \ddot \y\!+\! \Gamma_\y(\dot \y,\dot \y)) \d t\!
&= \!-\!\!\int_0^1\!\!\! &&\eta^4 \left[g_{\y(\cdot-\epsilon)}(\dem \ddot \y, \dem (\ddot \y + \Gamma_\y(\dot \y,\dot \y))) + (\dem \gy)(\dem \ddot \y, (\ddot \y 
+ \Gamma_\y(\dot \y,\dot \y)))\right] \\
& && + (8\dot \eta \eta) \eta^2 \left[g_{\y(\cdot-\epsilon)}(\dem \dot \y, \dem (\ddot \y + \Gamma_\y(\dot \y,\dot \y))) +(\dem \gy)(\dem \dot \y, (\ddot \y + \Gamma_\y(\dot \y,\dot \y)))\right] \nonumber\\ 
& && + (4 \ddot \eta \eta \!+\! 12 \dot \eta^2) \eta^2 \big[g_{\y(\cdot-\epsilon)}(\dem \y, \dem (\ddot \y \!+\! \Gamma_\y(\dot \y,\dot \y)))\!+\!(\dem \gy)(\dem  \y, (\ddot \y \!+\! \Gamma_\y(\dot \y,\dot \y)))  \big] \d t\,. \nonumber
\end{alignat}
Likewise, we obtain for the last term in  \eqref{eq:ELSpline} the decomposition
\begin{align*}
\int_0^1 h(\y,\dot \y, \ddot \y, \vartheta, \dot \vartheta) \d t  
= & \int_0^1 2 g_\y(\ddot \y,2\Gamma_\y(\dep (\eta^4 \dem \dot \y\!+\!4 \eta^3 \dot \eta \dem \y),\dot \y)+(D_y \Gamma_\y)(\dep (\eta^4 \dem \y))(\dot \y,\dot \y))\nonumber\\
& \quad\!+\!(D_\y g_\y)(\dep (\eta^4 \dem \y)) (\ddot \y, \ddot \y\!+\!2 \Gamma_\y(\dot \y,\dot \y))\!+\!2g_\y(\Gamma_\y(\dot \y,\dot \y),2\Gamma_\y(\dep (\eta^4 \dem \dot \y\!+\!4 \eta^3 \dot \eta \dem \dot \y),\dot \y)\nonumber\\
&\quad +(D_y \Gamma_\y)(\dep (\eta^4 \dem \y))(\dot \y,\dot \y))\!+\!(D_\y g_\y)(\dep (\eta^4 \dem \y))( \Gamma_\y(\dot \y,\dot \y), \Gamma_\y(\dot \y,\dot \y))\d t\,.
\end{align*}
Now, we estimate the different terms separately using the product rule for difference quotions,  the regularity estimate $\norm{\dot y}_{C^{1,\frac12}([0,1];\V)} \leq \hat C$,
and the observation that $\Gamma_\y$ is a bilinear form on $\V\times \V$ which is uniformly bounded  and continuously differentiable with respect to $\y\in\V$.
In addition for the fourth term we perform another discrete integration by parts. Altogether, we obtain
\begin{align*}
 \int_0^1 2 g_\y(\ddot \y,2\Gamma_\y(\dep (\eta^4 \dem \dot \y),\dot \y)) \d t   
&\leq C\, (\norm{{\eta(\cdot + \epsilon)\dem\dot{\y}}}_{L^2} + \norm{\eta^2 \dep \dem \dot \y}_{L^2})\norm{\ddot \y}_{L^2}\,, \\
 \int_0^1 8 g_\y(\ddot \y,2\Gamma_\y(\dep (\eta^3 \dot \eta \dem  \y),\dot \y))\d t   
&\leq C\, (\norm{{\eta(\cdot + \epsilon)\dem{\y}}}_{L^2} + \norm{\eta^2 \dep \dem \y}_{L^2})\norm{\ddot \y}_{L^2}  \,,\\
 \int_0^1 2 g_\y(\ddot \y,(D_y \Gamma_\y)(\dep (\eta^4 \dem \y))(\dot \y,\dot \y)) \d t   
 &\leq C (\norm{{\eta(\cdot +\epsilon)\dem {\y}}}_{L^2} + \norm{\eta^2 \dep \dem  \y}_{L^2})\norm{\ddot \y}_{L^2}  \,,\\
 \int_0^1 \!(D_\y g_\y)(\dep (\eta^4 \dem \y)) (\ddot \y, \ddot \y)  \d t  
 &= - \int_0^1 \!\! \dem (D_\y g_\y)(\eta^4 \dem \y)(\ddot \y,\ddot \y) 
+  (D_\y g_{\y(\cdot-\epsilon)})(\eta^2 \dem \y)(\eta^2 \dem  \ddot \y, \ddot \y) \\[-1ex]
& \; \quad \qquad +  (D_\y g_{\y(\cdot-\epsilon)})(\eta^2 \dem \y)(\eta^2 \dem  \ddot \y, \ddot \y(\cdot-\epsilon))\d t \\
&\leq C \norm{{\eta  \dem \y}}_{L^\infty} \norm{\ddot \y}^2_{L^2} 
+ C \norm{\eta \dem \y}_{L^\infty}  \norm{\eta^2 \dem \ddot \y}_{L^2}   \norm{\ddot \y}_{L^2} \,,\\
 \int_0^1  2 (D_\y g_\y)(\dep (\eta^4 \dem \y) )(\ddot \y, \Gamma_\y(\dot \y,\dot \y)) \d t  
 & \leq C (\norm{{\eta(\cdot +\epsilon)\dem \y}}_{L^2} + \norm{\eta^2 \dep \dem \y}_{L^2})  \norm{\ddot \y}_{L^2}\,,\\
 \int_0^1  4g_\y(\Gamma_\y(\dot \y,\dot \y),\Gamma_\y(\dep (\eta^4 \dem \dot \y),\dot \y))\d t  
 & \leq C (\norm{{\eta(\cdot +\epsilon)\dem\dot{\y}}}_{L^2} + \norm{\dep (\eta^2  \dem \dot \y)}_{L^2})  \,,\\
 \int_0^1 16g_\y(\Gamma_\y(\dot \y,\dot \y),\Gamma_\y(\dep (\eta^3 \dot \eta \dem \dot \y),\dot \y)) \d t  
 &\leq C (\norm{{\eta(\cdot +\epsilon)\dem\dot{\y}}}_{L^2} + \norm{\dep (\eta^2  \dem \dot \y)}_{L^2})  \,,\\
 \int_0^1 2g_\y(\Gamma_\y(\dot \y,\dot \y),(D_y \Gamma_\y)(\dep (\eta^4 \dem \y))(\dot \y,\dot \y)) \d t  
 &\leq C (\norm{{\eta(\cdot +\epsilon)\dem {\y}}}_{L^2} + \norm{\eta^2 \dep \dem  \y}_{L^2})  \,,\\
 \int_0^1  (D_\y g_\y)(\dep (\eta^4 \dem \y))( \Gamma_\y(\dot \y,\dot \y), \Gamma_\y(\dot \y,\dot \y)) \d t 
 &\leq C (\norm{{\eta(\cdot +\epsilon)\dem {\y}}}_{L^2} + \norm{\eta^2 \dep \dem  \y}_{L^2})  
\end{align*}
for a generic constant $C$ depending on $\eta$, $g$, and $\hat C$.
Next, we apply Jensen's inequality and Fubini's theorem and obtain
\begin{equation*}
\znorm{\beta \partial_\epsilon^\pm \gamma}^2
=\int_0^1\beta^2(t)\left\|\frac1\epsilon\int_t^{t+\epsilon}\dot\gamma\d s\right\|_\V^2\d t
\leq\|\beta\|_{L^\infty}^2\int_0^{1-\epsilon}\frac1\epsilon\int_t^{t+\epsilon}\|\dot\gamma\|_\V^2\d s\d t
\leq\|\beta\|_{L^\infty}^2\znorm{\dot\gamma}^2
\end{equation*}
for any compactly supported $\beta$, weakly differentiable $\gamma$, and small enough $\epsilon$.
We use  this to estimate
\begin{align*}
\znorm{{\eta(\cdot +\epsilon)\dem{\y}}}  +  \znorm{\eta \dem \y} \leq C \znorm{\dot \y} ,\;
\znorm{{\eta(\cdot +\epsilon)\dem\dot{\y}}}  \leq C \znorm{\ddot \y}, \;
\znorm{\eta^2 \dep \dem \y} \leq C \znorm{\ddot \y}\,.
\end{align*}
To estimate the term $\znorm{\dep(\eta^2  \dem \dot \y)}$
we apply the product rule and proceed as follows,
\begin{align*}
\znorm{\dep(\eta^2  \dem \dot \y)} &\leq  \znorm{\tfrac{\mathrm{d}}{\mathrm{d}t} (\eta^2 \dem \dot \y)} 
\leq  \znorm{2\eta \dot \eta \dem \dot \y} + \znorm{\eta^2\dem \ddot \y}
\leq  \znorm{\eta^2\dem \ddot \y}+C \znorm{\ddot \y}\,.
\end{align*} 
Thus, we are leads to 
\begin{align*}
\int_0^1 h(\y,\dot \y, \ddot \y, \vartheta, \dot \vartheta) \d t \leq& \; 
C \left(\znorm{\eta^2 \dem \ddot \y} + \znorm{\ddot \y} +  1\right) \cdot  \left( \znorm{\ddot \y} +1\right)  \,.
\end{align*}
Furthermore, the first term of \eqref{eq:ELSpline}, representing the variation \eqref{eq:ELpathenergy} of the path energy, can be estimated as follows,
\begin{equation*}
\partial_\y \pathenergy[\y](\vartheta) \leq 
C \left( \znorm{ \ddot \y}+\norm{ \dot \y}_{L^\infty}^2\right)\znorm{\vartheta}
\leq C \left(\znorm{ \ddot \y}^2 +  1\right)  \,.
\end{equation*}
Now, using the same arguments we estimate  \eqref{eq:leadingterm} from below and obtain
\begin{align*}
-\int_0^1 g_\y(\ddot \vartheta, \ddot \y+ \Gamma(\dot \y,\dot \y)) \d t   \geq & \; c^\ast \znorm{\eta^2 \dem \ddot \y}^2  
 - C \left( \znorm{ \ddot \y} + 1\right)  \cdot \left(\znorm{\eta^2  \dem \ddot \y} +\znorm{\ddot \y} +  1\right) \,.
\end{align*}
In summary, using the boundedness of $\znorm{\ddot \y}$, \eqref{eq:ELSpline} led to 
\begin{equation*}
c^\ast \znorm{\eta^2 \dem \ddot \y}^2 \leq C(\znorm{\eta^2 \dem \ddot \y}+1)
\end{equation*}
for a sufficiently large $C>0$, from which we obtain the uniform boundedness of $\znorm{\eta^2 \dem \ddot \y}$ independent of $\epsilon$ via Young's inequality.
Hence, for $\epsilon \to 0$ there exists a weakly converging subsequence of $\dem \ddot \y$ in $L^2((\delta, 1-\delta);\V)$ for any fixed $\delta >0$, 
whose limit is the weak derivative $\dddot\y \in L^2((\delta, 1-\delta);\V)$. Using
the continuous embedding $W^{3,2}((\delta, 1-\delta);\V)$ in $C^{2,\frac12}([\delta, 1-\delta];\V)$ finishes the proof.
\end{proof}


\section{Time-discrete geodesics and splines}\label{sec:splinesdiscrete}
As sketched in the introduction the time discretization is based on a functional $\energy$ 
which is expected to approximate the squared Riemannian distance.
In this section we will investigate the well-posedness of discrete splines. 
To this end, let us at first state the assumptions on the functional $\energy$.

\begin{definition}[Admissible $\energy$]\label{def:admissibleEnergy}
We say that $\energy: \V \times \V \to [0,\infty]$ is \emph{admissible} if
$$
\energy[\y,\tilde \y] = \energy^c[\y,\tilde \y] + \Qenergy(\y-\tilde \y,\y-\tilde \y)
$$
for the quadratic form $\Qenergy$ from Definition\,\ref{def:admissibleMetric} and some $\energy^c: \V \times \V \to [0,\infty]$ such that the following conditions hold.
\begin{enumerate}[label=(\roman*)]
\item\label{enm:discreteApprox} There exist $\varepsilon,C>0$ such that $|\energy [\y,\widetilde{\y}] - \dist^2 (\y,\widetilde{\y})| \leq C \dist^3 (\y,\widetilde{\y})$ for all $\y,\widetilde{\y}\in\manifold$ with $\dist(\y,\widetilde{\y})\leq\varepsilon$.
\item $\energy^c$ is four times continuously differentiable on $\V\times\V$.
\item $\energy^c$ is continuous under weak convergence in $\V\times\V$.
\item\label{enm:discreteCoercivity} $\energy$ is coercive in the sense $\energy[\y,\widetilde{\y}] \geq \gamma(\|\y-\widetilde{\y}\|_\V)$
for a strictly increasing, continuous function $\gamma$ with $\gamma(0)=0$ and $\lim_{d\to\infty}\gamma(d)=\infty$.
\end{enumerate}
\end{definition}

Using the approximation $\energy$ to the squared Riemannian distance, we can define discrete analogs of $\pathenergy$ and $\splineenergy$
(cf.\ the motivation in the introduction).

\begin{definition}[Discrete path and spline energy] \label{def:discreteEnergies}
The \emph{discrete path energy} $\Pathenergy^K$ (cf.\ \cite[Def.\,2.1]{RuWi12b}) and the \emph{discrete spline energy} $\Splineenergy^K$ are given as
\begin{align}
 \Pathenergy^K[\y_0, \ldots, \y_K] &= K \, \sum_{k=1}^K \energy[\y_{k-1}, \y_k]\, ,\label{eq:Pathenergy}\\
   \Splineenergy^K[\y_0, \ldots, \y_K] &= 4K^3\, \sum_{k=1}^{\widehat K} \energy[\y_k, \widetilde{\y}_k]\, ,\label{eq:Splineenergy}\\
  &\text{with }\widetilde{\y}_k \in \argmin_{\y\in\manifold} \Big( \energy[\y_{k-1},\y] + \energy[\y,\y_{k+1}] \Big)=\argmin_{\y\in\manifold}\Pathenergy^2[\y_{k-1},\y,\y_{k+1}] \, ,\label{eq:constraint}
\end{align}
both defined for discrete paths $(\y_0, \ldots, \y_K)\in\manifold^{K+1}$.
Above, $\widehat K$ denotes the number of constraints \eqref{eq:constraint}.
For natural and Hermite boundary conditions we will use $\widehat K=K-1$,
while for periodic boundary conditions we identify $\y_{K+1}\equiv\y_1$ and have $\widehat K=K$.
The \emph{discrete regularized spline energy} is defined as 
\begin{equation}\label{eq:mixedEnergyDiscrete}
 \Splineenergy^{\sigma,K}[\y_0, \ldots, \y_K] = \Splineenergy^K[\y_0, \ldots, \y_K]+\sigma\,\Pathenergy^K[\y_0, \ldots, \y_K]
\end{equation}
for some $\sigma\geq0$.
\end{definition}

\begin{remark}[Geodesic midpoint]
The points $\widetilde{\y}_k$ are intended to approximate the geodesic midpoint between $\y_{k-1}$ and $\y_{k+1}$ so that $\Splineenergy^K$ essentially penalizes the deviation of $(\y_0,\ldots,\y_K)$ from a (discretized) geodesic.
On some simple manifolds the geodesic midpoint might be calculated explicitly; in that case one may take $\widetilde{\y}_k$ as the true geodesic midpoint.
\end{remark}


\begin{remark}[Motivation based on the discrete covariant derivative]\label{remark:motivation}
\revision{Here we show how to interpret the discrete spline energy as a discretization of the time-continuous one.
Let us assume that $(\y(t))_{t\in [0,1]}$ is a three times continuously differentiable curve in $\manifold$.
In \cite[Thm.\,5.13]{RuWi12b} is was shown that the covariant derivative $\cov \dot \y$ (along the curve $\y$) can be approximated based on a concept of discrete parallel transport.
In fact, for $p \in \manifold$ and $\xi \in T_p\manifold$ a discrete covariant derivative $\nabla_\xi(\eta_0,\eta_1)$ was defined in \cite[Def.\,2.6]{RuWi12b} for $\eta_0\in T_p\manifold$ and $\eta_1\in T_{p+\xi}\manifold$
as an approximation of the continuous covariant derivative $\nabla_\xi\eta$ at $p\in\manifold$ for a vector field $\eta$ interpolating $\eta_0$ and $\eta_1$. 
Here we use the implicit notation $(\nabla_\xi\eta)(p) = \cov\eta(0)$, where the covariant derivative is along an arbitrary curve $\gamma:[-\epsilon,\epsilon]\to\manifold$ with $\gamma(0) = p$ and $\dot\gamma(0) = \xi(p)$.
In particular, \cite[Thm.\,5.13]{RuWi12b} establishes the consistency
\begin{align} \label{eq:tauquadrat}
\cov \dot \y(t_{k-1}) = \frac1{\tau^2} \nabla_{\tau v_k} (\tau \tilde v_k, \tau \tilde v_{k+1}) + O(\tau)\,, 
\end{align}
where $\tilde v_k = \dot \y(k\tau)$ is the curve velocity and $v_k=\frac{\y_{k}-\y_{k-1}}{\tau }$ its discrete approximation.}
If we replace $\tilde v_k$ by $v_k$, this approximation result still holds following the arguments in the proof of  \cite[Thm.\,5.11 \& 5.13]{RuWi12b} and the interpolation error estimate 
$\dot \y(k\tau)= \frac{\y_{k}-\y_{k-1}}{\tau } + O(\tau^2)$. Furthermore, using the definition \cite[Def. 2.9]{RuWi12b} of the discrete parallel transport it turns out that the discrete connection can be expressed as
$$
\nabla_{\tau v_k} (\tau  v_k, \tau  v_{k+1}) = \widehat{\y}_k-\y_k\,,
$$
where $(\y_k,\widetilde{\y}_k, \widehat{\y}_k)$ is a three point discrete geodesic with $\widetilde{\y}_k$ (as introduced in \eqref{eq:constraint}) 
the midpoint of the three point discrete geodesic $(\y_{k-1}, \widetilde{\y}_k, \y_{k+1})$. 
Moreover, for uniformly bounded $\|\cov \dot \y\|_\V$ we deduce from 
\eqref{eq:tauquadrat} that $\widehat{\y}_k-\y_k = O(\tau^2)$ and thus $\widetilde{\y}_k-\y_k = O(\tau^2)$. 
Then the discrete equidistribution result for points along discrete geodesics and in particular \cite[Lemma 5.8]{RuWi12b} implies 
$\widehat{\y}_k = \y_k + 2 (\widetilde{\y}_k- \y_k) + O(\|\widetilde{\y}_k-\y_k\|_\V^{3/2}) = \y_k + 2 (\widetilde{\y}_k- \y_k) + O(\tau^3)$ and thus 
$$
\nabla_{\tau v_k} (\tau  v_k, \tau  v_{k+1}) = 2(\widetilde{\y}_k-\y_k)+O(\tau^3)
\qquad\text{and}\qquad
\cov \dot \y (t_{k-1}) = \frac2{\tau^2}(\widetilde{\y}_k-\y_k)+O(\tau)\,.
$$
Next, the metric $g_\y$ can be approximated using the energy functional $\energy$ as $g_\y(v,v) = \energy[\y,\y+v] + O(\|v\|_\V^3)$ for small enough $v$.
Using a standard rectangle quadrature rule we thus obtain
\begin{multline*}
\splineenergy[\y]
=\int_0^1 g_{\y(t)}(\cov \dot \y(t),\cov \dot \y(t)) \d t
= \tau \sum_{k=1}^{K-1} g_{\y(t_{k-1})}\left( \cov \dot \y(t_{k-1}),\cov \dot \y(t_{k-1})\right) + O(\tau)\\
= \frac4{\tau^3} \sum_{k=1}^{K-1} g_{\y(t_{k-1})}\left( \widetilde\y_k-\y_k,\widetilde\y_k-\y_k\right) + O(\tau)
= \frac4{\tau^3} \sum_{k=1}^{K-1} \energy[\y_k,\widetilde\y_k] + O(\tau)
=  \Splineenergy^K[\y_0, \ldots, \y_K]   + O(\tau)\,.
\end{multline*}
This establishes the consistency between the discrete and continuous spline energy for a regularly sampled smooth curve on the manifold $\manifold$.
\end{remark}

\medskip

For simplicity we shall assume that the fixed interpolation times $t_i$ are multiples of $\tau=\tfrac1K$ so that the interpolation constraint turns into
\begin{equation}\label{eq:ICdiscrete}
\y_{K t_i} =
 \bar \y_i\, ,\quad i=1,\ldots, I\,.
\end{equation}
In other words, we shall only allow such $K$ that $Kt_i$ is an integer for $i=1,\ldots,I$ (alternatively, one could consider discrete curves with non-equidistant spacing in time; all definitions and results could easily be modified to allow for that case).
The counterparts of the boundary conditions, of which one has to be imposed in addition, are as follows,
\begin{align}
 &\text{natural b.\,c.,}&& \text{no additional condition}\label{eqn:discreteNaturalBC}\\
 &\text{Hermite b.\,c.,}&& K(\y_1-\y_0) = v_0, \quad K(\y_K-\y_{K-1}) = v_1 \quad \text{ for given }v_0,v_1\in\V\,,\label{eqn:discreteHermiteBC}\\
 &\text{periodic b.\,c.,}&& \y_0 = \y_K\,.\label{eqn:discretePeriodicBC}
\end{align}
The terms $K (\y_{1}- \y_{0})$ and $K (\y_{K}- \y_{K-1})$ in the Hermite boundary condition play the role of $\dot \y(0)$ and $\dot \y(1)$, respectively, in the continuous case.

Now, we are in the position to define a discrete spline interpolation.
\begin{definition}[Discrete geodesic and spline interpolation]\label{def:discreteSplineInterpolation}
For given data points $t_i\in[0,1]$ and $\bar\y_i\in\manifold$, $i=1,\ldots,I$, with $K t_i \in \N_0$ for some $K\in \N$ 
a \emph{discrete piecewise geodesic interpolation} $\y^K=(\y_0,\ldots, \y_K)$ is defined as a minimizer of the discrete path energy under the interpolation constraints \eqref{eq:ICdiscrete} and $t_1=0$, $t_I=1$, 
\begin{equation*}
\y^K\in\argmin \left\{ \Pathenergy^K[\widetilde{\y}^K] \,|\, \widetilde{\y}^K \in \manifold^{K+1} \text{ with \eqref{eq:ICdiscrete}} \right\}\,,
\end{equation*}
while we define a \emph{discrete spline interpolation} $\y^K$ as a minimizer of the discrete (regularized) spline energy under \eqref{eq:ICdiscrete} and one of the above boundary conditions,
\begin{equation*}
\y^K\in\argmin \left\{\Splineenergy^{\sigma,K} \,|\, \widetilde{\y}^K \in \manifold^{K+1} \text{ with \eqref{eq:ICdiscrete} and one of \eqref{eqn:discreteNaturalBC}-\eqref{eqn:discretePeriodicBC}} \right\}\,.
\end{equation*}
\end{definition}

\begin{remark}[Well-posedness of $\Splineenergy^K$]
Note that for infinite-dimensional $\manifold$ we face a similar problem in the time discrete case as in the time continuous case.
Indeed, without the structural assumption on the functional $\energy$ one cannot expect the discrete regularized spline energy to possess minimizers in general.
In fact, if one considers a minimizing sequence $\big((\y^j_0,\ldots, \y^j_K)\big)_{j=1,\ldots}$ in $\manifold^{K+1}$,
the coercivity of $\energy$ only leads to weak convergence in $\V^{K+1}$ for a subsequence.
However, weak convergence of $\y_{k-1}^j$ and $\y_{k+1}^j$ as $j\to\infty$ does not necessarily imply weak convergence of their geodesic midpoint for general functionals
$\energy$ obeying only the hypothesis  of \cite[H2, H4]{RuWi12b}.
Thus,  $\energy[\y_k^j,\widetilde{\y}_k^j]$ may not be lower semi-continous as $j\to\infty$, preventing the existence of a minimizer.
Indeed, $4K^4 \energy[\y_k^j,\widetilde{\y}_k^j]$ is the discrete counterpart of $g_\y(\cov{\dot \y},\cov{\dot \y})$,
and thus the lack of weak continuity of the former in the time discrete context is linked to the lack of weak continuity of the latter in the time continuous context.
\end{remark}
\begin{remark}[Uniqueness of geodesic midpoint]
Uniqueness of $\widetilde\y_k$ (at least for  $d_K=\max_{k\in\{1,\ldots,K\}}\|\y_k-\y_{k-1}\|_\V$ small enough) would require additional properties of $\energy$
such as local convexity or smoothness as in \cite[Thm.\,4.6]{RuWi12b}.
\end{remark}
Before we state an existence result in analogy to Theorem~\ref{thm:ExistenceContinuous} we prove the following 
technical lemma, which will enable us to show that along a minimizing sequence for the discrete regularized spline energy the constraint 
\eqref{eq:constraint} stays fulfilled in the limit.
Note that this lemma in essence plays the same role as Lemma\,\ref{thm:weakcontGamma} in the continuous case.
Just like there, it is crucial that the highest order part of $\energy$ is a spatially constant quadratic form $\Qenergy$.
\begin{lemma}[Interpolation energy convergence] \label{lemma:gamma}
Let $\energy$ be admissible and assume that two sequences $(\y_+^j)_{j=1,\ldots}$ and $(\y_-^j)_{j=1,\ldots}$ converge weakly in $\V$ to some $\y_+$ and $\y_-$, respectively.
Then the energies
$\Pathenergy_\pm^j[\y] = \Pathenergy^2[\y_-^j,\y,\y_+^j]-\Delta^j$
$\Gamma$-converge 
with respect to the weak topology in $\V$ to $\Pathenergy_\pm[\y] = \Pathenergy^2[\y_-,\y,\y_+]-\Delta$ with
\begin{align*}
\Delta^j
&=\Qenergy(\y^j_\pm- \y^j_-,\y^j_\pm - \y^j_-) + \Qenergy(\y^j_\pm -\y^j_+,\y^j_\pm-  \y^j_+)
&\text{for }\quad \y^j_\pm=\tfrac{\y^j_++\y^j_-}2,\\
\Delta
&=\Qenergy(\y_\pm- \y_-,\y_\pm - \y_-) + \Qenergy(\y_\pm -\y_+,\y_\pm-  \y_+)
&\text{for }\quad \y_\pm=\tfrac{\y_++\y_-}2.
\end{align*}
\end{lemma}
\begin{proof}
First we investigate the $\liminf$ property. To this end let $(\y^j)_{j=1,\ldots}$ be a weakly converging sequence in $\V$ with weak limit $\y$. 
We reformulate the quadratic terms in the energies. We observe that
\begin{equation*}\label{eq:Qrewrite}
\Qenergy(\y- \y_-,\y-  \y_-) +  \Qenergy(\y -\y_+,\y-  \y_+)
= 2\Qenergy(\y- \y_\pm,\y-  \y_\pm)
 + \Delta\,,
\end{equation*}
where $\y_\pm$ is the minimizer 
of the left-hand side. Note that the right-hand side is just the Taylor expansion of the left-hand side at the minimizer. An analogous decomposition is obtained 
replacing all $\y_-$, $\y_+$, $\y_\pm$, and $\Delta$ by  $\y^j_-$,  $\y^j_+$, $\y^j_\pm$, and $\Delta^j$, respectively, so that summarizing we can rewrite
\begin{align*}
\Pathenergy^j_\pm[\widetilde\y]
&=\energy^c[\y^j_-,\widetilde\y] + \energy^c[\y^j_+,\widetilde\y] + 2\Qenergy(\widetilde\y- \y^j_\pm,\widetilde\y- \y^j_\pm)\,,\\
\Pathenergy_\pm[\widetilde\y]
&=\energy^c[\y_-,\widetilde\y] + \energy^c[\y_+,\widetilde\y] + 2\Qenergy(\widetilde\y- \y_\pm,\widetilde\y- \y_\pm)\,.
\end{align*}
Obviously, $\y_\pm^j$ weakly converges to $\y_\pm$ in $\V$.
From the weak lower semi-continuity of
the functional $(\y,\widetilde{\y}) \mapsto 2\Qenergy(\y- \widetilde{\y},\y-  \widetilde{\y})$ and the weak continuity of $\energy^c$ we deduce the desired $\liminf$ property
\begin{align*}
 \liminf_{j\to \infty} \Pathenergy^j_\pm[\y^j] 
 &=  \liminf_{j\to \infty}
 \big( \energy^c[\y^j_-, \y^j] + \energy^c[\y^j_+,\y^j] + 2\Qenergy(\y^j- \y^j_\pm, \y^j- \y^j_\pm)\big)\\ 
 &\geq
 \energy^c[\y_-, \y] + \energy^c[\y_+,\y] + 2\Qenergy(\y- \y_\pm, \y- \y_\pm) \\
&= \Pathenergy_\pm[\y]\,.
\end{align*}

To prove the $\limsup$ inequality we consider any $\y \in \V$ and define the recovery sequence
$$
\y^j = \y + \y^j_\pm - \y_\pm
$$
for $j=1, \ldots$. Obviously, $\y^j$ weakly converges to $\y$ in $\V$ so that we obtain 
\begin{align*}
\limsup_{j\to\infty} \Pathenergy^j_\pm[\y^j] 
 &=  \limsup_{j\to \infty}
 \big( \energy^c[\y^j_-, \y^j] + \energy^c[\y^j_+,\y^j] + 2\Qenergy(\y^j- \y^j_\pm, \y^j- \y^j_\pm)\big)\\ 
 &=
 \energy^c[\y_-, \y] + \energy^c[\y_+,\y] + 2\Qenergy(\y- \y_\pm, \y- \y_\pm) \\
&= \Pathenergy_\pm[\y]\,.
\qedhere
\end{align*}
\end{proof}
\begin{lemma}[Constraint convergence] \label{thm:gamma}
Let $\energy$ be admissible and assume that two sequences $(\y_+^j)_{j=1,\ldots}$ and $(\y_-^j)_{j=1,\ldots}$ converge weakly in $\V$ to some $\y_+$ and $\y_-$, respectively.
Let $\widetilde\y^j\in\argmin_\y\Pathenergy^2[\y_-^j,\y,\y_+^j]$. If $\widetilde\y^j\rightharpoonup\widetilde\y$ weakly in $\V$, then $\widetilde\y$ minimizes $\Pathenergy^2[\y_-,\cdot,\y_+]$.
\end{lemma}
\begin{proof}
Note that the sets of minimizers for $\Pathenergy^2[\y_-^j,\cdot,\y_+^j]$ and $\Pathenergy^2[\y_-,\cdot,\y_+]$ coincide with the sets of minimizers for $\Pathenergy^j_\pm$ and $\Pathenergy_\pm$ (from the previous lemma), respectively.
The result on convergence of minimizers is now a standard property (see \cite[Thm.\,1.21]{Br02}) of the $\Gamma$-convergence from the previous lemma.
\end{proof}
\begin{theorem}[Existence of discrete spline interpolations] \label{thm:ExistenceDiscrete}
For $\sigma>0$ and $\energy$ admissible
there exists a discrete spline interpolation $\y^K=(\y_0, \ldots, \y_K)\in\manifold^{K+1}$ of \eqref{eq:ICdiscrete}
under discrete natural, Hermite, or periodic boundary conditions.
\end{theorem}
\begin{proof}
Let $\big((\y^j_0,\ldots, \y^j_K)\big)_{j=1,\ldots}$ be a minimizing sequence in $\manifold^{K+1}$ satisfying \eqref{eq:ICdiscrete},
and let the corresponding auxiliary variables be given by $(\widetilde{\y}^j_1,\ldots, \widetilde{\y}^j_{\widehat K})_{j=1,\ldots}$.
Now we consider an arbitrary comparison path $(\widehat{\y}_0,\ldots, \widehat{\y}_K)\in\manifold^{K+1}$ with $\hat  \y_{Kt_i} = \bar \y_i$, $i=1,\ldots, I$, and satisfying the boundary conditions.
Thus,  the energy on the minimizing sequence is bounded by
\[F^K[\widehat{\y}_0,\ldots, \widehat{\y}_K]+\sigma E^K[\widehat{\y}_0,\ldots, \widehat{\y}_K]\,,\]
which is finite since corresponding solutions of the constraint problems \eqref{eq:constraint} are known to exist \cite[Thm.\,4.3]{RuWi12b}.
Consequently, $\energy[\y^j_{k-1},\y^j_{k}]$ and $\energy[\y^j_k,\widetilde{\y}^j_{k}]$ are uniformly bounded for all $k$ and $j$.
Hence,  $\|\y^j_k\|_\V$ and $\|\widetilde{\y}^j_k\|_\V$ must be uniformly bounded due to the coercivity of $\energy$.
Due to the reflexivity of $\V$ 
there are subsequences, still denoted $(\y^j_0,\ldots, \y^j_K)$ and $(\widetilde{\y}^j_1,\ldots, \widetilde{\y}^j_{\widehat K})$,
which weakly converge in $\V$ to some $(\y_0,\ldots, \y_K)$ and $(\widetilde{\y}_0,\ldots, \widetilde{\y}_{\widehat K})$.
It is straightforward to see that  the limit path $(\y_0,\ldots, \y_K)$ fulfills the discrete boundary conditions and satisfies the interpolation constraint \eqref{eq:ICdiscrete}.
Furthermore, by Lemma\,\ref{thm:gamma}, $(\widetilde{\y}_0,\ldots, \widetilde{\y}_{\widehat K})$ satisfy \eqref{eq:constraint}.
Finally, by the weak lower semi-continuity of $\energy$ in both arguments, the functionals $\Splineenergy^K$ and $\Pathenergy^K$ are weakly lower semi-continuous  
along the sequences $(\y^j_1,\ldots, \y^j_{K})\rightharpoonup(\y_1,\ldots, \y_{K})$ and $(\widetilde{\y}^j_1,\ldots, \widetilde{\y}^j_{\widehat K})\rightharpoonup(\widetilde{\y}_1,\ldots, \widetilde{\y}_{\widehat K})$ so that we obtain
\begin{equation*}
\Splineenergy^K[\y_0, \ldots, \y_K]+\sigma\Pathenergy^K[\y_0, \ldots, \y_K]
\leq\liminf_{j\to\infty} \Splineenergy^K[\y_0^j, \ldots, \y_K^j]+\sigma\Pathenergy^K[\y_0^j, \ldots, \y_K^j]\,.
\end{equation*}
This proves that $(\y_0, \ldots, \y_K)$ minimizes $\Splineenergy^{\sigma,K}$ under \eqref{eq:ICdiscrete} and the chosen boundary condition.
\end{proof}

\section{$\Gamma$-convergence of the spline energy}\label{sec:gamma}
In this section we prove the $\Gamma$-convergence of the discrete regularized spline energy $\Splineenergy^{\sigma,K}$ 
to the continuous one $\splineenergy^\sigma$ as $K\to\infty$,
which justifies discrete spline interpolation as approximation of continuous spline interpolation.
In order to prove such a result we need to be able to compare $\Splineenergy^{\sigma,K}$ and $\splineenergy^\sigma$ as functionals.
For this reason we use a suitable interpolation to identify discrete with continuous curves 
so that we can rewrite the discrete energy $\Splineenergy^{\sigma,K}$ as a functional on continuous curves.
To this end, unless we consider periodic boundary conditions we define $\eta_{(\y_0,\ldots,\y_K)}:[0,1]\to\V$ as the cubic Hermite interpolation on intervals $[t^{k-\frac12}, t^{k+\frac12}]$ and an affine interpolation on 
$[0,t^{\frac12}]$ and $[t^{K-\frac12}, 1]$ with $t^{k\pm \frac12} = (k\pm \frac12) \tau$ and $\tau = \frac1K$,
\begin{equation*}
\eta_{(\y_0,\ldots,\y_K)}(t)=\begin{cases}
\y_0+(\y_1-\y_0)\frac{t}{\tau}&\text{if }t\in[0,t^{1/2}]\,,\\
\frac{\y_{k-1}+\y_k}2+(\y_k-\y_{k-1})\frac{t-t^{k-1/2}}{\tau}+(\y_{k+1}-2\y_k+\y_{k-1})\frac{(t-t^{k-1/2})^2}{2\tau^2}&\text{if }t\in[t^{k-1/2},t^{k+1/2}]\,,\\
\y_{K-1}+(\y_K-\y_{K-1})\frac{t-t^{K-1}}{\tau}&\text{if }t\in[t^{K-1/2},1]\,.
\end{cases}
\end{equation*}
In case of periodic boundary conditions, we shall instead use the simpler definition
\begin{equation*}\textstyle
\eta_{(\y_0,\ldots,\y_K)}(t)=
\frac{\y_{k-1}+\y_k}2+(\y_k-\y_{k-1})\frac{t-t^{k-1/2}}{\tau}+(\y_{k+1}-2\y_k+\y_{k-1})\frac{(t-t^{k-1/2})^2}{2\tau^2}\quad\text{for }t\in[t^{k-1/2},t^{k+1/2}]
\end{equation*}
with $k=1,\ldots,K$ and the convention $K+1\equiv1$ and $[t^{K-1/2},t^{K+1/2}]\equiv[0,t^{1/2}]\cup[t^{K-1/2},1]$ for notational convenience.
Note that the curve $\eta_{(\y_0,\ldots,\y_K)}$ lies in $C^1([0,1];\V)$.
With this interpolation at hand, the continuous representation of the discrete spline energy for $y\in W^{2,2}((0,1);\V)$ is given by
\begin{equation*}
\splineenergy^{\sigma,K}[\y]=\begin{cases}\Splineenergy^{\sigma,K}[\y_0,\ldots,\y_K]&\text{if }\y=\eta_{(\y_0,\ldots,\y_K) }
\text{ for some }(\y_0,\ldots,\y_K)\in\V^{K+1}\,,\\\infty&\text{else.}\end{cases}
\end{equation*}
\changed{
We will} also sometimes need to pass from a continuous to a discrete curve.
In detail, given $\y\in W^{2,2}((0,1);\V)$ we shall consider the discrete curve $(\y(t^0),\y(t^1),\ldots,\y(t^K))$.
With this discrete curve we also define 
\begin{equation*}
\eta_\y^K=\eta_{(\y(t^0),\ldots,\y(t^K))}\,.
\end{equation*}
In words, $\eta_\y^K$ is obtained from $\y$ by first evaluating $\y$ at regularly spaced points and then smoothly interpolating the midpoints in between.
%
In what follows we will show that 
\begin{itemize}
\item $\eta_\y^K\to\y$ strongly in $W^{2,2}((0,1);\V)$ for $\y\in C^3([0,1];\V)$ (Lemmas\,\ref{thm:weakConvergence} and \ref{thm:strongConvergence}),
\item for $y\in W^{2,2}((0,1);\V)$ there is a constant $\delta>0$ such that $d_K=\max_{k\in\{1,\ldots,K\}}\|\y_k-\y_{k-1}\|_\V<\delta$ with $y_k = y(t^k)$ implies
$$|\splineenergy^\sigma[\eta_\y^K]-\Splineenergy^{\sigma,K}[\y_0,\ldots,\y_K]|\leq f(\|\eta_\y^K\|_{W^{2,2}})/\sqrt K$$
for some increasing function $f$ (Lemmas\,\ref{thm:pathEnergyEstimate} and \ref{thm:splineEnergyEstimate}),
\item $C^3$-smooth curves on $\manifold$ are dense in $W^{2,2}((0,1);\manifold)$ (Lemma\,\ref{thm:dense}).
\end{itemize}
The $\Gamma$-convergence of the discrete against the continuous regularized spline energy will then follow quite automatically in Theorem\,\ref{thm:GammaConvergence}.

\begin{remark}[Alternative choices for $\eta_\y^K$]
Alternatively, one might also define $\eta_\y^K$ as the spline interpolation of $(\y(t^0),\ldots,\y(t^K))$, that is, a possibly non-unique curve which minimizes $\splineenergy^\sigma$ for fixed $\y(t^i)$, $i=0,\ldots,K$.
In that case one automatically has $\splineenergy^\sigma[\y]\geq\splineenergy^\sigma[\eta_\y^K]$, which in the proof of the $\limsup$-inequality would later render Lemma\,\ref{thm:strongConvergence} unnecessary.
With that choice, the $\Gamma$-convergence proof would have to be performed along the lines of the $\Gamma$-convergence proof for the discrete path energy in \cite[Thm.\,4.7]{RuWi12b}.
Note that there the chosen topology was $L^2((0,1);\y)$, but one could just as well choose the energy topology (in case of \cite[Thm.\,4.7]{RuWi12b} the weak $W^{1,2}((0,1);\V)$ topology
and in our case here the weak $W^{2,2}((0,1);\V)$ topology).
\end{remark}

To simplify the exposition, in the following Lemmas\,\ref{thm:weakConvergence}-\ref{thm:dense} and Theorems\,\ref{thm:GammaConvergence}-\ref{thm:equicoercivity} we will not explicitly treat the case of periodic boundary conditions.
The reader can readily assure herself that for periodic boundary conditions all statements and arguments remain true under the obvious modifications.
In particular, the interval $(0,1)$ will everywhere have to be replaced by the circle $S^1$.
\begin{lemma}[Weak convergence of interpolations]\label{thm:weakConvergence}
Let $\y\in W^{2,2}((0,1);\V)$, then $\eta_\y^K\rightharpoonup\y$ weakly in $W^{2,2}((0,1);\V)$.
\end{lemma}
\begin{proof}
Let us denote by $|\cdot|_{W^{n,2}}$ the $W^{n,2}((0,1);\V)$-seminorm, then
\begin{equation*}
|\eta_\y^K|_{W^{2,2}}^2=K^3\sum_{k=1}^{K-1}\|\y_{k-1}-2\y_k+\y_{k+1}\|_V^2\,.
\end{equation*}
Furthermore, due to the classical result by de Boor \cite{Bo63} the cubic spline $s:[0,2\tau]\to \R^2$ which interpolates $\y_{k-1},\y_k,\y_{k+1}$ at times $0,\tau,2\tau$ with natural boundary conditions
minimizes the $W^{2,2}((0,2\tau);\V)$-seminorm among all interpolating  curves in $W^{2,2}((0,2\tau);\V)$. It is easy to check that this cubic spline is given by 
\begin{equation*}
z(t)=\begin{cases}
\frac{t^3(\y_{k+1}-2\y_k+\y_{k-1})}{4\tau^3}-\frac{t(\y_{k+1}-6\y_k+5\y_{k-1})}{4\tau} + \y_{k-1},&t\in[0,\tau],\\[0.5ex]
-\frac{(t^3-6t^2\tau)(\y_{k+1}-2\y_k+\y_{k-1})}{4\tau^3} - \frac{t(7\y_{k+1}-18\y_k+11\y_{k-1})}{4\tau} + \frac{\y_{k+1}-2\y_k+3\y_{k-1}}{2},&t\in[\tau,2\tau],
\end{cases}
\end{equation*}
and satisfies
\begin{equation*}
|z|_{W^{2,2}}^2=\tfrac3{2\tau^3}\|\y_{k-1}-2\y_k+\y_{k+1}\|_\V^2\,.
\end{equation*}
Thus, the squared $W^{2,2}((t^{k-1},t^{k+1});\V)$-seminorm of $\y$ on any interval $[t^{k-1},t^{k+1}]$ is no smaller than $\tfrac3{2\tau^3}\|\y_{k-1}-2\y_k+\y_{k+1}\|_\V^2$ so that
\begin{equation*}
|\y|_{W^{2,2}}^2\geq\tfrac3{4\tau^3}\sum_{k=1}^{K-1}\|\y_{k-1}-2\y_k+\y_{k+1}\|_\V^2=\tfrac34|\eta_\y^K|_{W^{2,2}}^2\,.
\end{equation*}
Hence, $|\eta_\y^K|_{W^{2,2}}$ is uniformly bounded.
Moreover, it is easily verified that 
\begin{align*}
|\eta_\y^K|_{W^{1,2}}^2
&= \frac{\|\y_1\!-\!\y_0\|_\V^2}{2\tau}+\frac{\|\y_K\!-\!\y_{K-1}\|_\V^2}{2\tau} + \sum_{k=1}^{K-1} \frac1{6\tau} \left(3 \|\y_{k+1}\!-\!\y_{k}\|_\V^2 + 3 \|\y_{k}\!-\!\y_{k-1}\|_\V^2 - \|\y_{k+1}\!-\!2\y_k \!+\! \y_{k-1}\|_\V^2\right)\\
&\leq \frac{C}{\tau} \sum_{k=1}^{K}\|\y_k-\y_{k-1}\|_\V^2\leq C |\y|_{W^{1,2}}^2\,,
\end{align*}
so that $|\eta_\y^K|_{W^{1,2}}$ is uniformly bounded. 
The identity $\eta_\y^K(0)=\y(0)$ thus implies via Poincar\'e's inequality that $\|\eta_\y^K\|_{W^{1,2}}$ and therefore also $\|\eta_\y^K\|_{W^{2,2}}$ is uniformly bounded.
Consequently, every subsequence contains a weakly converging subsequence.
Now one readily verifies that every point $\eta_\y^K(t)$, $t\in[t^{k-1/2},t^{k+1/2}]$, is a convex combination of $\y_{k-1},\y_k,\y_{k+1}$.
As a result, $\eta_\y^K$ converges pointwise against $\y$.
Hence the limit of any weakly converging subsequence must coincide with $\y$,
and since the limit is the same for all subsequences, the whole sequence converges weakly against $\y$.
\end{proof}

In what follows we shall use the short form $|\cdot|$ for $\|\cdot\|_\V$.

\begin{lemma}[Strong convergence of interpolations]\label{thm:strongConvergence}
Let $\y\in C^3([0,1];\V)$, then $\eta_\y^K\to\y$ strongly in $W^{2,2}((0,1);\V)$.
\end{lemma}
\begin{proof}
From the previous lemma we already have weak convergence, so it remains to show that $|\eta_\y^K|_{W^{2,2}}\to|\y|_{W^{2,2}}$ as $K\to\infty$.
However, using Taylor expansion we obtain
\begin{align*}
\left||\y|_{W^{2,2}}^2\!-\!|\eta_\y^K|_{W^{2,2}}^2\right|
&=\left|\sum_{k=1}^K\int_{t^{k-1}}^{t^k}|\ddot\y|^2\d t
-\sum_{k=1}^{K-1}\int_{t^{k-1}}^{t^k}K^4|\y(t^{k+1})-2\y(t^k)+\y(t^{k-1})|^2\d t\right|\\
&\leq\tau\|\y\|_{C^2}^2+\!\sum_{k=1}^{K-1}\int_{t^{k-1}}^{t^k}\left|\ddot\y-\frac{\y(t^{k+1})\!-\!2\y(t^k)\!+\!\y(t^{k-1})}{\tau^2}\right|\left|\ddot\y+\frac{\y(t^{k+1})\!-\!2\y(t^k)\!+\!\y(t^{k-1})}{\tau^2}\right|\d t\\
&\leq\tau\|\y\|_{C^2}^2+\left(2\|\y\|_{C^2}+C\tau\|\y\|_{C^3}\right)\sum_{k=1}^{K-1}\int_{t^{k-1}}^{t^k}\left|\ddot\y-\frac{\y(t^{k+1})\!-\!2\y(t^k)\!+\!\y(t^{k-1})}{\tau^2}\right| \d t\\
&\leq\tau\|\y\|_{C^2}^2+\left(2\|\y\|_{C^2}+C\tau\|\y\|_{C^3}\right) C \tau \|\y\|_{C^3}
\mathop{\longrightarrow}_{K\to\infty}0
\end{align*}
for some constant $C<\infty$, which implies the strong convergence.
\end{proof}

\begin{lemma}[Path energy estimate]\label{thm:pathEnergyEstimate}
Setting $\y_k=\y(t^k)$, $k=0,\ldots,K$, if $d_K=\max_{k\in\{1,\ldots,K\}}|\y_k-\y_{k-1}|<\varepsilon\,$ for $\,\varepsilon$ from Definition\,\ref{def:admissibleEnergy},
then we have $$|\pathenergy[\eta_\y^K]-\Pathenergy^K[\y_0,\ldots,\y_K]|\leq C\frac{\|\eta_\y^K\|_{W^{2,2}}^3+\|\eta_\y^K\|_{W^{2,2}}^2}{\sqrt K}\,,$$
where the constant $C>0$ only depends on the metric $g$ and the function $\energy$.
\end{lemma}
\begin{proof}
Note that the extension of the metric $g$ and the energy $\energy$ onto $\V$ actually allows to interpret all of $\V$ as a Riemannian manifold
so that $\pathenergy[\eta_\y^K]$ and $\Pathenergy^K[\y_0,\ldots,\y_K]$ are well-defined even if $\eta_\y^K([0,1]),\{\y_0,\ldots,\y_K\}\not\subset\manifold$ and we obtain
\begin{align*}
\pathenergy[\eta_\y^K]&-\Pathenergy^K[\y_0,\ldots,\y_K]
=\sum_{k=1}^K\left(\int_{t^{k-1}}^{t^k}\metric_{\eta_\y^K}(\dot\eta_\y^K,\dot\eta_\y^K)\d t-K\energy(\y_{k-1},\y_k)\right)\\
&=\sum_{k=1}^K\left(\int_{t^{k-1}}^{t^k}\metric_{\eta_\y^K}(\dot\eta_\y^K,\dot\eta_\y^K)\d t-K\dist^2(\y_{k-1},\y_k)+KO(|\y_k-\y_{k-1}|^3)\right)\\
&=\sum_{k=1}^K\left(\int_{t^{k-1}}^{t^k}\metric_{\eta_\y^K}(\dot\eta_\y^K,\dot\eta_\y^K)\d t-K\metric_{\y_{k-1}}(\y_k-\y_{k-1},\y_k-\y_{k-1})+KO(|\y_k-\y_{k-1}|^3)\right)\\
&=\sum_{k=1}^K\left(\int_{t^{k-1}}^{t^k}(\metric_{\eta_\y^K}\!-\!\metric_{\y_{k-1}})\left(\frac{\y_k\!-\!\y_{k-1}}{\tau},\frac{\y_k\!-\!\y_{k-1}}{\tau}\right)\d t+O\big( \tau^\frac32 \|\eta_\y^K\|^2_{W^{2,2}}
\big)+KO(|\y_k\!-\!\y_{k-1}|^3)\right)\,,
\end{align*}
where in the last line we used $|\frac{\y_{k}-\y_{k-1}}{\tau}-\dot\eta_\y^K(t)|\leq|\frac{\y_{k}-\y_{k-1}}{\tau}-\dot\eta_\y^K(t^{k-1/2})|+\|\eta_\y^K\|_{C^{1,\frac12}}\tau^{1/2}=\|\eta_\y^K\|_{C^{1,\frac12}}\tau^{1/2}\leq O(\|\eta_\y^K\|_{W^{2,2}}\tau^{1/2})$.
This also implies $|\y_k-\y_{k-1}|\leq|\dot\eta_\y^K(t^{k-1/2})|\tau+O(\|\eta_\y^K\|_{W^{2,2}}\tau^{3/2})=O(\|\eta_\y^K\|_{W^{2,2}}\tau)$ so that we obtain
\begin{align*}
\pathenergy[\eta_\y^K]-\Pathenergy^K[\y_0,\ldots,\y_K]
&=\sum_{k=1}^K\left(\int_{t^{k-1}}^{t^k}\int_0^1 (D_\y\metric)_{s \y_{k-1} + (1-s)\eta_\y^K(t)}\d s \,(\eta_\y^K(t)-\y_{k-1})\left(\frac{\y_k-\y_{k-1}}{\tau},\frac{\y_k-\y_{k-1}}{\tau}\right)\d t\right)\\
& \quad +O(\|\eta_\y^K\|_{W^{2,2}}^3\tau+\|\eta_\y^K\|^2_{W^{2,2}}
\sqrt\tau)\\
&=O(\|\eta_\y^K\|_{W^{2,2}}^3\tau+\|\eta_\y^K\|_{W^{2,2}}^2\sqrt\tau)\,,
\end{align*}
where we used the boundedness of the metric derivative.
\end{proof}

\begin{lemma}[Spline energy estimate]\label{thm:splineEnergyEstimate}
For $\y:[0,1]\to\V$ set $\y_k=\y(t^k)$, $k=0,\ldots,K$.
If $d_K=\max_{k\in\{1,\ldots,K\}}|\y_k-\y_{k-1}|$ is small enough and $\|\eta_\y^K\|_{L^\infty}$ is bounded uniformly in $K$, then
$$|\splineenergy^\sigma[\eta_\y^K]-\Splineenergy^{\sigma,K}[\y_0,\ldots,\y_K]|\leq f(\|\eta_\y^K\|_{W^{2,2}})/\sqrt K$$
for some increasing function $f$ which only depends on the metric $g$, the function $\energy$, and the bound on $\|\eta_\y^K\|_{L^\infty}$.
\end{lemma}
\begin{proof} 
The estimate for the path energy follows from the previous lemma, only the estimate for the spline energy remains to be shown.

In the following estimates, the $O$-notation stands for a term whose constant only involves bounds of the (higher) derivatives of $\metric$.
Using that $\metric_{\widetilde{\y}_k} = \frac12 \energy_{,11}[\widetilde{\y}_k,\widetilde{\y}_k] 
=  \frac12 \energy_{,22}[\widetilde{\y}_k,\widetilde{\y}_k]= - \frac12 \energy_{,21}[\widetilde{\y}_k,\widetilde{\y}_k]$ (see \cite[Lemma 4.6]{RuWi12b}; an index $i$ after a comma shall denote differentiation with respect to the $i$\textsuperscript{th} argument)
and thus also $D_\y\metric_{\widetilde{\y}_k} = \frac12 (\energy_{,221}[\widetilde{\y}_k,\widetilde{\y}_k] +\energy_{,222}[\widetilde{\y}_k,\widetilde{\y}_k] )$
we get 
\begin{align}
& \metric_{\eta_\y^K}(\tfrac{D}{\d t}\dot\eta_\y^K,\psi)
=\metric_{\eta_\y^K}(\ddot\eta_\y^K,\psi)+\metric_{\eta_\y^K}(\Gamma_{\eta_\y^K}(\dot\eta_\y^K,\dot\eta_\y^K),\psi)\nonumber\\
&=\metric_{\eta_\y^K}(\ddot\eta_\y^K,\psi)+D_\y\metric_{\eta_\y^K}(\dot\eta_\y^K)(\dot\eta_\y^K,\psi)-\tfrac12D_\y\metric_{\eta_\y^K}(\psi)(\dot\eta_\y^K,\dot\eta_\y^K)\nonumber\\
&=\metric_{\widetilde{\y}_k}(\ddot\eta_\y^K,\psi)+O(|\widetilde{\y}_k-\eta_\y^K||\ddot\eta_\y^K||\psi|)+D_\y\metric_{\widetilde{\y}_k}(\dot\eta_\y^K)(\dot\eta_\y^K,\psi)\nonumber\\
& \quad -\tfrac12D_\y\metric_{\widetilde{\y}_k}(\psi)(\dot\eta_\y^K,\dot\eta_\y^K)+O(|\widetilde{\y}_k-\eta_\y^K||\dot\eta_\y^K|^2|\psi|)\nonumber\\
&=\tfrac12\energy_{,22}[\widetilde{\y}_k,\widetilde{\y}_k](\ddot\eta_\y^K,\psi)+\tfrac12\{\energy_{,221}+\energy_{,222}\}[\widetilde{\y}_k,\widetilde{\y}_k](\dot\eta_\y^K,\psi,\dot\eta_\y^K)\nonumber\\
&\quad-\tfrac14\{\energy_{,221}+\energy_{,222}\}[\widetilde{\y}_k,\widetilde{\y}_k](\dot\eta_\y^K,\dot\eta_\y^K,\psi)+O(|\eta_\y^K-\widetilde{\y}_k|(|\ddot\eta_\y^K|+|\dot\eta_\y^K|^2)|\psi|)\nonumber\\
&=\tfrac12\energy_{,22}[\widetilde{\y}_k,\widetilde{\y}_k](\ddot\eta_\y^K,\psi)-\tfrac14\{\energy_{,112}\!+\!\energy_{,221}\}[\widetilde{\y}_k,\widetilde{\y}_k](\dot\eta_\y^K,\dot\eta_\y^K,\psi)
+O(|\eta_\y^K\!-\!\widetilde{\y}_k|(|\ddot\eta_\y^K|\!+\!|\dot\eta_\y^K|^2))|\psi|\,. \label{eq:first}
\end{align}
In the last line we first used
$\energy_{,221}[\widetilde{\y}_k,\widetilde{\y}_k](\dot\eta_\y^K,\psi,\dot\eta_\y^K) = \energy_{,212}[\widetilde{\y}_k,\widetilde{\y}_k](\dot\eta_\y^K,\dot\eta_\y^K,\psi)$
due to Schwarz' theorem and then the identity
$$
2 \energy_{,212} +  \energy_{,222} -  \energy_{,221}  =   - 2 \energy_{,112} +  \energy_{,112} -  \energy_{,221}  = -(\energy_{,112} +  \energy_{,221} )
$$
at $(\widetilde{\y}_k,\widetilde{\y}_k)$ (cf.\ the transformations following \cite[(6.16)]{RuWi12b}). 
Next we use the necessary optimality conditions for $\widetilde{\y}_k$,
\begin{equation*}
\label{ELtildeyk}
0=\energy_{,2}[\y_{k-1},\widetilde{\y}_k](\psi) + \energy_{,1}[\widetilde{\y}_k,\y_{k+1}](\psi)\,.
\end{equation*}
Applying second order Taylor expansion and exploiting $\energy_{,1}[\widetilde{\y}_k,\widetilde{\y}_k]=\energy_{,2}[\widetilde{\y}_k,\widetilde{\y}_k]=0$ (see \cite[Lemma 4.6]{RuWi12b}) as well as Schwarz' theorem and the above-mentioned identities between the second derivatives of $\energy$ we obtain 
\begin{align}
0&=   \energy_{,21}[\widetilde{\y}_k,\widetilde{\y}_k](\psi,\y_{k-1}-\widetilde{\y}_k)+\energy_{,12}[\widetilde{\y}_k,\widetilde{\y}_k](\psi,\y_{k+1}-\widetilde{\y}_k)\nonumber\\
& \quad  +\tfrac12\left(\energy_{,211}[\widetilde{\y}_k,\widetilde{\y}_k](\psi,\widetilde{\y}_k\!-\!\y_{k-1},\widetilde{\y}_k\!-\!\y_{k-1})\!+\!\energy_{,122}[\widetilde{\y}_k,\widetilde{\y}_k](\psi,\y_{k+1}\!-\!\widetilde{\y}_k,\y_{k+1}\!-\!\widetilde{\y}_k)\right)\! +\! \mathcal{R}(\y_{k-1},\y_{k+1},\widetilde{\y}_k,\psi)\nonumber\\
& =  -\energy_{,11}[\widetilde{\y}_k,\widetilde{\y}_k](\psi,\y_{k-1}-\widetilde{\y}_k)-\energy_{,22}[\widetilde{\y}_k,\widetilde{\y}_k](\psi,\y_{k+1}-\widetilde{\y}_k)\nonumber\\
& \quad  +\tfrac12\left(\energy_{,112}[\widetilde{\y}_k,\widetilde{\y}_k](\widetilde{\y}_k\!-\!\y_{k-1},\widetilde{\y}_k\!-\!\y_{k-1},\psi)\!+\!\energy_{,221}[\widetilde{\y}_k,\widetilde{\y}_k](\y_{k+1}\!-\!\widetilde{\y}_k,\y_{k+1}\!-\!\widetilde{\y}_k,\psi)\right)\nonumber\\
&\quad+\mathcal{R}(\y_{k-1},\y_{k+1},\widetilde{\y}_k,\psi)\,, \label{eq:second}
\end{align}
where the remainder of the Taylor expansion satisfies $\mathcal{R}(\y_{k-1},\y_{k+1},\widetilde{\y}_k,\psi) \leq  C(|\y_{k-1}-\widetilde{\y}_k|^3+|\y_{k-1}-\widetilde{\y}_k|^3) |\psi|$
(cf.\ \cite[(6.19)]{RuWi12b} with the replacement of $\y_k^c$ by $\widetilde{\y}_k$ and $\y_k+\xi_k$ by $\y_{k+1}$).

In passing, let us remark here that the alternative choice $\widetilde{\y}_k=\frac{\y_{k-1}+\y_{k+1}}2$ instead of \eqref{eq:constraint} would violate the above equation and in general produces error terms no smaller than $O(|\y_{k+1}-\y_{k-1}|^2|\psi|)$,
which is too large to obtain a consistent approximation of the spline energy.

The constant $C$ depends on fourth derivatives of $\energy$ along the line segments  $[\y_{k-1},\widetilde{\y}_k]$ and $[\y_{k+1},\widetilde{\y}_k]$.
Since $\widetilde{\y}_k$ approximates $\frac{\y_{k-1}+\y_{k+1}}2$ for $d_K$ small enough \cite[Lemma\,5.5]{RuWi12b},
we have $C\leq\|\energy\|_{C^4(B\times B)}$ for $B\subset\V$ a ball around the origin with radius $2\|\eta_\y^K\|_{L^\infty}\leq2\|\eta_\y^K\|_{W^{2,2}}$,
so $C$ is an increasing function of $\|\eta_\y^K\|_{W^{2,2}}$.
Subtracting $K^2$ times \eqref{eq:second} from twice \eqref{eq:first} yields 
(after replacing $\energy_{,11}[\widetilde{\y}_k,\widetilde{\y}_k]=\energy_{,22}[\widetilde{\y}_k,\widetilde{\y}_k]=2\metric_{\widetilde{\y}_k}$)
\begin{align}
2\metric_{\eta_\y^K}(\tfrac{D}{\d t}\dot\eta_\y^K,\psi)
=&2\metric_{\widetilde{\y}_k}(\ddot\eta_\y^K-\tfrac{\y_{k+1}-2\widetilde{\y}_k+\y_{k-1}}{\tau^2},\psi)\nonumber\\
&-\tfrac12\big(\energy_{,112}[\widetilde{\y}_k,\widetilde{\y}_k](\dot\eta_\y^K,\dot\eta_\y^K,\psi)-\energy_{,112}[\widetilde{\y}_k,\widetilde{\y}_k](\tfrac{\widetilde{\y}_k-\y_{k-1}}\tau,\tfrac{\widetilde{\y}_k-\y_{k-1}}\tau,\psi)\nonumber\\
&\qquad+\energy_{,221}[\widetilde{\y}_k,\widetilde{\y}_k](\dot\eta_\y^K,\dot\eta_\y^K,\psi)-\energy_{,221}[\widetilde{\y}_k,\widetilde{\y}_k](\tfrac{\y_{k+1}-\widetilde{\y}_k}\tau,\tfrac{\y_{k+1}-\widetilde{\y}_k}\tau,\psi)\big)\nonumber\\
&+O\left(|\eta_\y^K-\widetilde{\y}_k|(|\ddot\eta_\y^K|+|\dot\eta_\y^K|^2)+\tfrac{|\y_{k-1}-\widetilde{\y}_k|^3}{\tau^2}+\tfrac{|\y_{k+1}-\widetilde{\y}_k|^3}{\tau^2}\right)|\psi|\nonumber\\
=&2\metric_{\widetilde{\y}_k}(2\tfrac{\widetilde{\y}_k-\y_k}{\tau^2},\psi) 
-\tfrac12 \energy_{,112}[\widetilde{\y}_k,\widetilde{\y}_k](\dot\eta_\y^K+\tfrac{\widetilde{\y}_k-\y_{k-1}}\tau,\dot\eta_\y^K-\tfrac{\widetilde{\y}_k-\y_{k-1}}\tau,\psi) \nonumber\\
& -\tfrac12 \energy_{,221}[\widetilde{\y}_k,\widetilde{\y}_k](\dot\eta_\y^K+\tfrac{\y_{k+1}-\widetilde{\y}_k}\tau,\dot\eta_\y^K-\tfrac{\y_{k+1}-\widetilde{\y}_k}\tau,\psi)\big)\nonumber\\
&+O\left(|\eta_\y^K-\widetilde{\y}_k|(|\ddot\eta_\y^K|+|\dot\eta_\y^K|^2)+\tfrac{|\y_{k-1}-\widetilde{\y}_k|^3}{\tau^2}+\tfrac{|\y_{k+1}-\widetilde{\y}_k|^3}{\tau^2}\right) |\psi|\nonumber\\
=&2\metric_{\eta_\y^K}(2\tfrac{\widetilde{\y}_k-\y_k}{\tau^2},\psi)\nonumber\\
&+O\Big(|\eta_\y^K-\widetilde{\y}_k|\big(|\ddot\eta_\y^K|+|\dot\eta_\y^K|^2+\tfrac{|\widetilde{\y}_k-\y_k|}{\tau^2}\big)+\tfrac{|\y_{k-1}-\widetilde{\y}_k|^3}{\tau^2}+\tfrac{|\y_{k+1}-\widetilde{\y}_k|^3}{\tau^2}\nonumber\\
&\; \; \qquad+\left(|\dot\eta_\y^K|+\tfrac{|\y_{k+1}-\widetilde{\y}_k|}\tau+\tfrac{|\y_{k-1}-\widetilde{\y}_k|}\tau\right)\left(|\dot\eta_\y^K-\tfrac{\widetilde{\y}_k-\y_{k-1}}\tau|+|\dot\eta_\y^K-\tfrac{\y_{k+1}-\widetilde{\y}_k}\tau|\right)\Big)|\psi|\,,
\label{eqn:covDerivEstimate}
\end{align}
where we used the fact $h(a,a,\psi)-h(b,b,\psi)=h(a+b,a-b,\psi)$ for any trilinear $h$ which is symmetric in its first two arguments.
Let us estimate the different components of the remainder terms.
Without writing down the dependence on $t$ explicitly, we consider $t\in[t^{k-1/2},t^{k+1/2}]$ throughout.
Note that $$d_K=\max_{k\in\{1,\ldots,K\}}|\tau\dot\eta_\y^K(t^{k-1/2})| \leq\tau\|\dot\eta_\y^K\|_{L^\infty}\leq\tau\|\eta_\y^K\|_{W^{2,2}}\,.$$
Using that $d_K^{3/2}\leq d_K$ for $d_K$ small enough, $|\dot\eta_\y^K|\leq\|\eta_\y^K\|_{C^{1,1/2}}\leq\|\eta_\y^K\|_{W^{2,2}}$,
and that $\widetilde{\y}_k=\tfrac{\y_{k+1}+\y_{k-1}}2+O(d_K^{3/2})$ by \cite[Lemma\,5.5]{RuWi12b} for $d_K$ small enough,
we obtain the different estimates
\begin{align*}
|\eta_\y^K-\y_k|
&=\left|\frac{\y_k-\y_{k-1}}\tau(t-t^{k-1/2})+\frac{\y_{k-1}-2\y_k+\y_{k+1}}{\tau^2}\frac{(t-t^{k-1/2})^2}2+O(d_K)\right|\\
&\leq\left|\frac{\y_k-\y_{k-1}}\tau\right|\tau+\left|\frac{\y_{k-1}-\y_k}\tau+\frac{\y_{k+1}-\y_k}\tau\right|\frac\tau2+O(d_K)
=O(d_K)\leq O(\tau\|\eta_\y^K\|_{W^{2,2}})\,,\\
|\eta_\y^K-\widetilde{\y}_k|
&\leq|\eta_\y^K-\y_k| + |\y_k-\tfrac{\y_{k+1}+\y_{k-1}}2  | + |\tfrac{\y_{k+1}+\y_{k-1}}2- \widetilde{\y}_k| \leq O(d_K)+O(d_K)+O(d_K^{3/2})\\
&=O(d_K)\leq O(\tau\|\eta_\y^K\|_{W^{2,2}})\,,\\
|\y_{k+1}-\widetilde{\y}_k|&=|\tfrac{\y_{k-1}-\y_{k+1}}2+O(d_K^{3/2})|\leq O(d_K)\leq O(\tau\|\eta_\y^K\|_{W^{2,2}})\,,\\
|\y_{k-1}-\widetilde{\y}_k|&=|\tfrac{\y_{k+1}-\y_{k-1}}2+O(d_K^{3/2})|\leq O(d_K)\leq O(\tau\|\eta_\y^K\|_{W^{2,2}})\,,\\
|\dot\eta_\y^K-\tfrac{\widetilde{\y}_k-\y_{k-1}}\tau|
&=\left|\tfrac{\y_k-\y_{k-1}}\tau+\tfrac{\y_{k+1}-2\y_k+\y_{k-1}}{\tau^2}(t-t^{k-1/2})-\tfrac{(\y_{k+1}+\y_{k-1})/2-\y_{k-1}}\tau\right|+O(d_K^{3/2}/\tau)\\
&=\left|\tfrac{\y_{k+1}-2\y_k+\y_{k-1}}{\tau^2}\right|\cdot|t-t^{k-1/2}-\tfrac\tau2|+O(\|\eta_\y^K\|_{W^{2,2}}^{3/2}\sqrt\tau)
=O(|\ddot\eta_\y^K|\tau+\|\eta_\y^K\|_{W^{2,2}}^{3/2}\sqrt\tau)\,.
\end{align*}
In the last estimate we used that $\ddot\eta_\y^K$ is constant on $[t^{k-1/2},t^{k+1/2}]$. Analogously, we get  
$$|\dot\eta_\y^K-\tfrac{\y_{k+1}-\widetilde{\y}_k}\tau|=O(|\ddot\eta_\y^K|\tau+\|\eta_\y^K\|_{W^{2,2}}^{3/2}\sqrt\tau)\,.$$
Since \eqref{eqn:covDerivEstimate} holds for all $\psi\in\V$, the above estimates for $\tau\leq1$ imply 
\begin{equation*}
\tfrac{D}{\d t}\dot\eta_\y^K=2\tfrac{\widetilde{\y}_k-\y_k}{\tau^2}+O\left(\tilde f(\|\eta_\y^K\|_{W^{2,2}})\sqrt\tau+\tau\|\eta_\y^K\|_{W^{2,2}}(|\ddot\eta_\y^K|+\tfrac{|\widetilde{\y}_k-\y_k|}{\tau^2})\right)
\end{equation*}
for some increasing function $\tilde f$.
This implies $\tfrac{|\widetilde{\y}_k-\y_k|}{\tau^2}=O(|\tfrac{D}{\d t}\dot\eta_\y^K|+\tilde f(\|\eta_\y^K\|_{W^{2,2}})\sqrt\tau+\tau\|\eta_\y^K\|_{W^{2,2}}|\ddot\eta_\y^K|)$ for any $t\in[t^{k-1/2},t^{k+1/2}]$ if $\tau$ is small enough depending on $\|\eta_\y^K\|_{W^{2,2}}$.
Furthermore, from the definition of $\tfrac{D}{\d t}\dot\eta_\y^K$ it follows that $$|\tfrac{D}{\d t}\dot\eta_\y^K|=|\ddot\eta_\y^K|+O(|\dot\eta_\y^K|^2)=|\ddot\eta_\y^K|+O(\|\eta_\y^K\|_{W^{2,2}}^2)$$ so that altogether
\begin{equation*}
\tfrac{D}{\d t}\dot\eta_\y^K=2\tfrac{\widetilde{\y}_k-\y_k}{\tau^2}+O\left(\tilde f(\|\eta_\y^K\|_{W^{2,2}})\sqrt\tau+\tau\|\eta_\y^K\|_{W^{2,2}}|\ddot\eta_\y^K|\right)\,,
\end{equation*}
potentially after adapting $\tilde f$.
Therefore, using once again that $\ddot\eta_\y^K$ is constant on $[t^{k-1/2},t^{k+1/2}]$ we achieve
\begin{align*}
\sum_{k=1}^{K-1}\int_{t^{k-1/2}}^{t^{k+1/2}}\metric_{\eta_\y^K}(\tfrac{D}{\d t}\dot\eta_\y^K,\tfrac{D}{\d t}\dot\eta_\y^K)\d t
=&\sum_{k=1}^{K-1}\int_{t^{k-1/2}}^{t^{k+1/2}} 
\metric_{\eta_\y^K}\left(2 \tfrac{\widetilde{\y}_k-\y_k}{\tau^2},2 \tfrac{\widetilde{\y}_k-\y_k}{\tau^2}\right)  
+ 2 \metric_{\eta_\y^K}\left(\tfrac{D}{\d t}\dot\eta_\y^K- 2\tfrac{\widetilde{\y}_k-\y_k}{\tau^2},\tfrac{D}{\d t}\dot\eta_\y^K\right) \\
& \qquad \qquad \qquad \; - \metric_{\eta_\y^K}\left(\tfrac{D}{\d t}\dot\eta_\y^K - 2 \tfrac{\widetilde{\y}_k-\y_k}{\tau^2},\tfrac{D}{\d t}\dot\eta_\y^K - 2 \tfrac{\widetilde{\y}_k-\y_k}{\tau^2}\right)\d t\\
=&\sum_{k=1}^{K-1}\tfrac4{\tau^4}\int_{t^{k-1/2}}^{t^{k+1/2}}\metric_{\eta_\y^K}(\widetilde{\y}_k-\y_k,\widetilde{\y}_k-\y_k)\d t\\
&+\sum_{k=1}^{K-1}O\Big(\left[\tilde f(\|\eta_\y^K\|_{W^{2,2}})\sqrt\tau+\tau\|\eta_\y^K\|_{W^{2,2}}|\ddot\eta_\y^K(t^k)|\right]\int_{t^{k-1/2}}^{t^{k+1/2}}|\tfrac{D}{\d t}\dot\eta_\y^K|\d t\\
& \qquad \qquad \;\; + \left[\tilde f(\|\eta_\y^K\|_{W^{2,2}})\sqrt\tau+\tau\|\eta_\y^K\|_{W^{2,2}}|\ddot\eta_\y^K(t^k)|\right]^2 \tau \Big)\,.
\end{align*}
Simple Taylor expansion now yields
\begin{align*}
\energy[\y_k,\widetilde{\y}_k]
&=\energy[\y_k,\y_k]+\energy_{,2}[\y_k,\y_k](\widetilde{\y}_k-\y_k)+\tfrac12\energy_{,22}[\y_k,\y_k](\widetilde{\y}_k-\y_k,\widetilde{\y}_k-\y_k)+O(|\widetilde{\y}_k-\y_k|^3)\\
&=\metric_{\y_k}(\widetilde{\y}_k-\y_k,\widetilde{\y}_k-\y_k)+O(|\widetilde{\y}_k-\y_k|^3)\\
&=\metric_{\eta_\y^K}(\y_k-\widetilde{\y}_k,\y_k-\widetilde{\y}_k)+O(|\widetilde{\y}_k-\y_k|^3+|\eta_\y^K-\y_k||\widetilde{\y}_k-\y_k|^2)\\
&=\metric_{\eta_\y^K}(\y_k-\widetilde{\y}_k,\y_k-\widetilde{\y}_k)+O((|\widetilde{\y}_k-\eta_\y^K|+|\eta_\y^K-\y_k|)|\widetilde{\y}_k-\y_k|^2)\\
&=\metric_{\eta_\y^K}(\y_k-\widetilde{\y}_k,\y_k-\widetilde{\y}_k)+O\left(\tfrac{|\widetilde{\y}_k-\y_k|^2}{\tau^4}\|\eta_\y^K\|_{W^{2,2}}\tau^5\right)\\
&=\metric_{\eta_\y^K}(\y_k-\widetilde{\y}_k,\y_k-\widetilde{\y}_k)+O((|\ddot\eta_\y^K|+\|\eta_\y^K\|_{W^{2,2}}^2)^2\|\eta_\y^K\|_{W^{2,2}}\tau^5)\,,
\end{align*}
where the constants only depend on the third derivative of $\energy$ and the first derivative of $\metric$. Here, 
we used that $\tfrac{|\widetilde{\y}_k-\y_k|}{\tau^2}=|\tfrac{D}{\d t}\dot\eta_\y^K|+O\left(\tilde f(\|\eta_\y^K\|_{W^{2,2}})\sqrt\tau+\tau\|\eta_\y^K\|_{W^{2,2}}|\ddot\eta_\y^K|\right)=O(|\ddot\eta_\y^K|+\|\eta_\y^K\|_{W^{2,2}}^2)$ for sufficiently small $\tau$ depending on $\|\eta_\y^K\|_{W^{2,2}}$.
Finally, we obtain
\begin{align*}
&\int_{t^{1/2}}^{t^{K-1/2}}\metric_{\eta_\y^K}(\tfrac{D}{\d t}\dot\eta_\y^K,\tfrac{D}{\d t}\dot\eta_\y^K)\d t\\
&=\frac{4}{\tau^3}\sum_{k=1}^{K-1}\energy[\y_k,\widetilde{\y}_k]
+\sum_{k=1}^{K-1}O\Big(
\left[\tilde f(\|\eta_\y^K\|_{W^{2,2}})\sqrt\tau+\tau\|\eta_\y^K\|_{W^{2,2}}|\ddot\eta_\y^K|\right]\textstyle\int_{t^{k-1/2}}^{t^{k+1/2}}|\tfrac{D}{\d t}\dot\eta_\y^K|\d t\Big)\\
&\qquad \qquad\qquad\qquad\qquad\qquad\quad\;+ \left[\tilde f(\|\eta_\y^K\|_{W^{2,2}})\sqrt\tau+\tau\|\eta_\y^K\|_{W^{2,2}}|\ddot\eta_\y^K|\right]^2 \tau 
+\tau^2\|\eta_\y^K\|_{W^{2,2}}(|\ddot\eta_\y^K|+\|\eta_\y^K\|_{W^{2,2}}^2)^2\Big)\\
&=\frac{4}{\tau^3}\sum_{k=1}^{K-1}\energy[\y_k,\widetilde{\y}_k]+\sum_{k=1}^{K-1}O\Big(\left(\tilde f(\|\eta_\y^K\|_{W^{2,2}})\sqrt\tau
+\tau\|\eta_\y^K\|_{W^{2,2}}|\ddot\eta_\y^K|\right)\tau(|\ddot\eta_\y^K|+\|\eta_\y^K\|_{W^{2,2}}^2)\\
&\qquad\qquad\qquad\qquad\qquad\qquad\;\; + \tilde f^2(\|\eta_\y^K\|_{W^{2,2}})\tau^2 +\tau^2\|\eta_\y^K\|_{W^{2,2}}(|\ddot\eta_\y^K|+\|\eta_\y^K\|_{W^{2,2}}^2)^2\Big)\\
&=\frac{4}{\tau^3}\sum_{k=1}^{K-1}\energy[\y_k,\widetilde{\y}_k]+\sum_{k=1}^{K-1}O\Big(\hat f(\|\eta_\y^K\|_{W^{2,2}})(\tau^{3/2}+\tau^{3/2}|\ddot\eta_\y^K|+\tau^2|\ddot\eta_\y^K|^2)\Big)
\end{align*}
for some increasing function $\hat f$.
Noting $\sum_{k=1}^{K-1}\tau|\ddot\eta_\y^K|^2=|\eta_\y^K|_{W^{2,2}}^2$
and $\sum_{k=1}^{K-1}\tau|\ddot\eta_\y^K|=|\eta_\y^K|_{W^{2,1}}\leq|\eta_\y^K|_{W^{2,2}}$ we arrive at
\begin{equation*}
\int_{t^{1/2}}^{t^{K-1/2}}\metric_{\eta_\y^K}(\tfrac{D}{\d t}\dot\eta_\y^K,\tfrac{D}{\d t}\dot\eta_\y^K)\d t
=\frac{4}{\tau^3}\sum_{k=1}^{K-1}\energy[\y_k,\widetilde{\y}_k]+f(\|\eta_\y^K\|_{W^{2,2}})\sqrt\tau
\end{equation*}
for some increasing function $f$.
Finally, it is straightforward to show
$$
\int_{[0,t^{1/2}]\cup[t^{K-1/2},1]}\metric_{\eta_\y^K}(\tfrac{D}{\d t}\dot\eta_\y^K,\tfrac{D}{\d t}\dot\eta_\y^K)\d t \leq
C (\|\eta_\y^K\|^2_{W^{2,2}})\tau\,.\qedhere
$$
\end{proof}

\begin{lemma}[Density of smooth curves]\label{thm:dense}
The set $C^3([0,1];\manifold)$ is dense in $W^{2,2}((0,1);\V)$.
Likewise, the set of curves in $C^3([0,1];\manifold)$ satisfying \eqref{eq:IC} is dense among the curves in $W^{2,2}((0,1);\manifold)$ satisfying \eqref{eq:IC}.
Finally, the set 
$$
C_H=\{\y\in C^3([0,1];\manifold)\,|\,\text{\eqref{eq:IC} and }\exists\delta>0:\y(t)=
\bar\y_1+tv_0\text{ on }(0,\delta)\text{ and }\y(t)=\bar\y_I-tv_1\text{ on }(1-\delta,1)\}
$$ for given $v_0\in T_{\bar\y_1}\manifold,v_1\in T_{\bar\y_I}\manifold$
is dense in $\{\y\in W^{2,2}((0,1);\manifold)\,|\,\y(0)=\bar\y_1,\y(1)=\bar\y_I,\dot\y(0)=v_0,\dot\y(1)=v_1\}$.
\end{lemma}
\begin{proof}
First we show that piecewise cubic polynomials are dense in $W^{2,2}((0,1);\V)$.
To this end let us consider  a curve $\y\in W^{2,2}((0,1);\V)$ and define $z_\y^K$ as the cubic spline interpolation of $\y(t^0),\ldots,\y(t^K)$ with $t^k=\frac kK$.
In more detail, if without loss of generality we assume $\y(t^0)=0$ and we set $\V^K=\mathrm{span}(\y(t^1),\ldots,\y(t^K))$,
then $z_\y^K$ is defined as the standard cubic spline through $\y(t^0),\ldots,\y(t^K)$ in the finite-dimensional linear space $\V^K$.
By the already mentioned result by de Boor \cite{Bo63}, $z_\y^K$ minimizes $\splineenergy[z]=\int_0^1|\ddot z(t)|^2\d t$ among all interpolating curves $z:(0,1)\to\V^K$
and therefore also among all interpolating curves $z:(0,1)\to\V$.
Therefore, $|z_\y^K|_{W^{2,2}}\leq|\y|_{W^{2,2}}$ for all $K$.
Since $z_\y^K(0)=\y(0)$ and $z_\y^K(1)=\y(1)$ are fixed we may apply Poincar\'e's inequality to derive that the norms $\|z_\y^K\|_{W^{2,2}}$ are bounded uniformly in $K$.
Thus, from any subsequence for $K\to\infty$ we can extract a weakly converging further subsequence, which must converge to $\y$ since it also converges pointwise to $\y$.
Therefore, $z_\y^K\rightharpoonup\y$ weakly in $W^{2,2}((0,1);\V)$.
The strong convergence $z_\y^K\to\y$ now follows from $|\y|_{W^{2,2}}\leq\liminf_{K\to\infty}|z_\y^K|_{W^{2,2}}\leq|\y|_{W^{2,2}}$,
where the first inequality represents the weak lower semi-continuity of the seminorm.

Next we imply that smooth curves are dense in $W^{2,2}((0,1);\V)$.
Indeed, by standard mollifying arguments for curves in finite-dimensional spaces, any $z_\y^K$ can be smoothed
to yield a function $\tilde z_\y^K\in C^\infty([0,1];\V^K)\subset C^\infty([0,1];\V)$ with $\|\tilde z_\y^K- z_\y^K\|_{W^{2,2}}<\frac1K$
so that $\tilde z_\y^K\to\y$ in $W^{2,2}((0,1);\V)$.

The argument can readily be adapted to the case in which the interpolation condition \eqref{eq:IC} has to be satisfied in addition.
Indeed, the cubic splines are now simply required to also interpolate the points $\bar\y_i$ at times $t_i$, $i=1,\ldots,I$,
and the curve mollification in the finite-dimensional space $\V^K$ is performed such that it does not violate \eqref{eq:IC}.
Similar modifications can be performed to achieve the desired density of the set $C_H$;
indeed, now the cubic interpolations are adapted to just start and end with a sufficiently small linear segment $\bar\y_1+tv_0$ and $\bar\y_I-tv_1$,
and the mollification again is performed so as to still keep a short linear segment near both curve ends.
\end{proof}

We are finally in the position to prove the $\Gamma$-convergence of the discrete spline energy to the continuous one.
At this point we shall also take account of the constraint that the continuous and discrete curve be in $\manifold$
and satisfy the interpolation condition \eqref{eq:IC} as well as one of the boundary conditions \eqref{eqn:naturalBC}-\eqref{eqn:periodicBC}
(\eqref{eq:ICdiscrete} and one of \eqref{eqn:discreteNaturalBC}-\eqref{eqn:discretePeriodicBC} for the discrete curve).
To this end we introduce the indicator functions of the corresponding constraints as
\begin{align*}
\mathcal I[\y]&=\begin{cases}
0&\text{if }\y\text{ satisfies \eqref{eq:IC} and the chosen boundary condition,}\\
\infty&\text{else,}
\end{cases}\\
\mathbf I^K[\y_0,\ldots,\y_K]&=\begin{cases}
0&\text{if }(\y_0,\ldots,\y_K)\text{ satisfies \eqref{eq:ICdiscrete} and the chosen boundary condition,}\\
\infty&\text{else,}
\end{cases}\\
\mathcal I^K[\y]&=\begin{cases}
\mathbf I^K[\y_0,\ldots,\y_K]&\text{if }\y=\eta_{(\y_0,\ldots,\y_K)}\text{ for some }(\y_0,\ldots,\y_K)\in\V^{K+1}\,,\\
\infty&\text{else.}
\end{cases}
\end{align*}
With regard to our convention for \eqref{eq:ICdiscrete} concerning the compatibility of the interpolation times $t_i$ and the number $K$ of discrete points along a discrete curve,
we shall in the following and without explicit mention always interpret $K\to\infty$ as a sequence of integers approaching infinity such that $Kt_i$ is integer for $i=1,\ldots,I$.

\begin{theorem}[$\Gamma$-limit of discrete regularized spline energy]\label{thm:GammaConvergence}
Let $(\manifold,g)$ as well as $\energy$ be admissible and $\sigma>0$.
With respect to weak convergence in $W^{2,2}((0,1);\V)$
we have $\Gammalim_{K\to\infty}\splineenergy^{\sigma,K}+\mathcal I^K=\splineenergy^\sigma+\mathcal I$.
\end{theorem}
\begin{proof} We have to prove the two defining properties of $\Gamma$-convergence, namely, the $\liminf$- and the $\limsup$-inequality.\medskip

\noindent\underline{\emph{$\liminf$-inequality:}}
Let $\y^K\rightharpoonup\y$ in $W^{2,2}((0,1);\V)$, we need to show $\liminf_{K\to\infty}\splineenergy^{\sigma,K}[\y^K]+\mathcal I^K[\y^K]\geq\splineenergy^\sigma[\y]+\mathcal I[\y]$.
Upon taking a subsequence, we may replace the $\liminf$ by a $\lim$ and may assume $\splineenergy^{\sigma,K}[\y^K]+\mathcal I^K[\y^K]<C$ for some constant $C<\infty$ and uniformly for all $K$ (otherwise there is nothing to show).
Thus we have $\y^K=\eta_{(\y_0^K,\ldots,\y_K^K)}$ for some sequence $(\y_0^K,\ldots,\y_K^K)$ of $(K+1)$-tuples in $\manifold$.
Since the discrete path energy is part of $\splineenergy^{\sigma,K}$,
the bound $\splineenergy^{\sigma,K}(\y^K)<C$ together with Definition\,\ref{def:admissibleEnergy}\ref{enm:discreteCoercivity} implies that
$$d_K=\max_{k\in\{1,\ldots,K\}}|\y_k^K-\y_{k-1}^K|\leq\gamma^{-1}(\pathenergy^K[\y^K]/K)\leq\gamma^{-1}(\tfrac{C}{\sigma K})\,,$$ 
which converges to zero as $K\to\infty$.

We next show $\mathcal I[\y]=0$.
Note that by construction $\y^K(t)=\eta_{(\y_0^K,\ldots,\y_K^K)}(t)$ lies for any $t\in(0,1)$ in the convex hull of three consecutive points $\y_{k-1}^K,\y_k^K,\y_{k+1}^K$.
Thus, for each $i=1,\ldots,I$ we have $|\y^K(t_i)-\bar\y_i|=|\y^K(t_i)-\y_{Kt_i}^K|\leq d_K$ so that the limit $\y$ satisfies \eqref{eq:IC}.
Furthermore, the limit $\y$ satisfies the continuous counterpart of the chosen discrete boundary condition.
Indeed, for natural and periodic boundary conditions this is trivial (recall that for periodic boundary conditions all arguments are performed for $(0,1)$ replaced by $S^1$),
and for Hermite boundary conditions we have $\dot\y^K(0)=K(\y_1^K-\y_0^K)=v_0$ and $\dot\y^K(1)=v_1$,
which implies that $\y$ satisfies the continuous Hermite boundary conditions.

To conclude, by the weak lower semi-continuity of $\splineenergy^\sigma$ from Lemma\,\ref{thm:energyLSC} and by Lemma\,\ref{thm:splineEnergyEstimate} we have
\begin{equation*}
\splineenergy^\sigma[\y]+\mathcal I[\y]
=\splineenergy^\sigma[\y]
\leq\liminf_{K\to\infty}\splineenergy^\sigma[\y^K]
\leq\liminf_{K\to\infty}\splineenergy^{\sigma,K}[\y^K]+f(\|\y^K\|_{W^{2,2}})/\sqrt K
\leq\liminf_{K\to\infty}\splineenergy^{\sigma,K}[\y^K]+\mathcal I^K[\y^K]\,,
\end{equation*}
where we have used the uniform boundedness of $\|\y^K\|_{W^{2,2}}$ due to the weak convergence of $\y^K$.\medskip

\noindent\underline{\emph{$\limsup$-inequality:}}
Let $\y\in C^3([0,1];\manifold)$ with finite energy
(in the case of Hermite boundary conditions we require additionally $\y\in C_H$)
and choose as the recovery sequence $\y^K = \eta_\y^K$.
By definition we have $\mathcal I^K[\y^K]=0$.
As $K\to\infty$, we have $d_K=\max_{k\in\{1,\ldots,K\}}|\y(t^k)-\y(t^{k-1})|\to0$ as well as $\eta_\y^K\to\y$ strongly in $W^{2,2}((0,1);\V)$ by Lemma\,\ref{thm:strongConvergence}.
Thus, by the strong $W^{2,2}((0,1);\V)$-continuity of $\splineenergy^\sigma$ from Lemma\,\ref{thm:energyLSC} and by Lemma\,\ref{thm:splineEnergyEstimate} we have
\begin{equation*}
\splineenergy^\sigma[\y]+\mathcal I[\y]
=\splineenergy^\sigma[\y]
=\lim_{K\to\infty}\splineenergy^\sigma[\eta_\y^K]
\geq\limsup_{K\to\infty}\splineenergy^{\sigma,K}[\eta_\y^K]-f(\|\eta_\y^K\|_{W^{2,2}})/\sqrt K
=\limsup_{K\to\infty}\splineenergy^{\sigma,K}[\eta_\y^K]+\mathcal I^K[\y^K]\,,
\end{equation*}
where we have used the uniform boundedness of $\|\y^K\|_{W^{2,2}}$ due to the strong convergence of $\eta_{\y}^K$.
Thus, we obtain $\Gammalimsup_{K\to\infty}\splineenergy^{\sigma,K}+\mathcal I^K\leq\splineenergy^\sigma+\mathcal I$ on $C^3([0,1];\manifold)$ (on $C_H$ in case of Hermite boundary conditions).
By Lemma\,\ref{thm:dense} we have that $C^3([0,1];\manifold)\cap\mathop{\mathrm{dom}}\mathcal I$ \; ($C_H\cap\mathop{\mathrm{dom}}\mathcal I=C_H$ in case of Hermite boundary conditions) is dense in $W^{2,2}((0,1);\V)\cap\mathop{\mathrm{dom}}\mathcal I$.
Hence, the strong $W^{2,2}((0,1);\V)$-continuity of $\splineenergy^\sigma$
implies
$\Gammalimsup_{K\to\infty}\splineenergy^{\sigma,K}+\mathcal I^K\leq\splineenergy^\sigma+\mathcal I$ on all of $W^{2,2}((0,1);\V)$.
\end{proof}

\begin{remark}[Mosco convergence]
Since the recovery sequence converges strongly in the above proof, the $\Gamma$-limit actually even is a Mosco limit.
\end{remark}

Finally we show that discrete spline interpolations converge against continuous ones.
This is a well-known immediate consequence of the equicoercivity of the discrete spline energies, whose proof parallels the coercivity proof performed for Theorem\,\ref{thm:ExistenceContinuous}.

\begin{theorem}[Equicoercivity]\label{thm:equicoercivity}
Let $(\manifold,g)$ as well as $\energy$ be admissible and $\sigma>0$.
Any sequence $\y^K$ with $\splineenergy^{\sigma,K}[\y^K]+\mathcal I^K[\y^K]$ uniformly bounded contains a subsequence converging weakly in $W^{2,2}((0,1);\V)$.
As a consequence, any sequence of minimizers for $\splineenergy^{\sigma,K}+\mathcal I^K$ contains a subsequence converging weakly to a minimizer of $\splineenergy^\sigma+\mathcal I$.
\end{theorem}
\begin{proof}
Let $\y^K$ be a sequence with $\splineenergy^{\sigma,K}[\y^K]+\mathcal I^K[\y^K]$ uniformly bounded.
Just as in the previous proof we have $\y^K=\eta_{(\y_0^K,\ldots,\y_K^K)}$ for some sequence $(\y_0^K,\ldots,\y_K^K)$ of $(K+1)$-tuples in $\manifold$,
and we may deduce $d_K=\max_{k\in\{1,\ldots,K\}}\|\y_k^K-\y_{k-1}^K\|_\V\to0$.
Furthermore, as in the proof of Lemma\,\ref{thm:weakConvergence} one obtains the bound
\begin{equation*}
|\y^K|_{W^{1,2}}^2
=|\eta_{(\y_0^K,\ldots,\y_K^K)}|_{W^{1,2}}^2
\leq CK\sum_{k=1}^K\|\y_k^K-\y_{k-1}^K\|_\V^2
\end{equation*}
for the $W^{1,2}((0,1);\V)$-seminorm and some fixed constant $C>0$.
Due to Definition\,\ref{def:admissibleMetric}\ref{enm:coercivity} and Definition\,\ref{def:admissibleEnergy}\ref{enm:discreteApprox} in combination with $d_K\to0$ we have
\begin{equation*}
c^\ast\|\y^K_k-\y^K_{k-1}\|_\V^2\leq\dist^2(\y^K_{k-1},\y^K_k)\leq2\energy[\y^K_{k-1},\y^K_k]
\end{equation*}
for $K$ large enough so that with Jensen's inequality we obtain
\begin{equation*}
\frac{\splineenergy^{\sigma,K}[\y^K]}\sigma
\geq\Pathenergy^K[\y_0^K,\ldots,\y_K^K]
=K\sum_{k=1}^K\energy[\y^K_{k-1},\y^K_k]
\geq\frac{c^\ast}2 K\sum_{k=1}^K\|\y^K_k-\y^K_{k-1}\|_\V^2
\geq\frac{c^\ast}{2C}|\y^K|_{W^{1,2}}\,.
\end{equation*}
Therefore, $|\y^K|_{W^{1,2}}$ is uniformly bounded.
With Poincar\'e's inequality and the fact that $\y^K(t^k)$ lies in the convex hull of $\y_{k-1}^K,\y_k^K,\y_{k+1}^K$ so that $\|\y^K(t_i)-\bar\y_i\|_\V\leq d_K\to0$
we obtain uniform boundedness of $\|\y^K\|_{W^{1,2}}$.
It remains to show the boundedness of the seminorm
\begin{equation*}
|\y^K|_{W^{2,2}}^2=|\eta_{(\y_0^K,\ldots,\y_K^K)}|_{W^{2,2}}^2=K^3\sum_{k=1}^{K-1}\|\y_{k-1}^K-2\y_k^K+\y_{k+1}^K\|_V^2\,.
\end{equation*}
\newcommand{\yc}{{\y_k^K}}%
\renewcommand{\yt}{{\widetilde\y_k^K}}%
\renewcommand{\ym}{{\y_{k-1}^K}}%
\renewcommand{\yp}{{\y_{k+1}^K}}%
Let us denote the auxiliary points defined in \eqref{eq:constraint} and belonging to $(\y_0^K,\ldots,\y_K^K)$ by $(\widetilde\y_1^K,\ldots,\widetilde\y_{K-1}^K)$.
Note that by \cite[Lemma\,5.5]{RuWi12b} we have $\|\yt-\frac12(\ym+\yp)\|_\V=O(\|\yp-\ym\|_\V^{3/2})=O(d_K^{3/2})$ for $K$ large enough
so that by the triangle inequality $\|\yt-\yc\|_\V\leq\|\yt-\frac12(\ym+\yp)\|_\V+\frac12\left(\|\ym-\yc\|_\V+\|\yp-\yc\|_\V\right)=O(d_K)$.
Thus we can estimate $\energy[\yc,\yt]$ below by $\frac{c^\ast}2\|\yt-\yc\|_\V^2$ for $K$ large enough and obtain
\begin{align*}
\splineenergy^\sigma[\y^K]
&\geq4K^3\sum_{k=1}^{K-1}\energy[\yc,\yt] \geq\frac{c^\ast}2K^3\sum_{k=1}^{K-1}\|2\yt-2\yc\|_V^2\\
&=\frac{c^\ast}2\sum_{k=1}^{K-1}\tau\left\|\frac{\yp+\ym-2\yc}{\tau^2}-\frac{\yp+\ym-2\yt}{\tau^2}\right\|_\V^2\\
&\geq\frac{c^\ast}2\sum_{k=1}^{K-1}\frac\tau2\left\|\frac{\yp+\ym-2\yc}{\tau^2}\right\|_\V^2-\tau\left\|\frac{\yp+\ym-2\yt}{\tau^2}\right\|_\V^2\\
&=\frac{c^\ast}2\left(\frac12|\y^K|_{W^{2,2}}^2-\sum_{k=1}^{K-1}\tau\left\|\frac{\yp+\ym-2\yt}{\tau^2}\right\|_\V^2\right)\,.
\end{align*}
Apparently, $(\yp+\ym-2\yt)/\tau^2$ is the discrete equivalent to $\Gamma_{\y}(\dot\y,\dot\y)$ from the proof of Theorem\,\ref{thm:ExistenceContinuous}.
We will next show that $\|\yp+\ym-2\yt\|_V\leq C(\|\yc-\ym\|_\V+\|\yp-\y_{k}\|_\V)^2$ for some constant $C>0$
in analogy to the estimate $\|\Gamma_{\y}(\dot\y,\dot\y)\|_\V\leq C\|\dot\y\|_\V^2$.
To this end, consider the second order Taylor expansion as in the proof of  \cite[Thm. 4.10]{RuWi12b},
\begin{align*}
\energy[\ym,\yt] =&  \energy[\yt,\yt] \!+\! \energy_{,1}[\yt,\yt] (\ym\!-\!\yt) \\& + 
\int_0^1 (1\!-\!s) \energy_{,11}[\yt \!+\! s(\ym\!-\!\yt),\yt] (\ym\!-\!\yt,\ym\!-\!\yt) \d s\,, \\
\energy[\yt,\yp] =&  \energy[\yt,\yt] \!+\! \energy_{,2}[\yt,\yt] (\yp\!-\!\yt)  \\& + 
\int_0^1 (1\!-\!s) \energy_{,22}[\yt,\yt \!+\! s(\yp\!-\!\yt)] (\yp\!-\!\yt,\yp\!-\!\yt) \d s\,,
\end{align*}
and use this together with $ \energy[\yt,\yt]=0$ and $\energy_{,1}[\yt,\yt] = \energy_{,2}[\yt,\yt]=0$
to expand the first order condition $0=\energy_{,2}[\ym,\yt]+ \energy_{,1}[\yt,\yp]$ for $\yt$ into
\begin{align*}
0 = \int_0^1 & (1-s)^2 \energy_{,111}[\yt + s(\ym\!-\!\yt),\yt] (\ym\!-\!\yt,\ym\!-\!\yt) \\[-1.5ex]
& + (1\!-\!s) \energy_{,112}[\yt \!+ \!s(\ym\!-\!\yt),\yt] (\ym-\yt,\ym-\yt) \\
&+ 2 (1\!-\!s) \energy_{,11}[\yt \!+ \!s(\ym\!-\!\yt),\yt] (\yt-\ym) \\
&+ (1\!-\!s)^2 \energy_{,222}[\yt,\yt \!+\! s(\yp\!-\!\yt)] (\yp-\yt,\yp-\yt) \\
& + (1\!-\!s) \energy_{,221}[\yt,\yt \!+\! s(\yp\!-\!\yt)] (\yp-\yt,\yp-\yt) \\
&+ 2 (1\!-\!s) \energy_{,22}[\yt, \yt \!+\! s(\yp\!-\!\yt)] (\yt-\yp) \d s\,.
\end{align*}
Using $g_\yt = \frac12 \energy_{,11}[\yt,\yt] = \frac12 \energy_{,22}[\yt,\yt]$  we obtain for the third and the sixth term
\begin{multline*}
 2 \int_0^1   (1-s) \left(\energy_{,11}[\yt + s(\ym-\yt),\yt] (\yt-\ym) + \energy_{,22}[\yt, \yt + s(\yp-\yt)] (\yt-\yp)\right) \d s\\
=  -2 g_\yt(\ym-2\yt+\yp,\cdot) + O(\|\ym-\yt\|_\V^2+\|\yp-\yt\|_\V^2) \,,
\end{multline*}
where the involved constants only depend on the third derivative of $\energy$ on a ball around the origin of radius $\|\y^K\|_{L^\infty}$.
All other terms can as well be estimated by $O(\|\ym-\yt\|_\V^2+\|\yp-\yt\|_\V^2)$. Furthermore, with $\|\yt-\frac12(\ym+\yp)\|_\V=O(\|\yp-\ym\|_\V^{3/2})$ 
(cf.\ again \cite[Lemma\,5.5]{RuWi12b}) we have the straightforward estimate
\begin{equation*}
\|\ym-\yt\|_\V^2+\|\yp-\yt\|_\V^2 \leq  C \|\ym-\yp\|_\V^2  \leq C \left( \|\ym-\yc\|_\V+\|\yp-\yc\|_\V\right)^2
\end{equation*}
for some $C>0$.
Thus, we finally verified the claim and obtain
\begin{equation*}
\splineenergy^\sigma[\y^K]
\geq \frac{c^\ast}4\left(|\y^K|_{W^{2,2}}^2- C \sum_{k=1}^{K}\tau \left\|\frac{\yc-\ym}{\tau} \right\|_\V^4 \right)\,.
\end{equation*}
Next one can readily compute
\begin{multline*}
|\y^K|_{W^{1,4}}^4
=|\eta_{(\y_0^K,\ldots,\y_K^K)}|_{W^{1,4}}^4
=\frac{\|\y_1^K-\y_0^K\|_\V^4}{2\tau^3}+\frac{\|\y_K^K-\y_{K-1}^K\|_\V^4}{2\tau^3}
+\frac{K^3}5\sum_{k=1}^{K-1}\Big(\|\yc-\ym\|_\V^4\\
+\|\yc-\ym\|_\V^3\|\yp-\yc\|_\V
+\|\yc-\ym\|_\V^2\|\yp-\yc\|_\V^2+\|\yc-\ym\|_\V\|\yp-\yc\|_\V^3+\|\yp-\yc\|_\V^4\Big)\\
\geq\frac15\sum_{k=1}^{K}\tau\left\|\frac{\yc-\ym}{\tau} \right\|_\V^4
\end{multline*}
so that the above estimate turns into
\begin{equation*}
\splineenergy^\sigma[\y^K]
\geq \frac{c^\ast}4\left(|\y^K|_{W^{2,2}}^2-5C|\y^K|_{W^{1,4}}^4\right)\,.
\end{equation*}
The left-hand side is uniformly bounded, and in the proof of Theorem \ref{thm:ExistenceContinuous} we have already shown that this together with the uniform bound on $\|\y^K\|_{W^{1,2}}$ implies uniform boundedness of $|\y^K|_{W^{2,2}}$ independent of $K$.
Thus there exists a subsequence, which is weakly converging  in $W^{2,2}$.
\renewcommand{\yt}{{\widetilde\y_k}}
\renewcommand{\ym}{{\y_{k-1}}}
\renewcommand{\yp}{{\y_{k+1}}}

The statement about sequences of minimizers now is a standard consequence of the $\Gamma$-convergence from the previous theorem
and the fact that with an arbitrary smooth $\y$ satisfying the interpolation and boundary conditions we have
$\inf_{\widehat\y}\splineenergy^{\sigma,K}[\widehat\y]+\mathcal I^K[\widehat\y]\leq\splineenergy^{\sigma,K}[\eta_\y^K]$,
which is uniformly bounded.
\end{proof}


\section{Example settings and applications}\label{sec:numerics}
In this section we consider the application to manifolds of increasing complexity,
first to $d$-dimensional surfaces embedded in $\R^m$,
then to the high-dimensional shape manifold of discrete shells,
and finally to the infinite-dimensional shape manifold of viscous rods.

\subsection{Embedded finite-dimensional manifolds}\label{sec:embedded}
Here we consider a closed, smooth $d$-dimensional manifold embedded in $\R^m$, $m>d$ (if $\manifold$ has a boundary, it shall be smooth).
We shall work with a (potentially local) parameterization $\varphi:\R^d\supset U\to\R^m$ of the manifold
so that we can choose $\V=\Y=\R^d$ and identify $\manifold$ (at least locally) with $U$.
The metric then is the metric induced by the embedding space, $$g_\y(v,w)=g_\y^c(v,w)=(D\varphi(\y)v)^T(D\varphi(\y)w)\,,\qquad\Qenergy\equiv0\,.$$
Similarly, the approximation $\energy$ to the squared Riemannian distance can be taken as the squared Euclidean distance in the embedding space,
\begin{equation*}
\energy:\manifold\times\manifold\to\R\,,\quad
\energy[\y_1,\y_2]=|\varphi(\y_1)-\varphi(\y_2)|^2
\end{equation*}
with $|\y|=\sqrt{\y^T\y}$.

\begin{remark}[Applicability of model analysis]\label{rem:proofChangesFiniteDimM}
Unless $U=\R^d$, the manifold is not admissible in the sense of Definition\,\ref{def:admissibleManifold}.
Nevertheless, the model analysis of the previous sections can be carried over to this setting.
Indeed, $\splineenergy^\sigma$ and $\Splineenergy^\sigma$ are invariant with respect to changing the manifold parameterization and thus are well-defined without specifying a particular chart.
Next, the existence of continuous spline interpolations from Theorem\,\ref{thm:ExistenceContinuous} holds
(the regularity estimate only holds for curves not touching the manifold boundary),
since along a minimizing sequence of curves on the embedded manifold we have uniformly bounded path energy and thus path length.
Hence, by Go\l\c{a}b's theorem a subsequence of those curves converges in the Hausdorff sense to a limit curve.
In the existence and regularity analysis we may therefore restrict ourselves to a chart $U$ of a small neighbourhood around that curve,
on which the admissibility conditions on $g$ hold and thus the proof works without further modification.
Likewise, the existence of discrete spline interpolations from Theorem\,\ref{thm:ExistenceDiscrete} holds.
Indeed, here the same proof can directly be performed in $\R^m$ rather than the parameterization domain.
Finally, the $\Gamma$-convergence from Theorem\,\ref{thm:GammaConvergence} holds true as long as the limit curve does not touch the manifold boundary;
indeed, then again we can work inside a chart of a local neighbourhood around the limit curve.
The corresponding convergence of discrete to continuous interpolations from Theorem\,\ref{thm:equicoercivity} can be transferred via the same trick as for existence
(as in the $\Gamma$-convergence we require, though, that the continuous interpolation has positive distance from the manifold boundary $\partial\manifold$):
We consider a sequence of discrete curves with increasing refinement and with distance $\varepsilon$ from $\partial\manifold$ for some $\varepsilon > 0$.
We then observe that the discrete path energy and thus the path length of the continuous representations is bounded,
so we can extract a subsequence converging in the Hausdorff sense, and from that point on we may restrict inside the proof of Theorem\,\ref{thm:equicoercivity} to a chart around a neighbourhood of the limit curve
(which must be at least distance $\varepsilon$ from $\partial\manifold$).
This way we obtain convergence of discrete to continuous spline interpolations under the constraint of staying $\varepsilon$ away from $\partial\manifold$,
and $\varepsilon\to0$ then yields the desired result.
\end{remark}

In what follows, let us explicitly derive the all terms arising in the nonlinear system of equations 
\begin{align}
\partial_{\y_k}\Splineenergy^K[\y_0, \ldots, \y_K] + \sigma \partial_{\y_k}\Pathenergy^K[\y_0, \ldots, \y_K] =0 \label{eq:ELeqFinite}
\end{align}
for all $k \in \{1, \ldots, K\} \setminus \{Kt_i \,|\, i = 1,\ldots, I\}$ which has to be solved when computing a discrete Riemannian spline.
To this end, we use the adjoint calculus.
The constraint for $\widetilde{\y}_k$ can be expressed via the first order optimality conditions 
\begin{align*}
0&=\partial_2{\energy}[\y_{k-1},\widetilde{\y}_k]+\partial_1{\energy}[\widetilde{\y}_k,\y_{k+1}]
\end{align*}
so that, taking the derivative with respect to $\y_{k-1}$ and $\y_{k+1}$, we have
\begin{align*}
0&=\left(\partial_2^2{\energy}[\y_{k-1},\widetilde{\y}_k]+\partial_1^2{\energy}[\widetilde{\y}_k,\y_{k+1}]\right)\partial_{\y_{k-1}}\widetilde{\y}_k+\partial_1\partial_2{\energy}[\y_{k-1},\widetilde{\y}_k]\,,\\
0&=\left(\partial_2^2{\energy}[\y_{k-1},\widetilde{\y}_k]+\partial_1^2{\energy}[\widetilde{\y}_k,\y_{k+1}]\right)\partial_{\y_{k+1}}\widetilde{\y}_k+\partial_2\partial_1{\energy}[\widetilde{\y}_k,\y_{k+1}]
\end{align*}
(above, $\partial_i\partial_j{\energy}[\y_a,\y_b]$ should be interpreted as matrix with $v^T\partial_i\partial_j{\energy}[\y_a,\y_b]w=\partial_i(\partial_j{\energy}[\y_a,\y_b]v)w$).
Now let $p_k$ denote the solution to
\begin{align*}
0&=\left(\partial_2^2{\energy}[\y_{k-1},\widetilde{\y}_k]+\partial_1^2{\energy}[\widetilde{\y}_k,\y_{k+1}]\right)^T p_k+\partial_2{\energy}[\y_k,\widetilde{\y}_k]^T\,.
\end{align*}
We then obtain 
$$\partial_{\y_{k-1}}{\energy}[\y_k,\widetilde{\y}_k]=\partial_2{\energy}[\y_k,\widetilde{\y}_k]\partial_{\y_{k-1}}\widetilde{\y}_k=(\partial_1\partial_2{\energy}[\y_{k-1},\widetilde{\y}_k]^T p_k)^T=(\partial_2\partial_1{\energy}[\y_{k-1},\widetilde{\y}_k]p_k)^T$$
and analogously $\partial_{\y_{k+1}}{\energy}[\y_k,\widetilde{\y}_k]=(\partial_1\partial_2{\energy}[\widetilde{\y}_k,\y_{k+1}]p_k)^T$. 
Therefore,
\begin{align*}
\partial_{\y_k}\Pathenergy^K[\y_0, \ldots, \y_K]
&=K\left(\partial_2{\energy}[\y_{k-1},\y_k]+\partial_1{\energy}[\y_k,\y_{k+1}]\right)\,,\\
\partial_{\y_k}\Splineenergy^K[\y_0, \ldots, \y_K]
&=\frac{4}{\tau^3}\left(\partial_1{\energy}[\y_k,\widetilde{\y}_k]+(\partial_1\partial_2{\energy}[\widetilde{\y}_{k-1},\y_{k}]p_{k-1})^T+(\partial_2\partial_1{\energy}[\y_{k},\widetilde{\y}_{k+1}]p_{k+1})^T\right)
\end{align*}
for all $k/K\notin\{t_1,\ldots,t_I\}$ (where in case of periodic boundary conditions, $k-1$ and $k+1$ have to be interpreted modulo $K$).
In our implementation we solve \eqref{eq:ELeqFinite} using a Quasi-Newton method.
For the sake of completeness we also list here the derivatives of ${\energy}$,
\begin{align*}
{\energy}[\y_1,\y_2]&=|\varphi(\y_1)-\varphi(\y_2)|^2\,,\\
\partial_1{\energy}[\y_1,\y_2]&=\partial_2{\energy}[\y_2,\y_1]=2(\varphi(\y_1)-\varphi(\y_2))^TD\varphi(\y_1)\,,\\
\partial_1^2{\energy}[\y_1,\y_2]&=\partial_2^2{\energy}[\y_2,\y_1]=2D\varphi(\y_1)^TD\varphi(\y_1)+2(\varphi(\y_1)-\varphi(\y_2))^TD^2\varphi(\y_1)\,,\\
\partial_1\partial_2{\energy}[\y_1,\y_2]&=\partial_2\partial_1{\energy}[\y_2,\y_1]=-2D\varphi(\y_2)^TD\varphi(\y_1)\,.
\end{align*}
Example curves are shown in Figure\,\ref{fig:finiteDim}.

\begin{figure}
\centering
\setlength\unitlength{.2\linewidth}
\includegraphics[height=\unitlength,trim=120 70 120 70,clip]{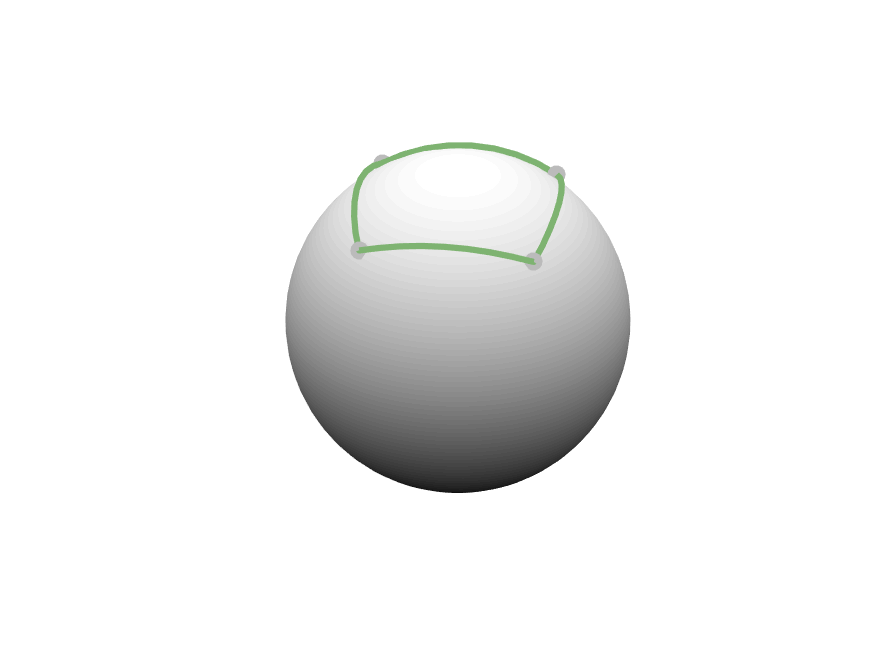}
\includegraphics[height=\unitlength,trim=80 70 80 70,clip]{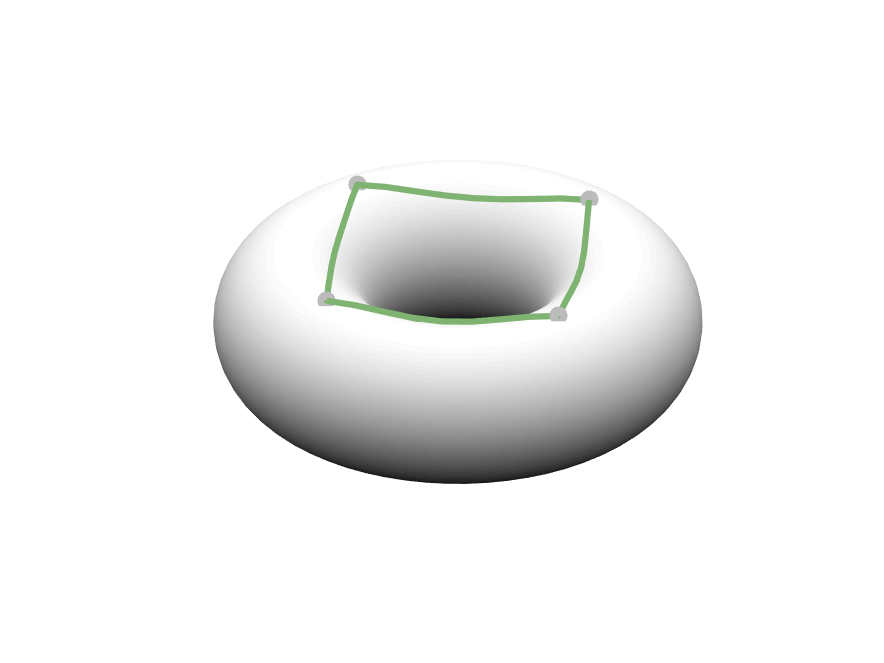}
\includegraphics[height=\unitlength,trim=80 70 80 70,clip]{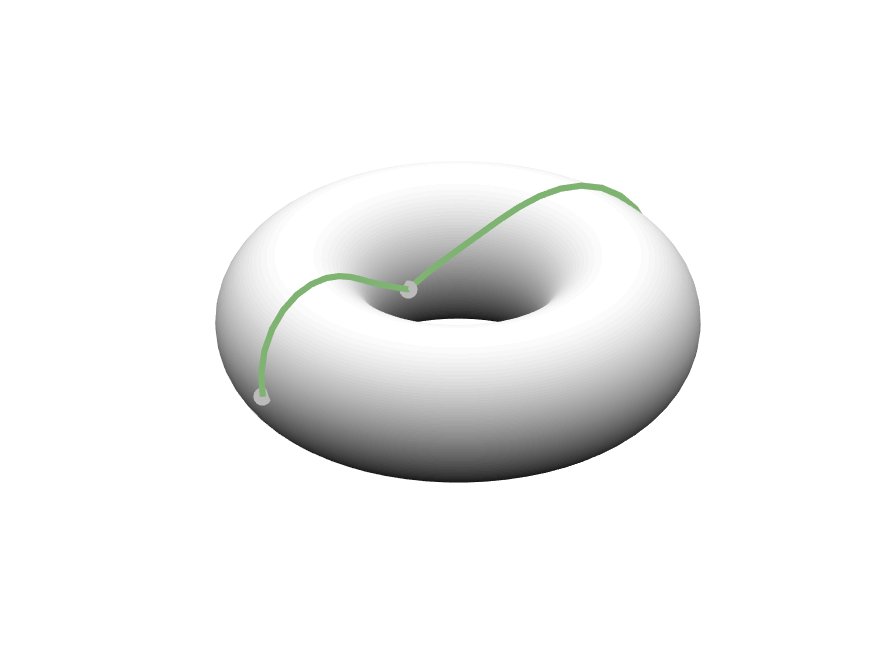}\\
\includegraphics[height=\unitlength,trim=120 70 120 70,clip]{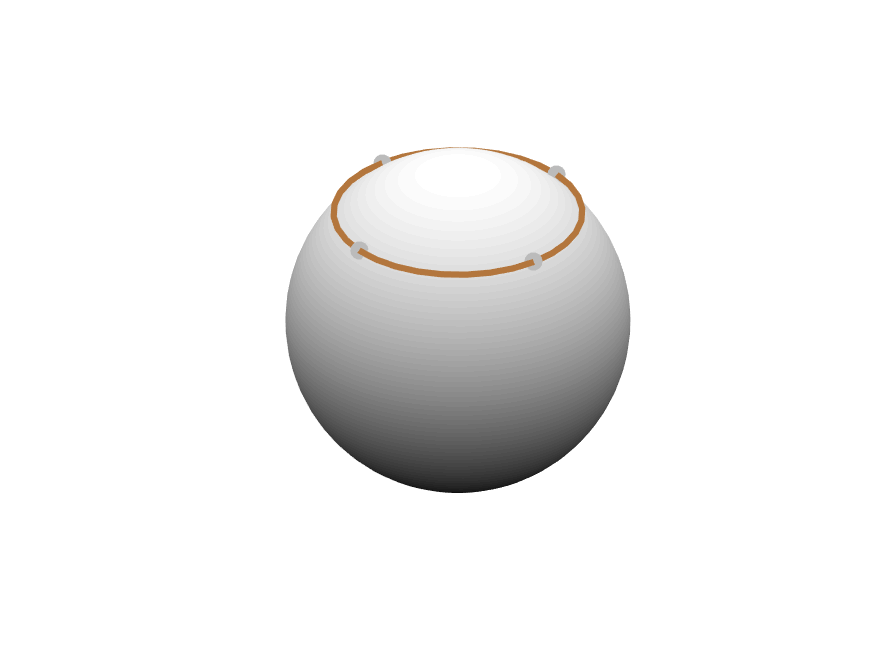}
\includegraphics[height=\unitlength,trim=80 70 80 70,clip]{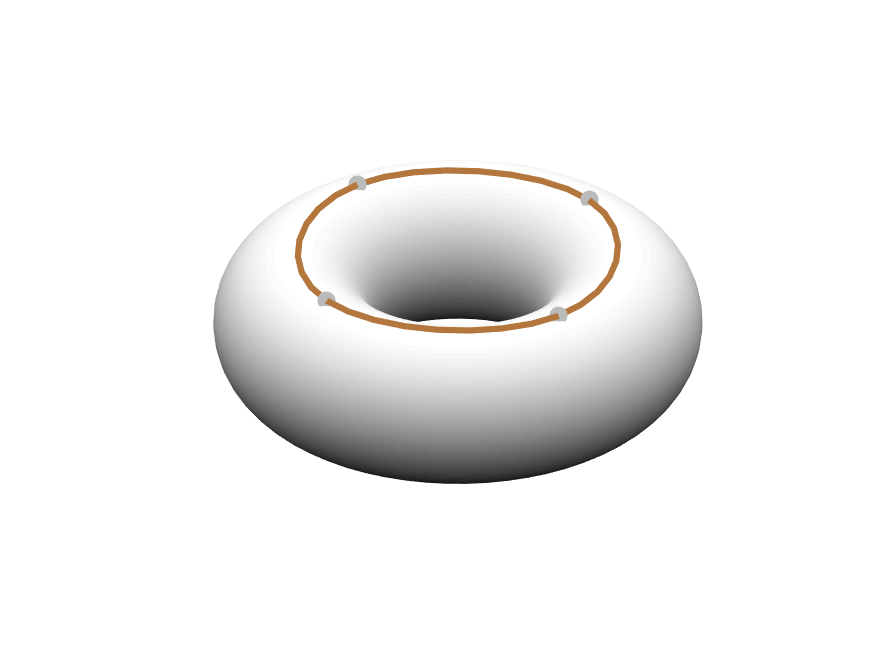}
\includegraphics[height=\unitlength,trim=80 70 80 70,clip]{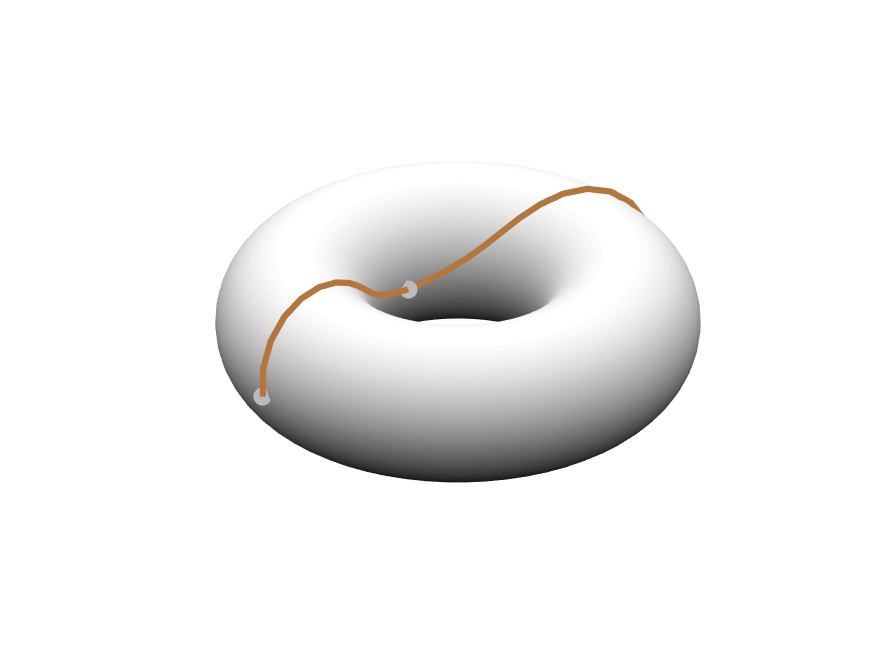}
\caption{Interpolation of given points on a sphere and a torus using a piecewise geodesic (top) and a Riemannian spline (bottom) with $\sigma=0$.
The discretization uses $K=32$ steps along the discrete path.}
\label{fig:finiteDim}
\end{figure}

\subsection{Discrete Shells} \label{sec:shell}
\newcommand{\nSet}{{\mathcal{V}}}
\newcommand{\eSet}{{\mathcal{E}}}
\newcommand{\tSet}{{\mathcal{T}}}
\newcommand{\lenE}{\energy_L}
\newcommand{\dihE}{\energy_\Theta}
\newcommand{\areaE}{\energy_A}

Here we consider the space of discrete shells \cite{GrHiDeSc03}, which is a shape space particularly useful for computer graphics applications.
This shape space is physically motivated (cf.  \cite{HeRuSc14}). A shell is a thin material layer around a mid surface embedded in $\R^3$.
A metric on the space of shells reflects the energy that is dissipated under a plastic deformation of the material layer of the shell.
For thin material layers this physical energy dissipation is composed of an amount due to in-plane stretching of the membrane layer as well as an amount due to bending.
Now \emph{discrete shells} are the discrete counterparts of the shell mid surfaces and consist of triangulated surfaces of fixed connectivity.
The space of discrete shells forms a finite-dimensional manifold which readily fits into the framework introduced before, while the situation is more complicated for continuous shells.
The next section will consider the one-dimensional cousin of continuous shells, the space of rods, 
for which it is a little easier to fit it into our framework.

In words, the \emph{space of discrete shells} is given by all shape regular triangle meshes of same connectivity modulo rigid body motions.
Indeed, given a reference triangulation $\tSet$, a set of nodes $\nSet$ and a set of edges $\eSet$ we define
\begin{multline*}
\manifold
=\{\y:\nSet\to\R^3\,|\,\text{any triangle }\y(T),\,T\in\tSet,\text{ contains a disc of radius }\rho \text{ and has diameter }\leq h,\\
\,\y(v_1)=v_1,\,\y(e_1)||e_1,\,\y(e_1)\wedge\y(e_2)||e_1\wedge e_2\}\,.
\end{multline*}
for fixed $0 < \rho < h$ and $v_1$ one vertex of the mesh, $e||d$ if $e=\alpha d$ for some $\alpha>0$, and $e_1, e_2\in\eSet$ two edges incident to $v_1$.
Above we used that a discrete shell $\y:\nSet\to\R^3$ induces a mapping on edges $e\in\eSet$ and triangles $T\in\tSet$.
The last three conditions in the definition of $\manifold$ just fix a particular location and orientation of the shape.
The \emph{dissipation} between two discrete shells in $\manifold$ splits into a membrane distortion and a bending contribution and is defined as in \cite{HeRuSc14}
\begin{equation*}
\energy[\y,\widetilde{\y}]=\energy^c[\y,\widetilde{\y}]=\zeta\, \energy_{\mathrm{mem}}[\y, \widetilde{\y} ]  + \eta \, \dihE[\y,\widetilde{\y}]
\end{equation*}
with weights $\zeta,\eta>0$.
Here the membrane energy $\energy_{\mathrm{mem}}$ is given by
$$
\energy_{\mathrm{mem}}[\y, \widetilde{\y}] =  \sum_{T \in \tSet} \mathrm{area}(\y(T)) W_{\mathrm{mem}}( \revision{D}(\tilde \y\circ\y^{-1}) ) 
$$
\revision{with the matrix-valued Jacobian $D(\tilde \y\circ\y^{-1})$} and the energy density
\begin{align}\label{eq:MembraneEnergyDensity}
W_{\mathrm{mem}}( A ) =   \dfrac{\mu}{2}\mathrm{tr}\sqrt{A^*A} + \dfrac{\lambda}{4}\det \sqrt{A^*A} - \frac{2\mu+\lambda}{4} \log \det \sqrt{A^*A} \!- \mu - \dfrac{\lambda}{4}\,,
\end{align}
where $\lambda$ and $\mu$ are the Lam\'e constants of a Newtonian constitutive law for the energy dissipation, $A^*$ denotes the adjoint operator of $A$, and $\mathrm{tr} B$ and $\det B$ denote the trace and determinant of $B$ as an endomorphism on the tangent bundle of the triangular shell surface $\y$. Notice that $\det\sqrt{A^*A}$ describes area distortion, while $\mathrm{tr}\sqrt{A^*A}$ measures length distortion.
Obviously the polyconvex function $W_{\mathrm{mem}}( A )$ induces a rigid body motion invariant functional $\energy$, and the identity is its unique minimizer.
The $ \log \det \sqrt{A^*A}$ term penalizes material compression, which in the discrete setting prevents degeneration of triangles.
The bending energy is defined by 
\begin{equation*}
\dihE[\y,\widetilde{\y}] = \sum_{e \in \eSet} l_e[\y]^2 \frac{ (\theta_e[\y] - \theta_e[\widetilde{\y}])^2}{d_e[\y]}\,,
\end{equation*}
where $l_e[\y]$ is the length of the edge $\y(e)$, $\theta_e[\y]$ the dihedral angle between the triangles adjacent to $\y(e)$, and $3d_e[\y]$ the area of those triangles.

The \emph{metric} on the space of discrete shells is defined for $v,w:\nSet\to\R^3$ as the second derivative of the dissipation in directions $v$ and $w$,
\begin{equation*}
g_\y(v,w)=\tfrac12\partial_2^2\energy[\y,\y](v,w)\,.
\end{equation*}

Any discrete shell $\y\in\manifold$ as well as tangent vectors $v\in T_\y\manifold$ can obviously be identified with the corresponding vector in $\R^{3N}$ of nodal values, where $N$ denotes the number of vertices in $\nSet$.
Thus we may choose $\V=\Y=\R^{3N}$.
Note that $\manifold\subset\R^{3N}$ only has a piecewise smooth boundary, however, the framework of admissible manifolds and our previous analysis may readily be extended to include also this case.
That $g_\y$ is positive definite and thus represents a metric has been shown in \cite{HeRuSc14};
that it is admissible in the sense of Definition\,\ref{def:admissibleMetric} follows from the smoothness of $g$ (it is infinitely differentiable) and the compactness of $\manifold$.
Admissibility for $\energy$ follows from the continuity of $\energy$ and \cite[Lemma\,4.6]{RuWi12b}.
Thus we have existence of continuous and discrete spline interpolations, and the discrete ones converge against the continuous ones.
 \begin{figure}[h]
 \centering
 \setlength{\unitlength}{.06\linewidth}%
\begin{picture}(13,6)
\put(0,4.2){\resizebox{\unitlength}{!}{\includegraphics{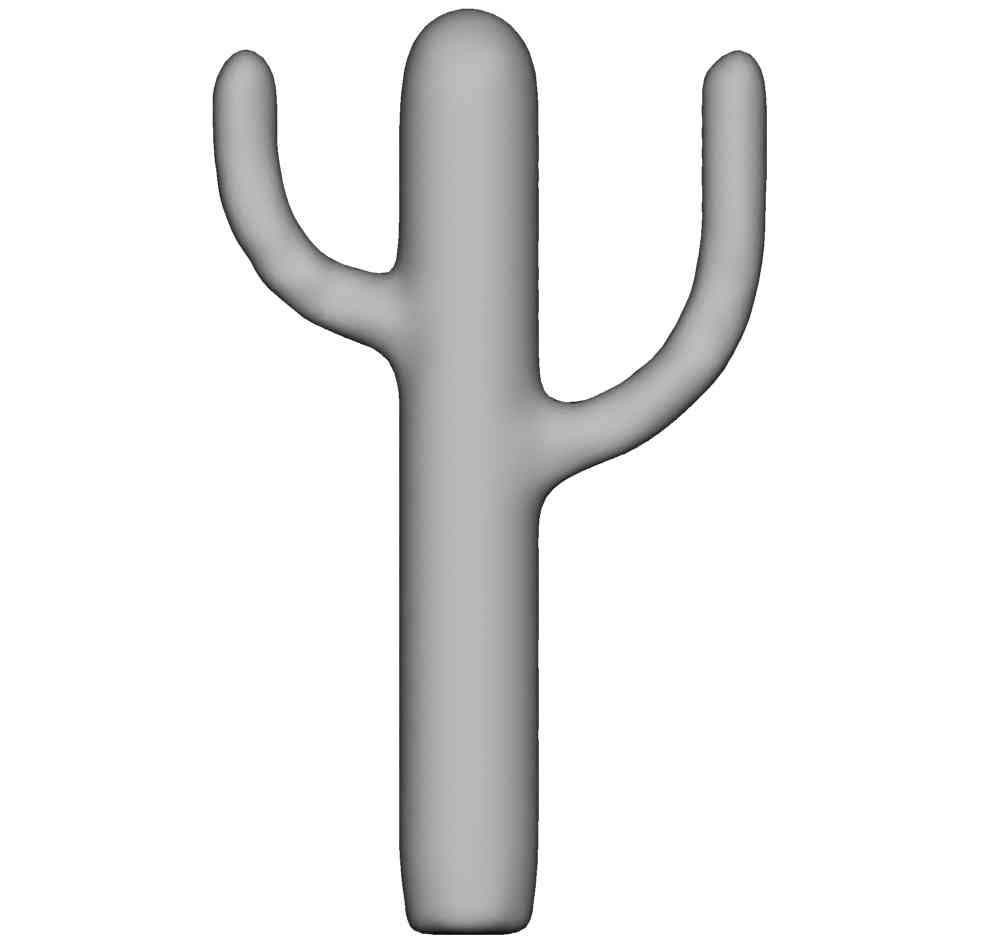}}}
\put(5,4.2){\resizebox{\unitlength}{!}{\includegraphics{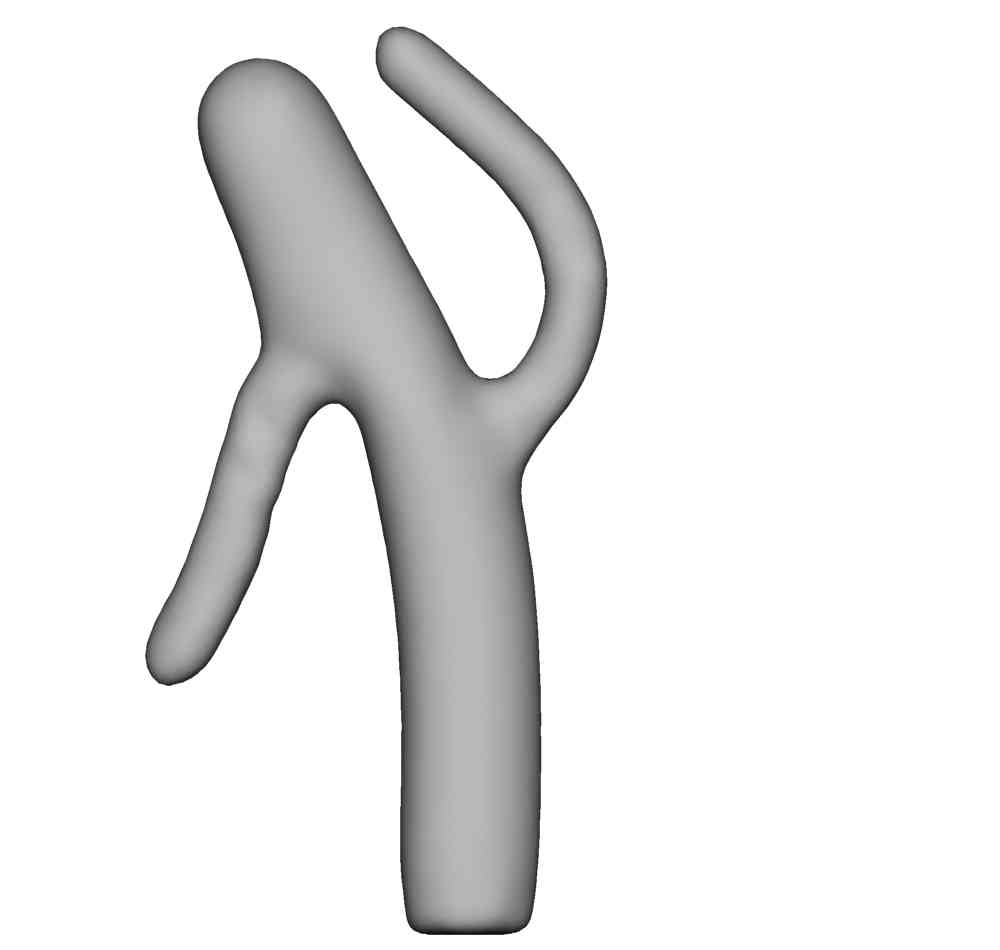}}}
\put(10,4.2){\resizebox{\unitlength}{!}{\includegraphics{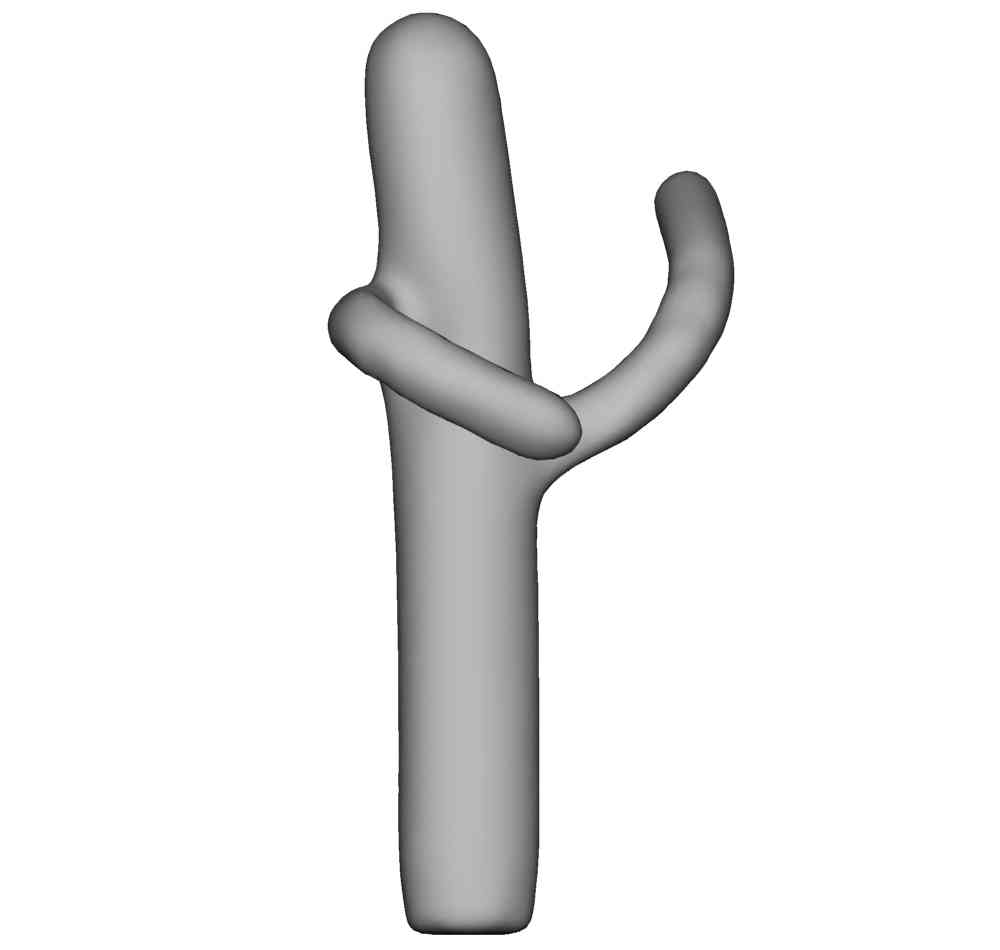}}}
\put(2,3.0){\resizebox{\unitlength}{!}{\includegraphics{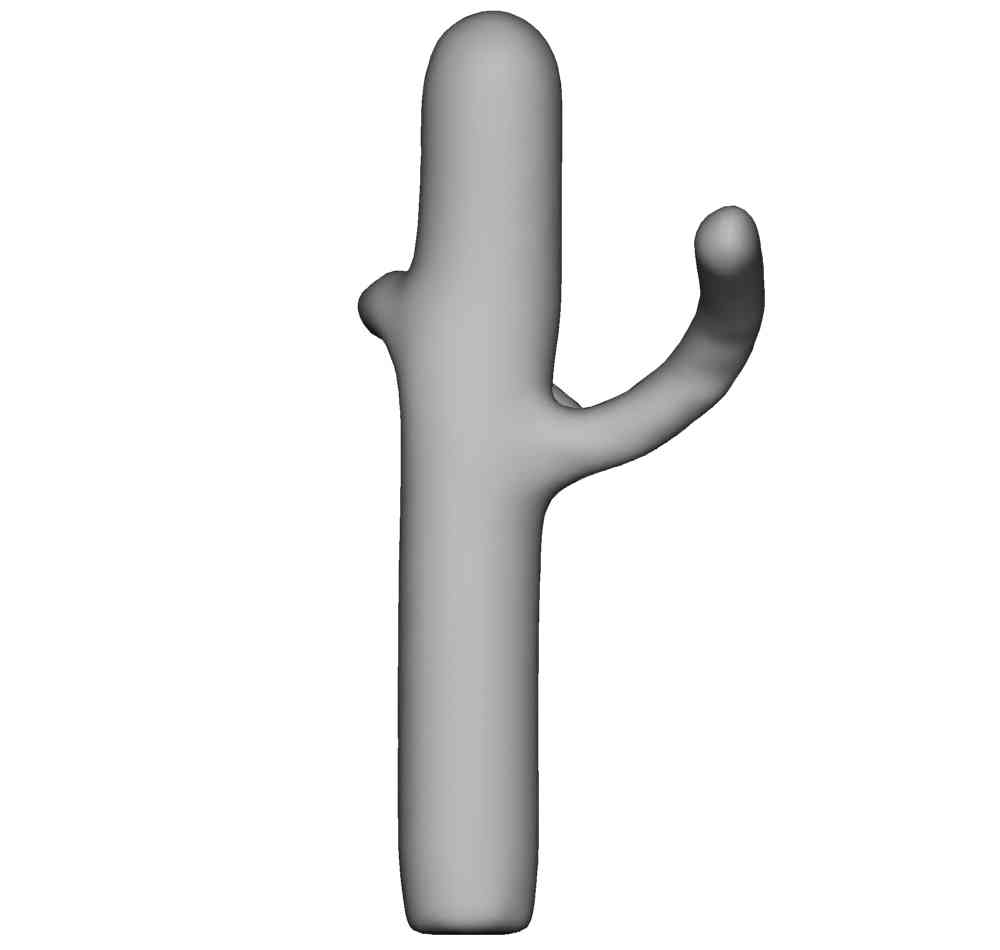}}}
\put(7,3.0){\resizebox{\unitlength}{!}{\includegraphics{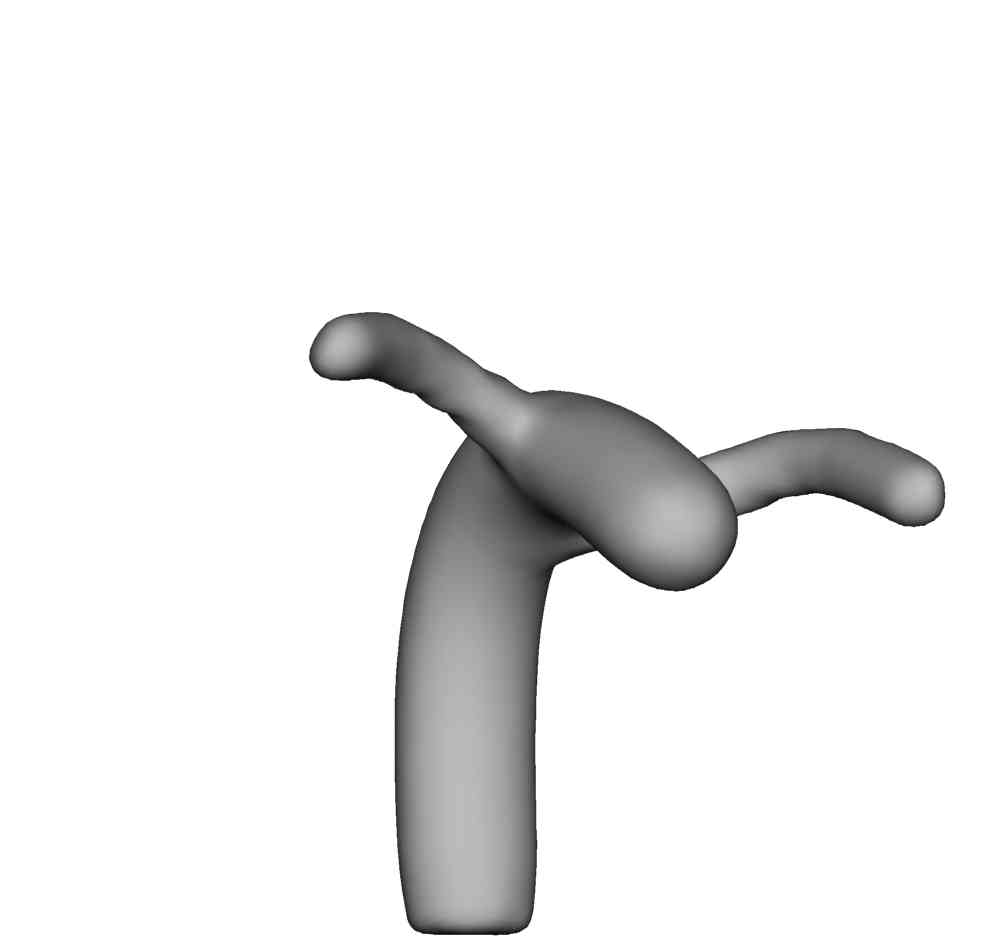}}}
\put(12,3.0){\resizebox{\unitlength}{!}{\includegraphics{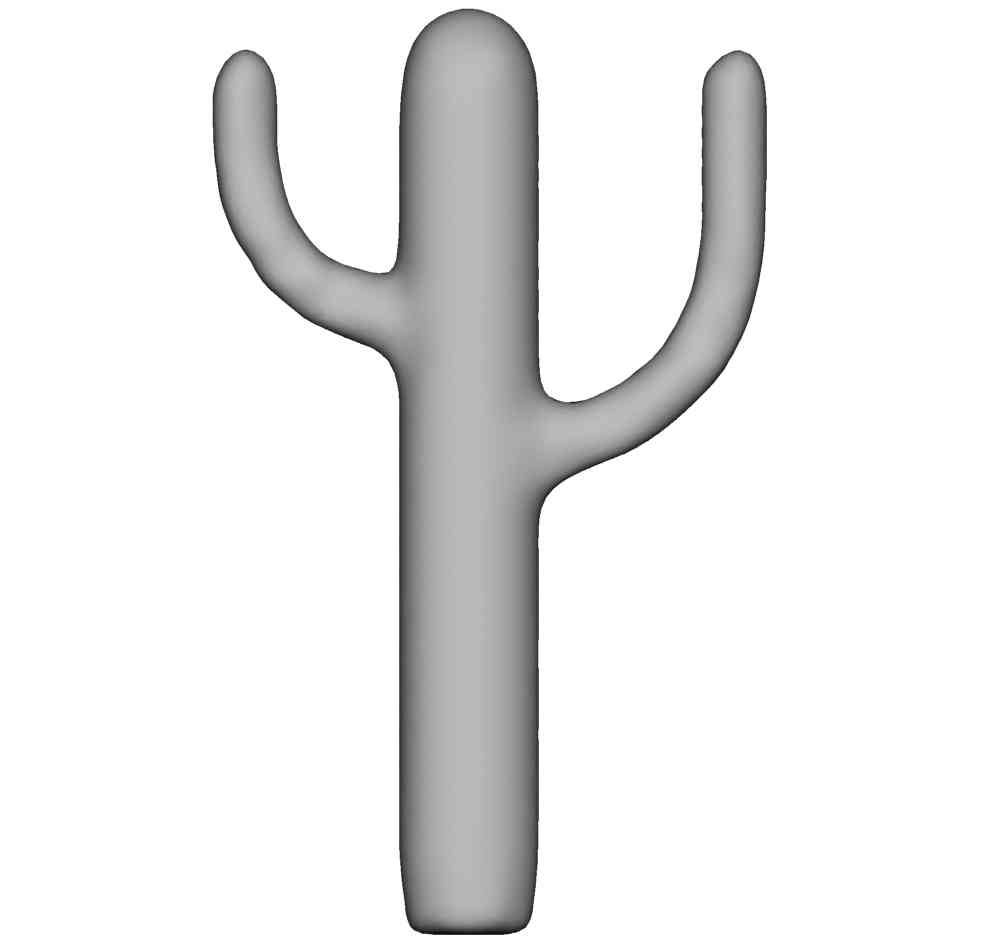}}}
\put(1,4.2){\resizebox{\unitlength}{!}{\includegraphics{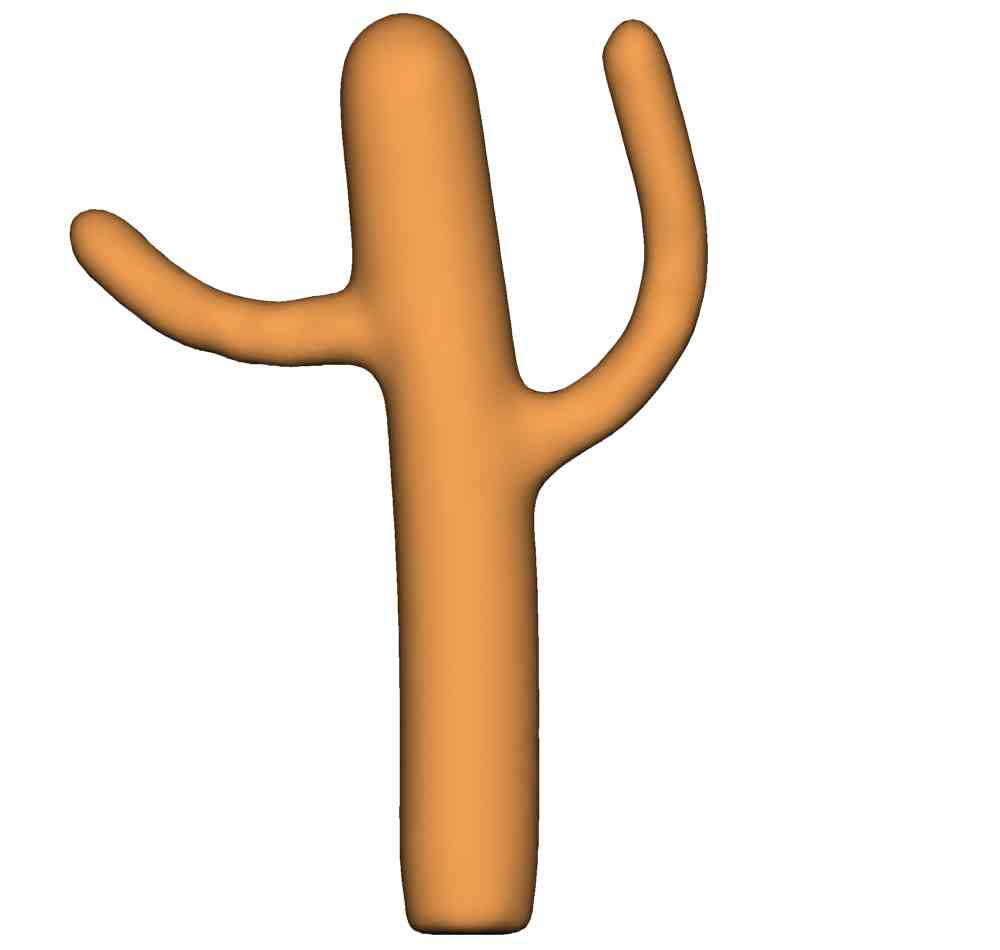}}}
\put(2,4.2){\resizebox{\unitlength}{!}{\includegraphics{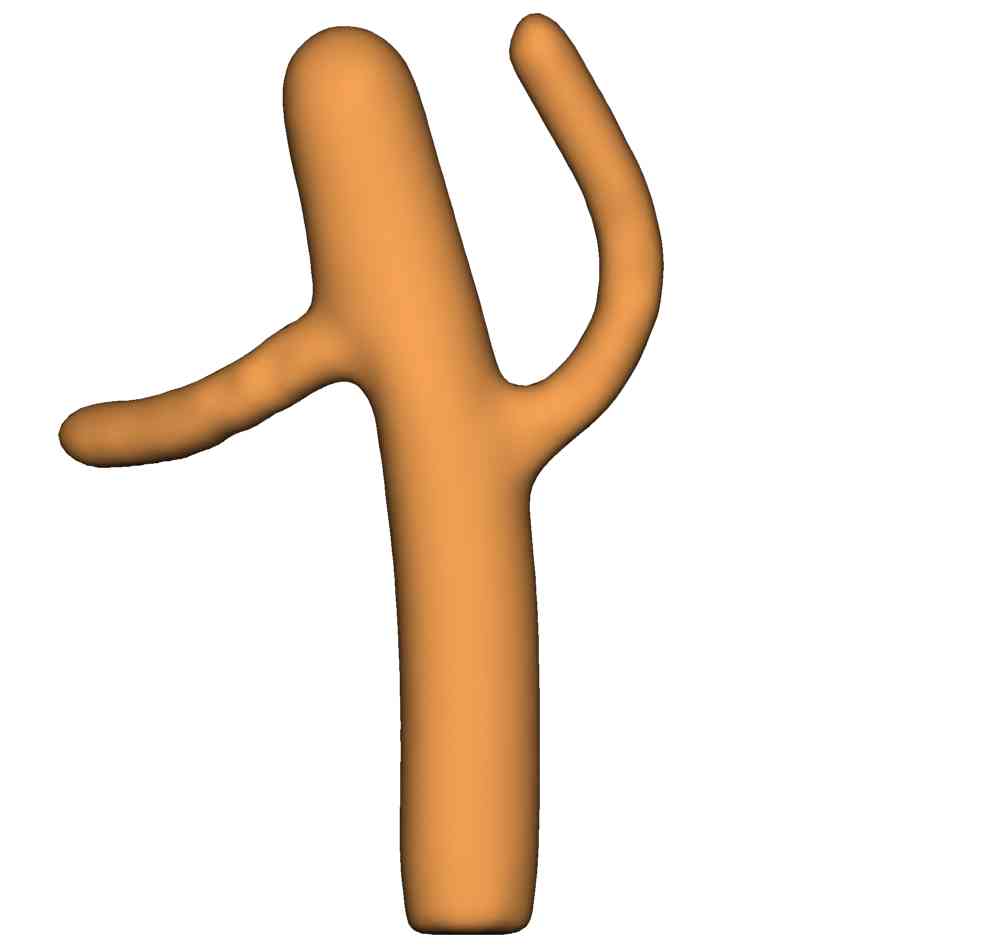}}}
\put(3,4.2){\resizebox{\unitlength}{!}{\includegraphics{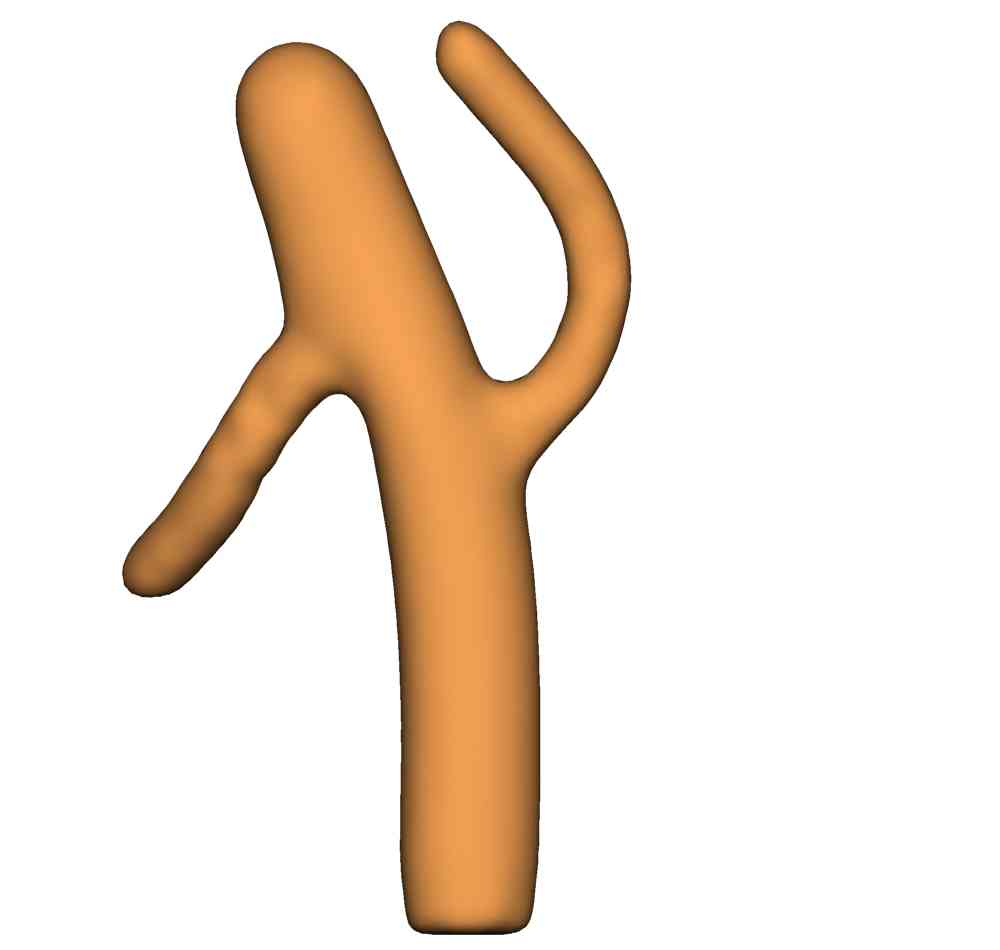}}}
\put(4,4.2){\resizebox{\unitlength}{!}{\includegraphics{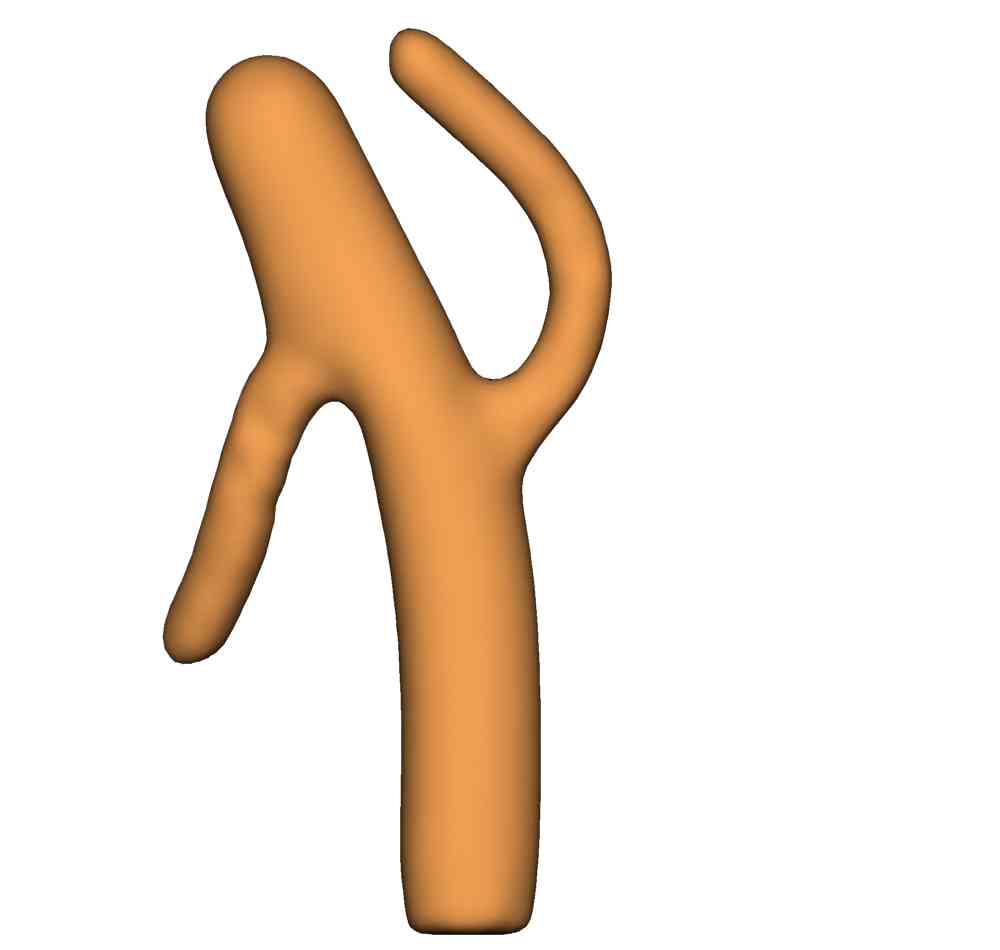}}}
\put(6,4.2){\resizebox{\unitlength}{!}{\includegraphics{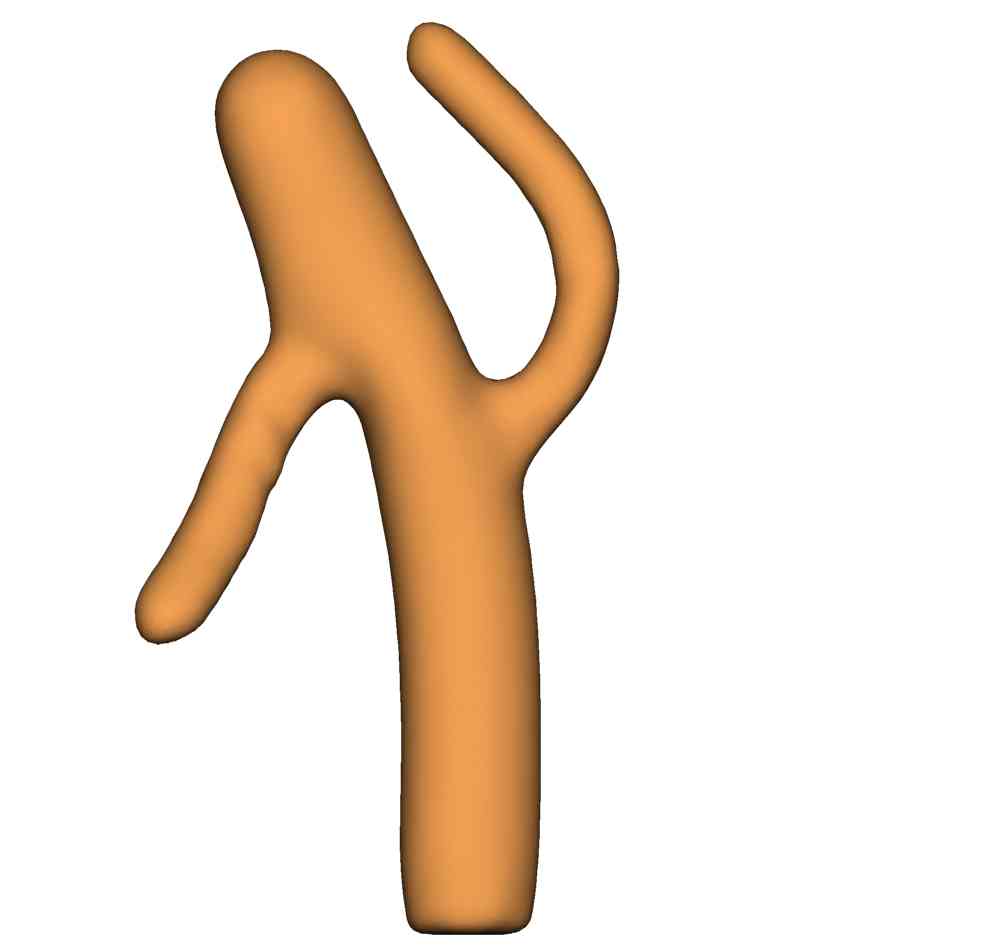}}}
\put(7,4.2){\resizebox{\unitlength}{!}{\includegraphics{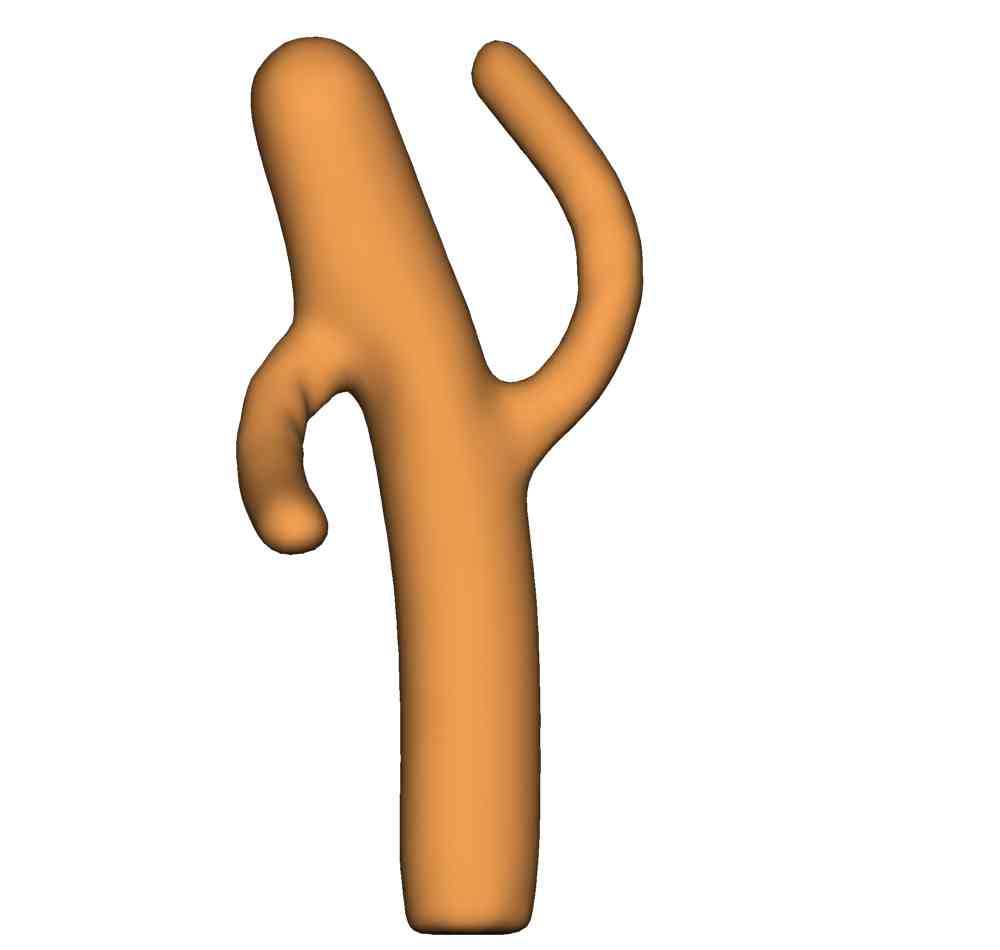}}}
\put(8,4.2){\resizebox{\unitlength}{!}{\includegraphics{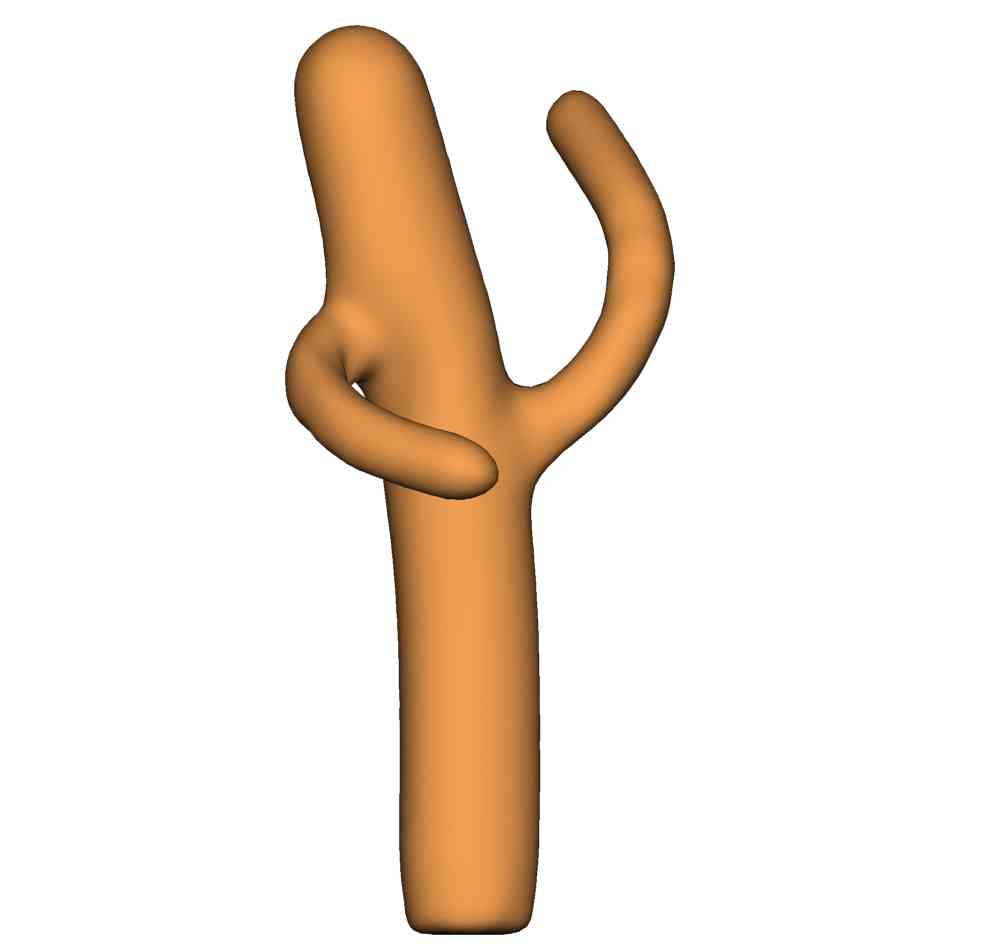}}}
\put(9,4.2){\resizebox{\unitlength}{!}{\includegraphics{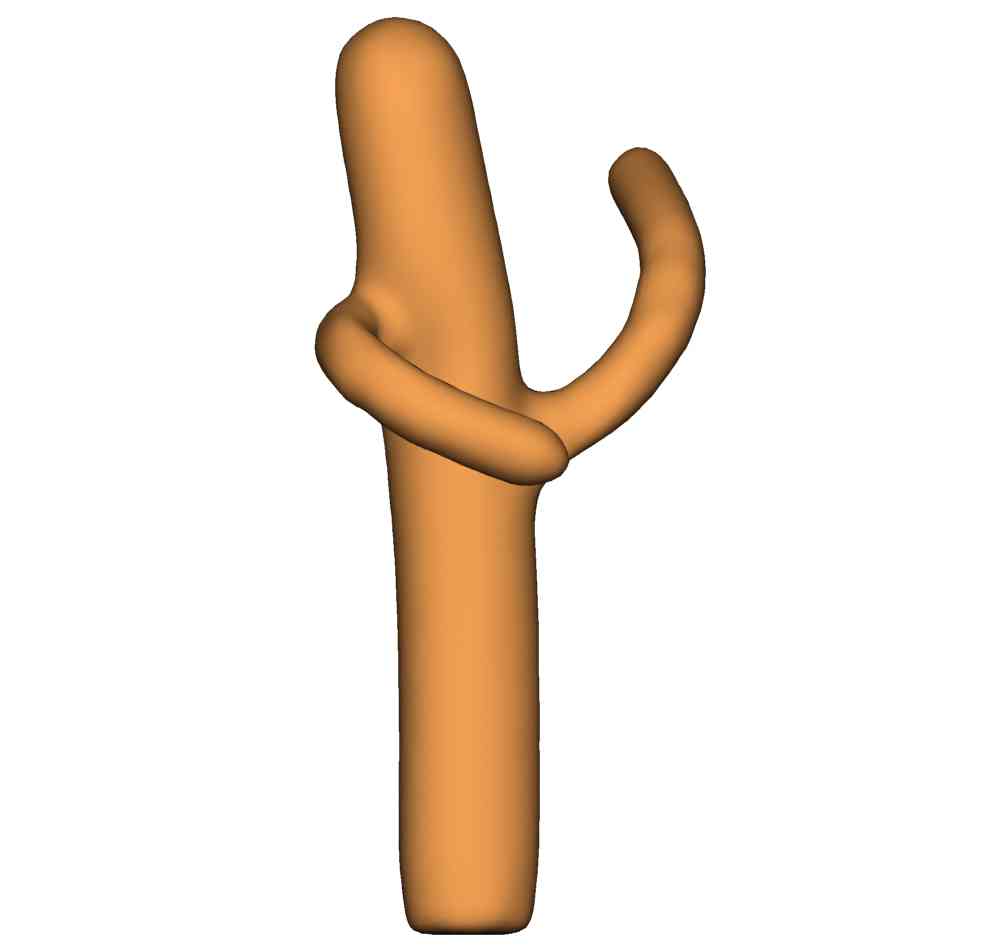}}}
\put(11,4.2){\resizebox{\unitlength}{!}{\includegraphics{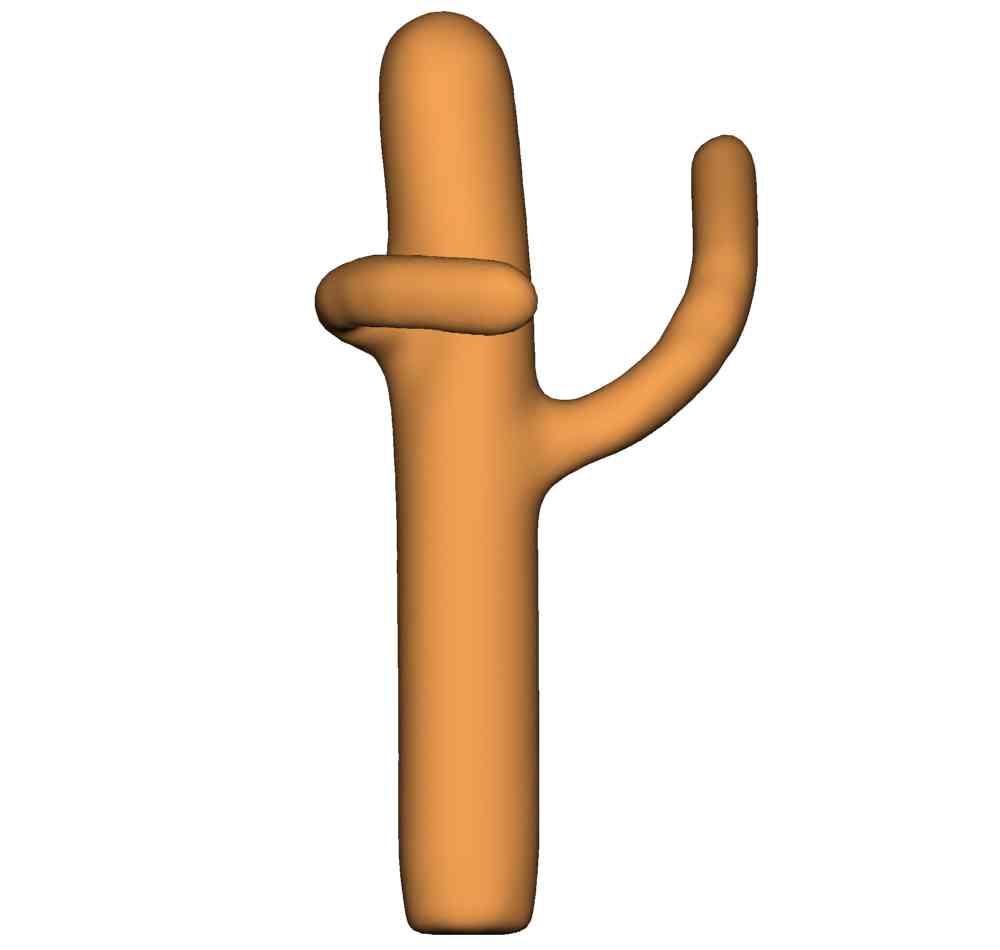}}}
\put(12,4.2){\resizebox{\unitlength}{!}{\includegraphics{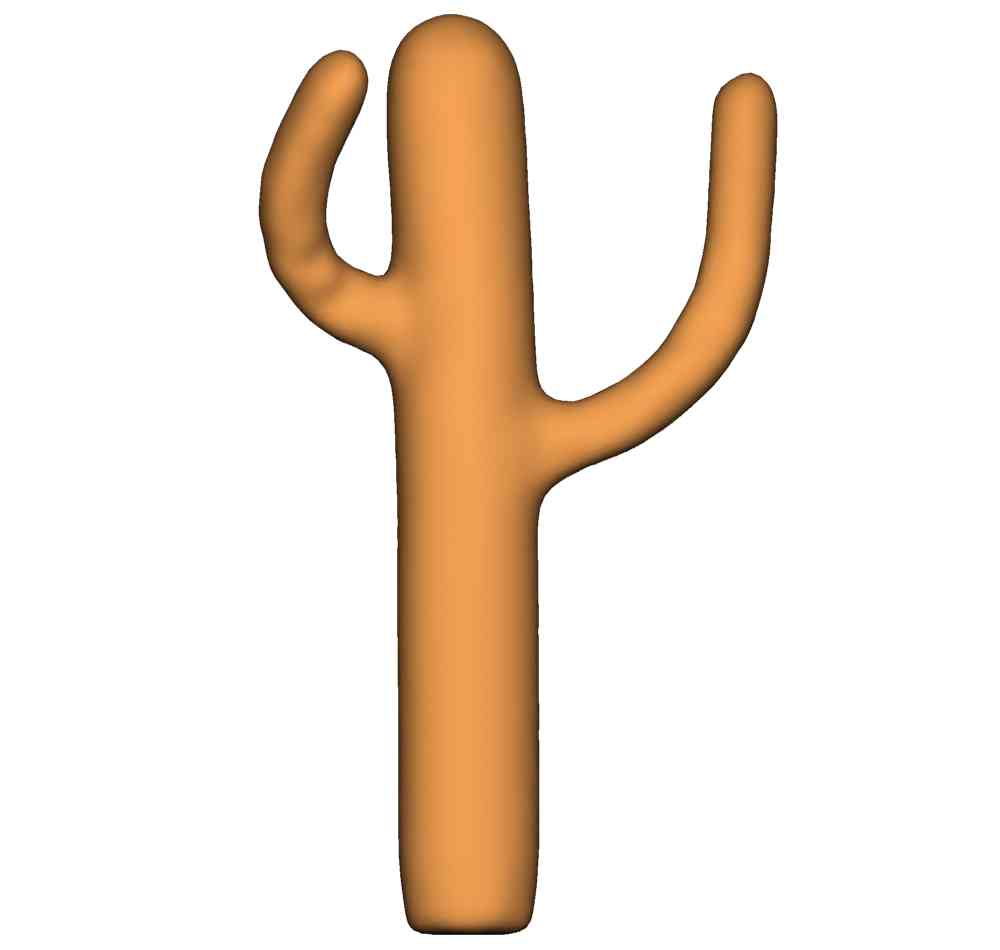}}}
\put(0,3.0){\resizebox{\unitlength}{!}{\includegraphics{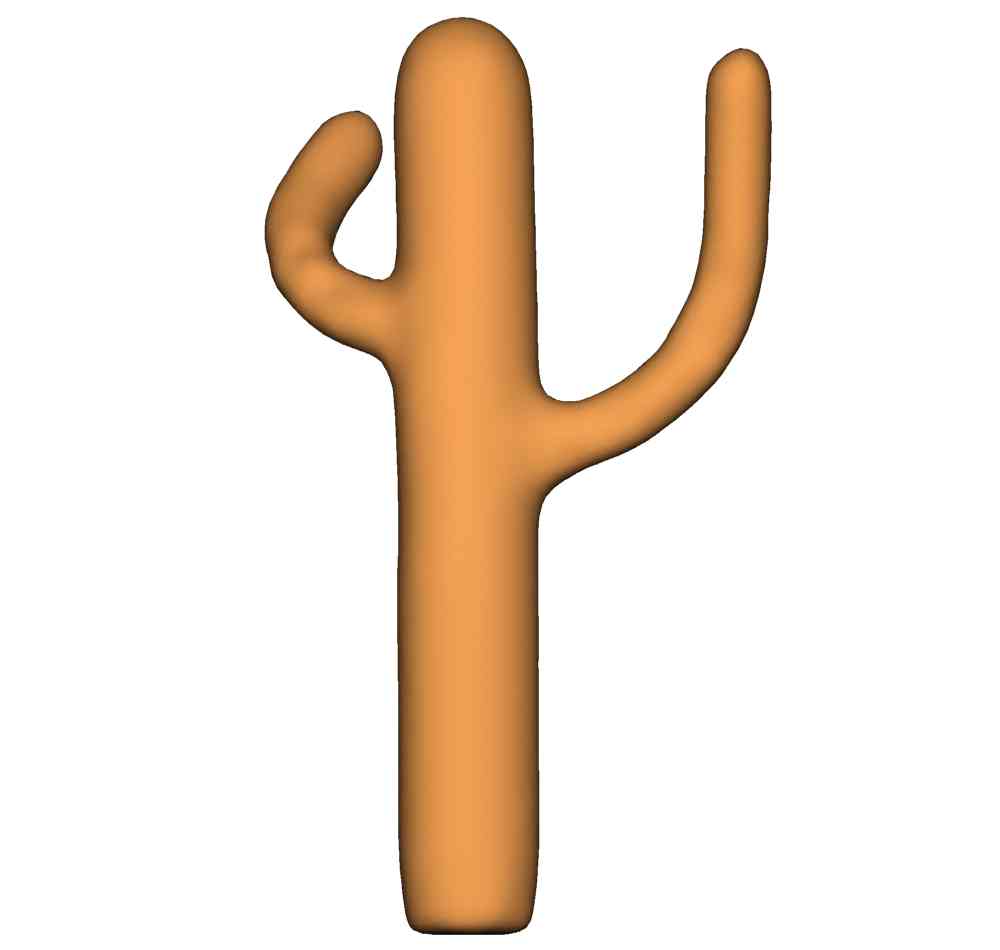}}}
\put(1,3.0){\resizebox{\unitlength}{!}{\includegraphics{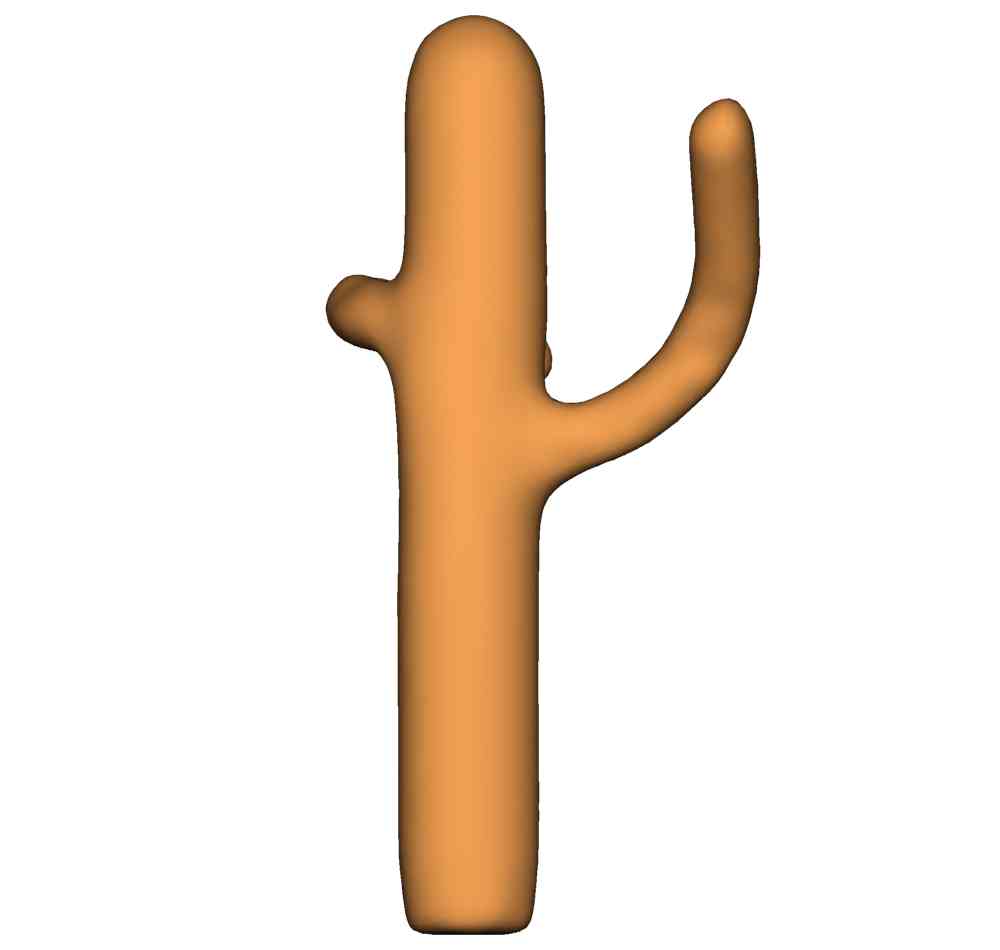}}}
\put(3,3.0){\resizebox{\unitlength}{!}{\includegraphics{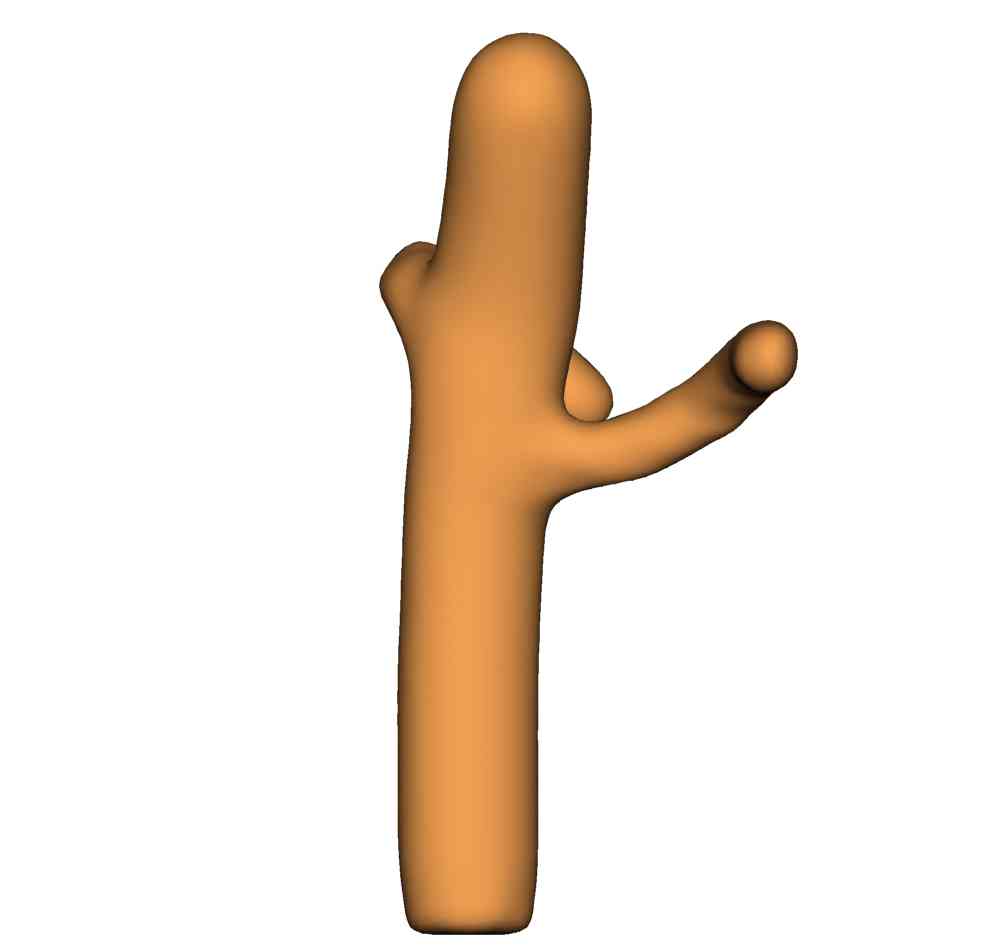}}}
\put(4,3.0){\resizebox{\unitlength}{!}{\includegraphics{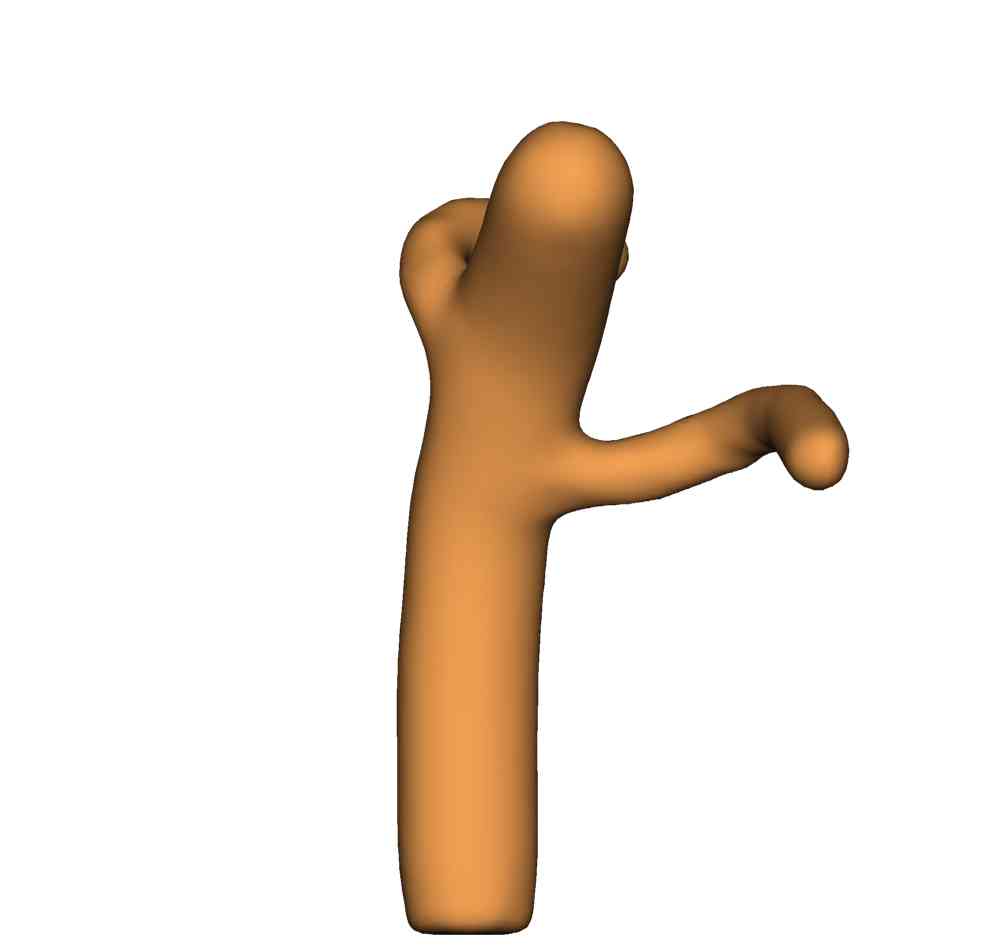}}}
\put(5,3.0){\resizebox{\unitlength}{!}{\includegraphics{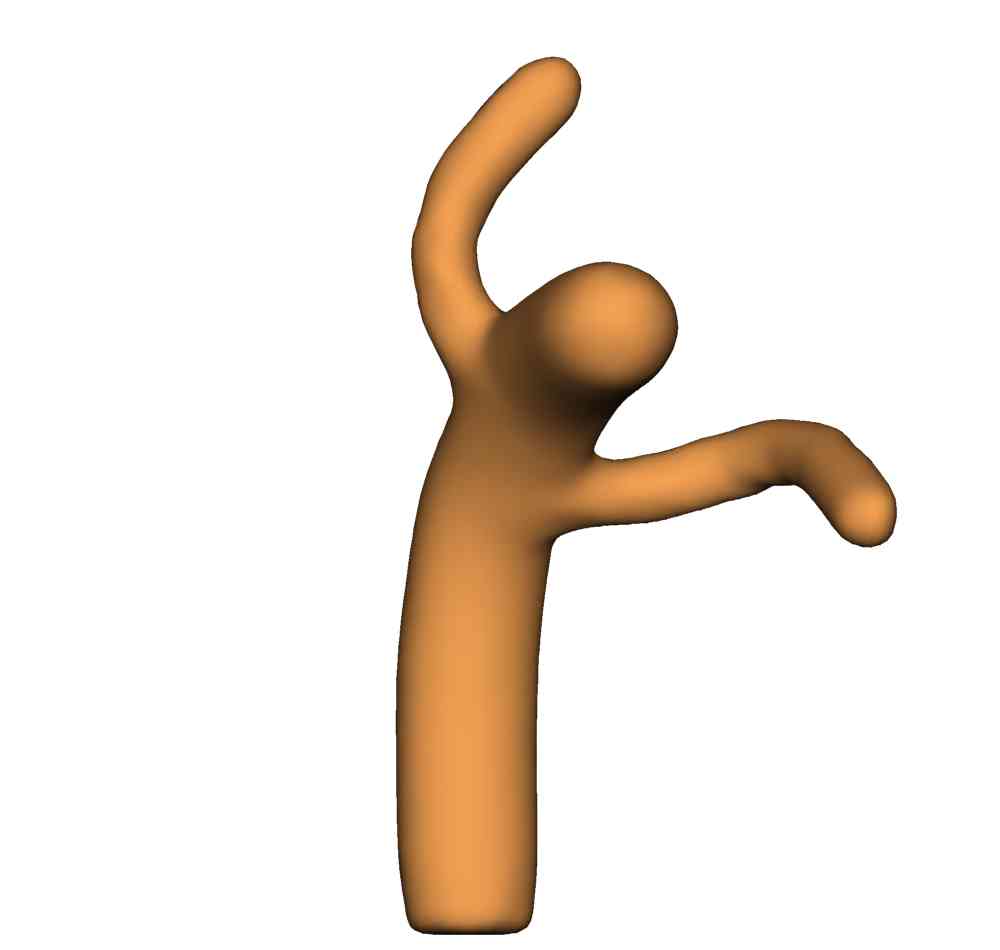}}}
\put(6,3.0){\resizebox{\unitlength}{!}{\includegraphics{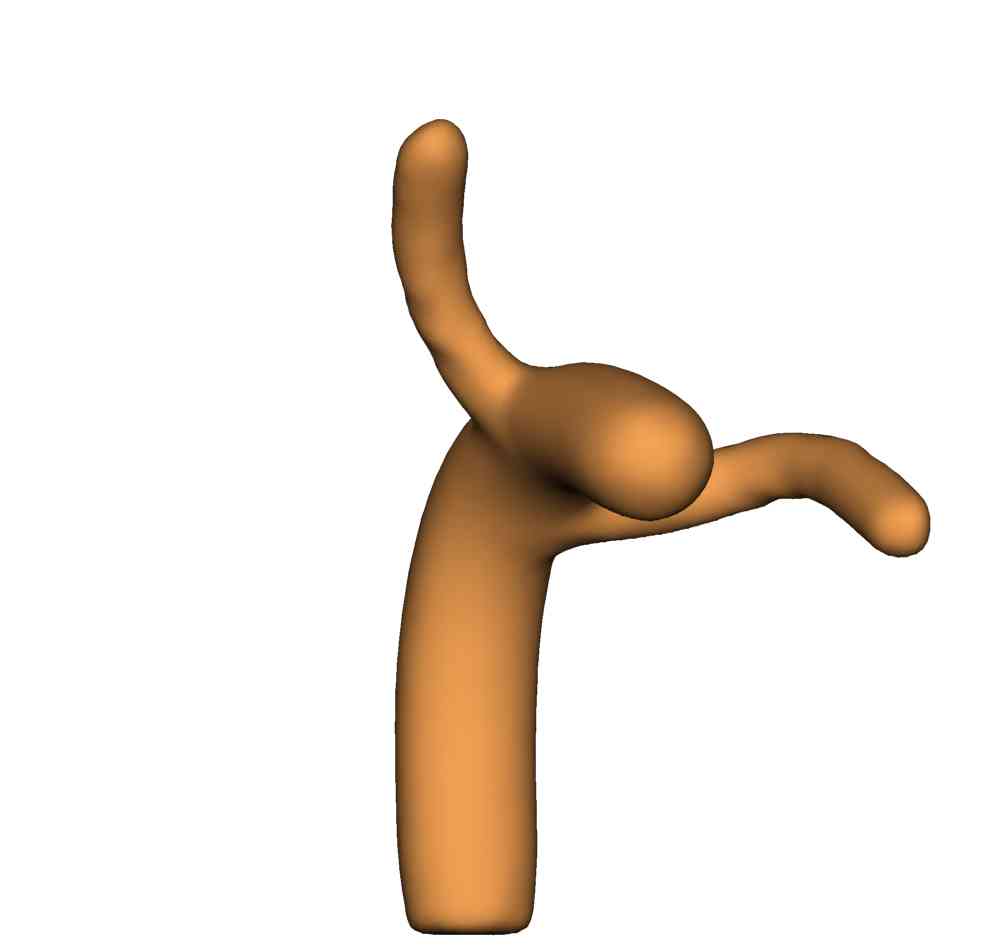}}}
\put(8,3.0){\resizebox{\unitlength}{!}{\includegraphics{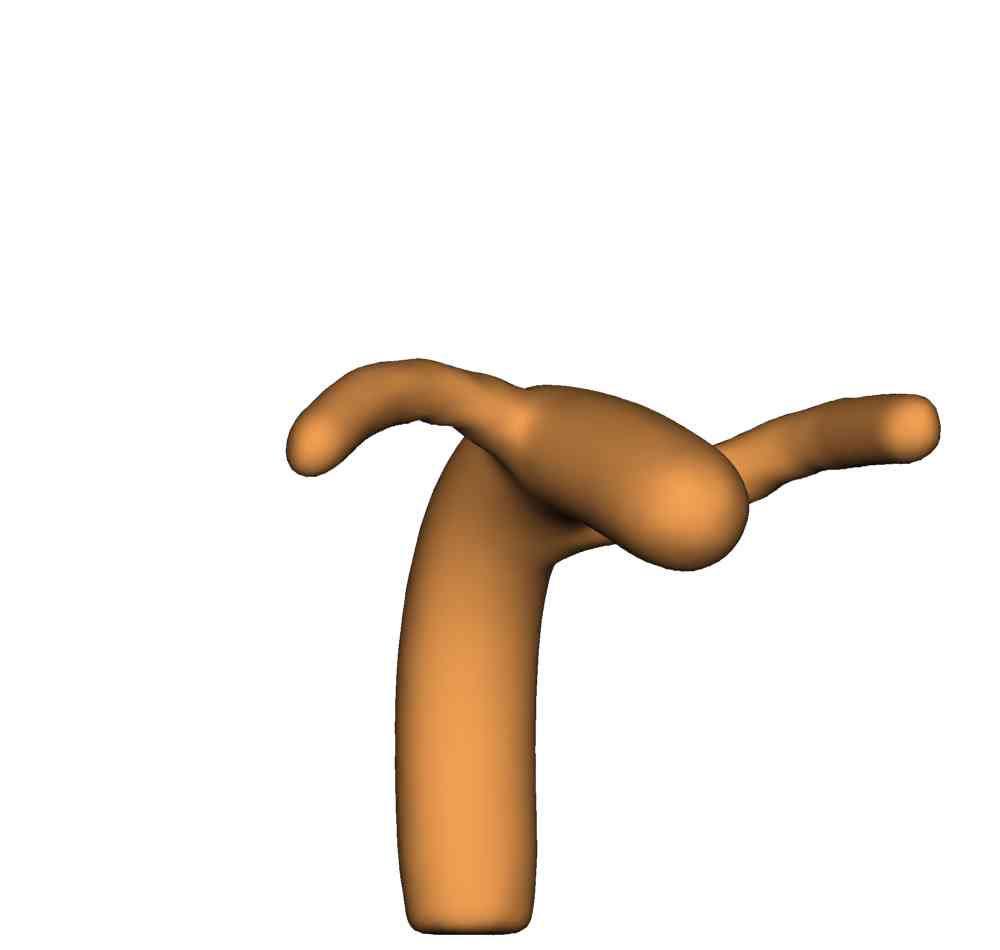}}}
\put(9,3.0){\resizebox{\unitlength}{!}{\includegraphics{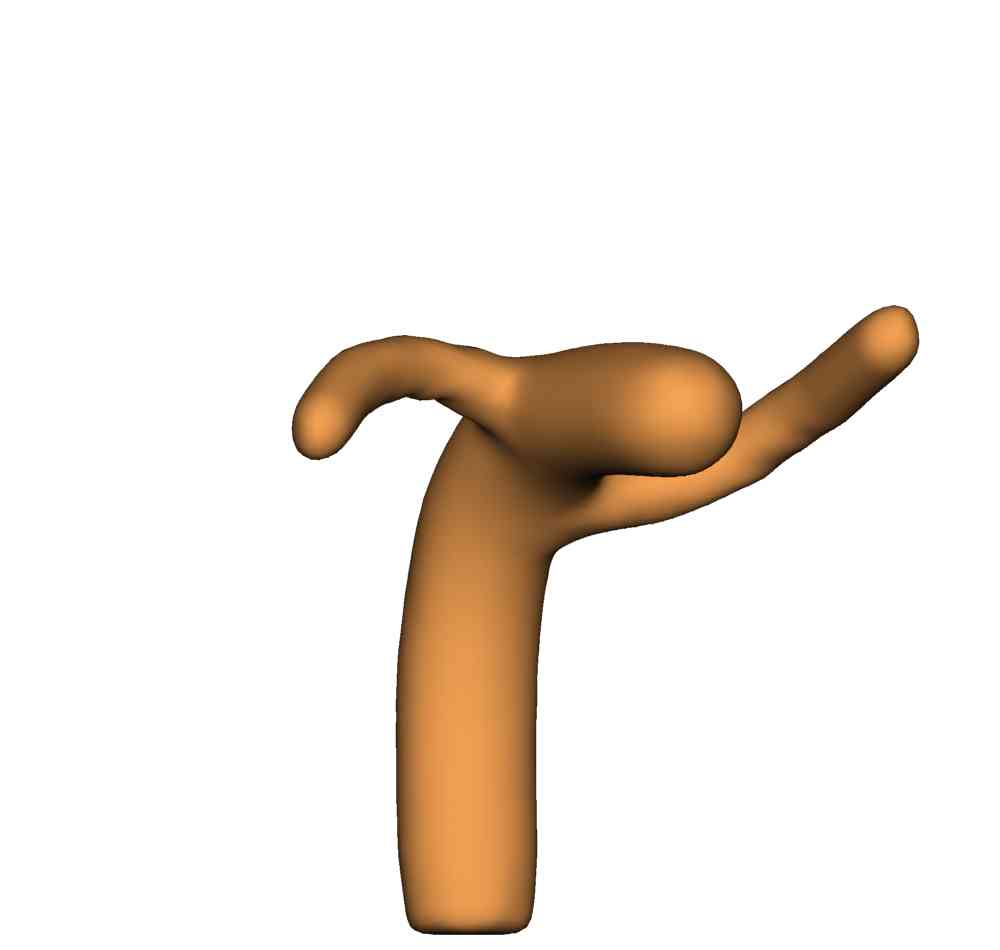}}}
\put(10,3.0){\resizebox{\unitlength}{!}{\includegraphics{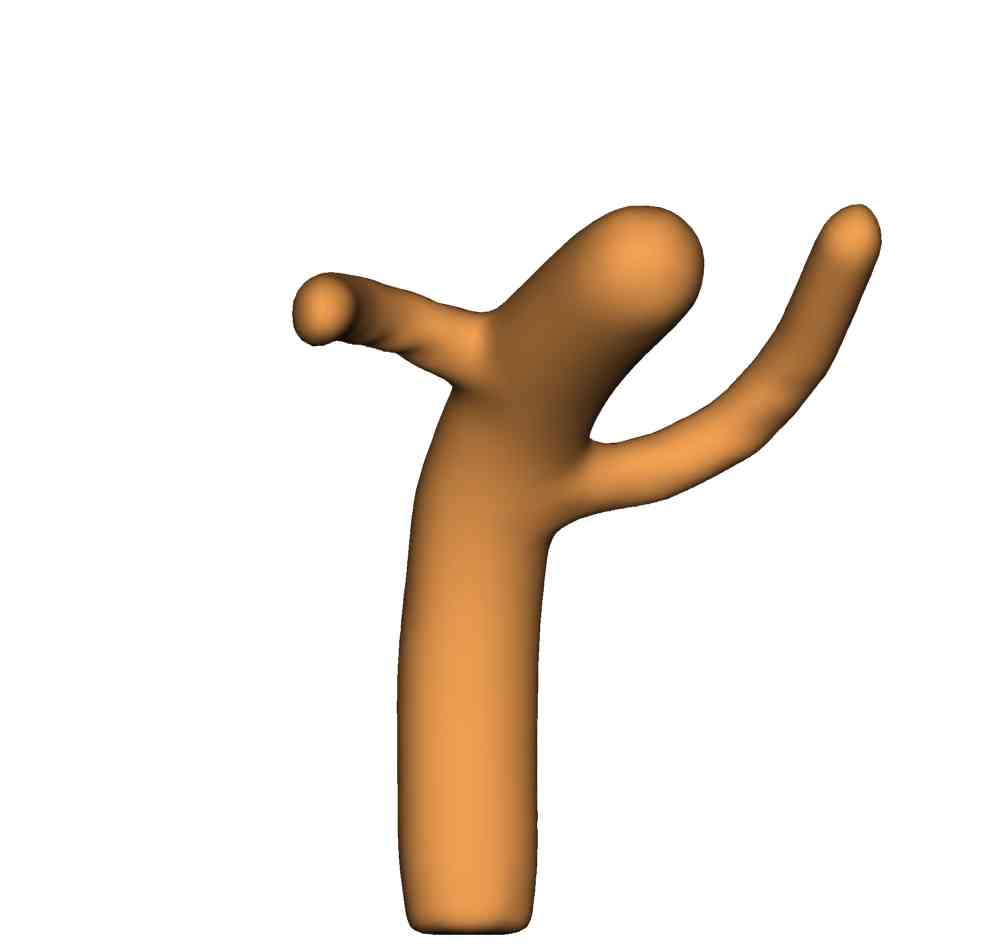}}}
\put(11,3.0){\resizebox{\unitlength}{!}{\includegraphics{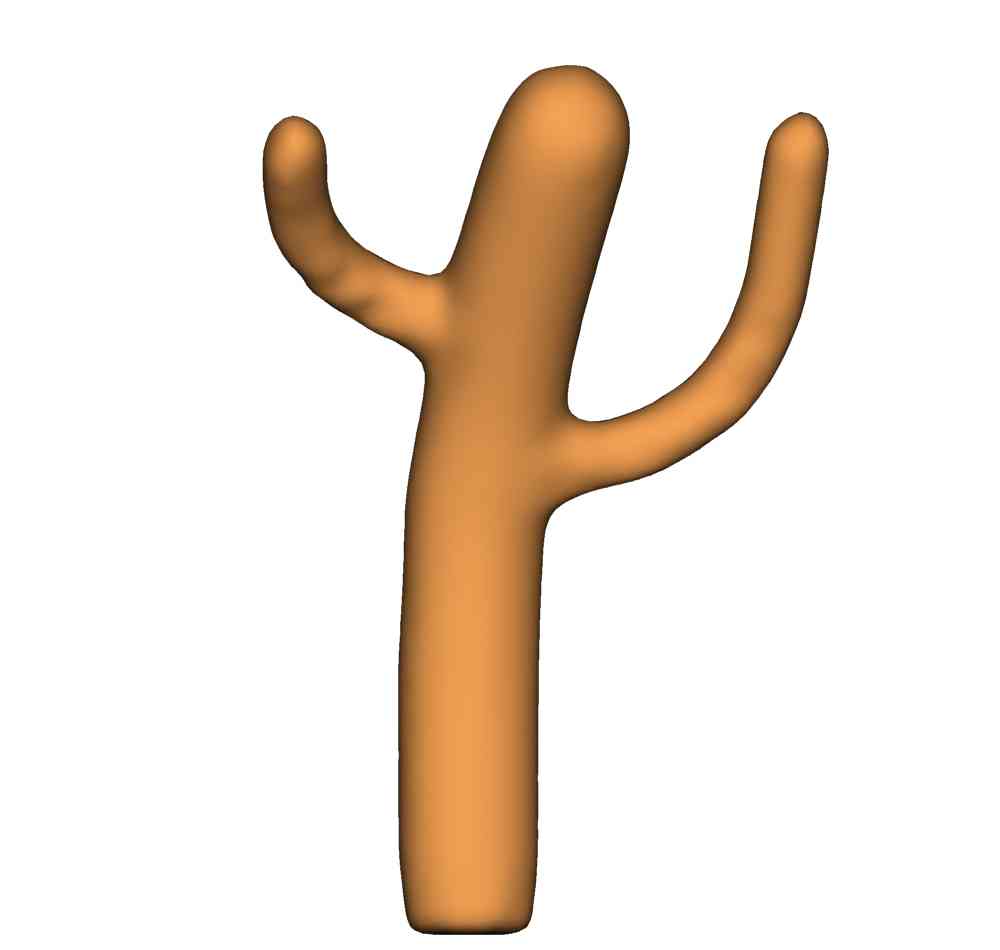}}}
\put(0,1.5){\resizebox{\unitlength}{!}{\includegraphics{cactus_spline_NL_0p0001mu_key_cropped0}}}
\put(5,1.5){\resizebox{\unitlength}{!}{\includegraphics{cactus_spline_NL_0p0001mu_key_cropped1}}}
\put(10,1.5){\resizebox{\unitlength}{!}{\includegraphics{cactus_spline_NL_0p0001mu_key_cropped2}}}
\put(2,0.3){\resizebox{\unitlength}{!}{\includegraphics{cactus_spline_NL_0p0001mu_key_cropped3}}}
\put(7,0.3){\resizebox{\unitlength}{!}{\includegraphics{cactus_spline_NL_0p0001mu_key_cropped4}}}
\put(12,0.3){\resizebox{\unitlength}{!}{\includegraphics{cactus_spline_NL_0p0001mu_key_cropped5}}}
\put(1,1.5){\resizebox{\unitlength}{!}{\includegraphics{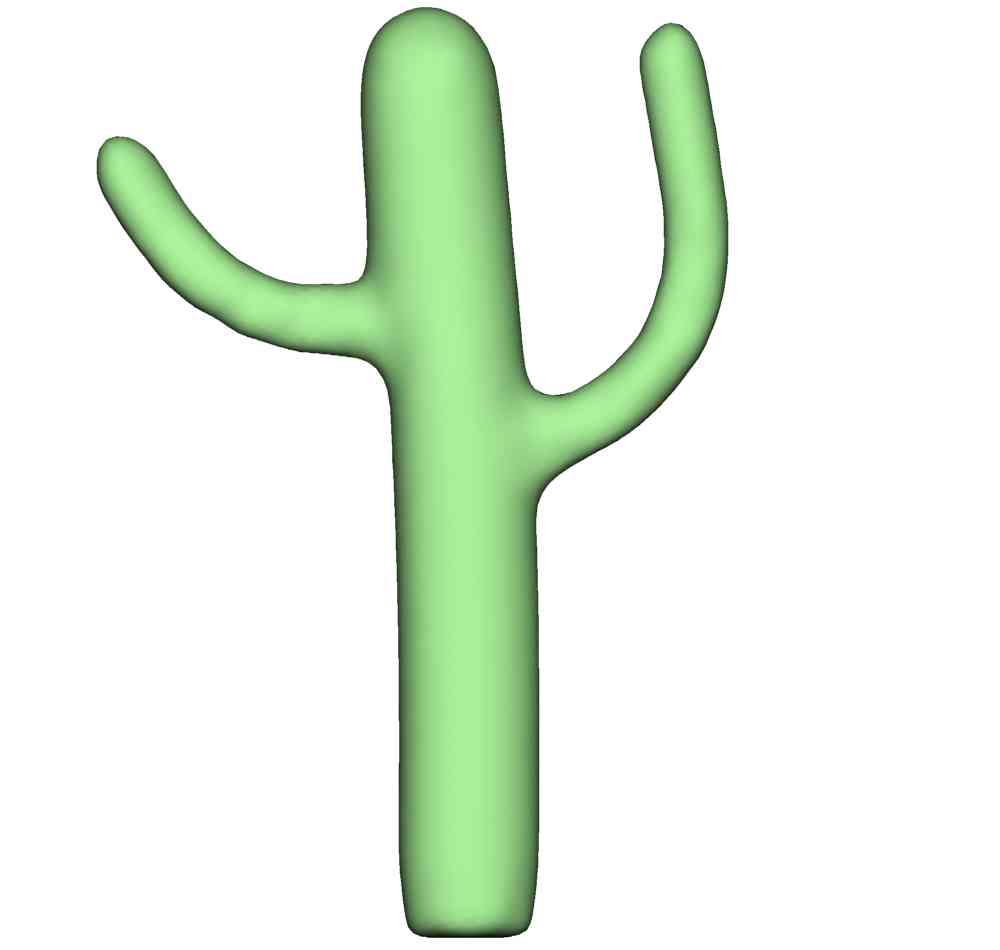}}}
\put(2,1.5){\resizebox{\unitlength}{!}{\includegraphics{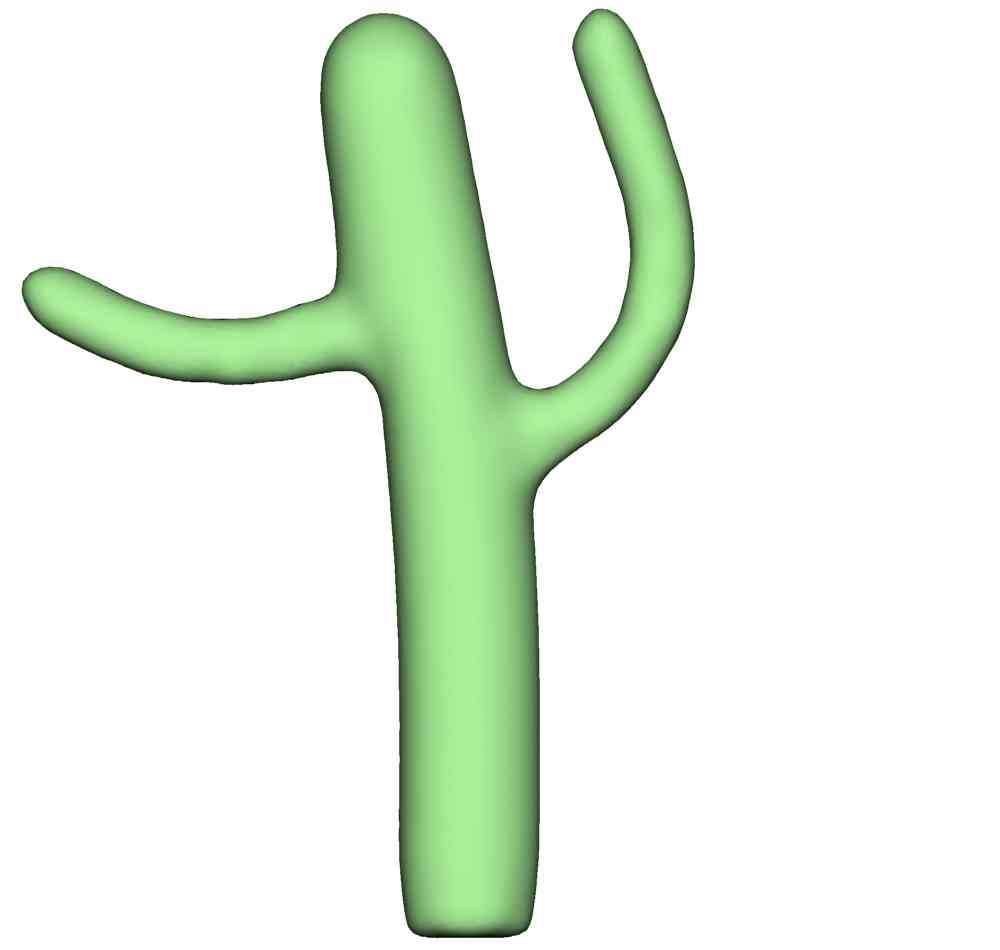}}}
\put(3,1.5){\resizebox{\unitlength}{!}{\includegraphics{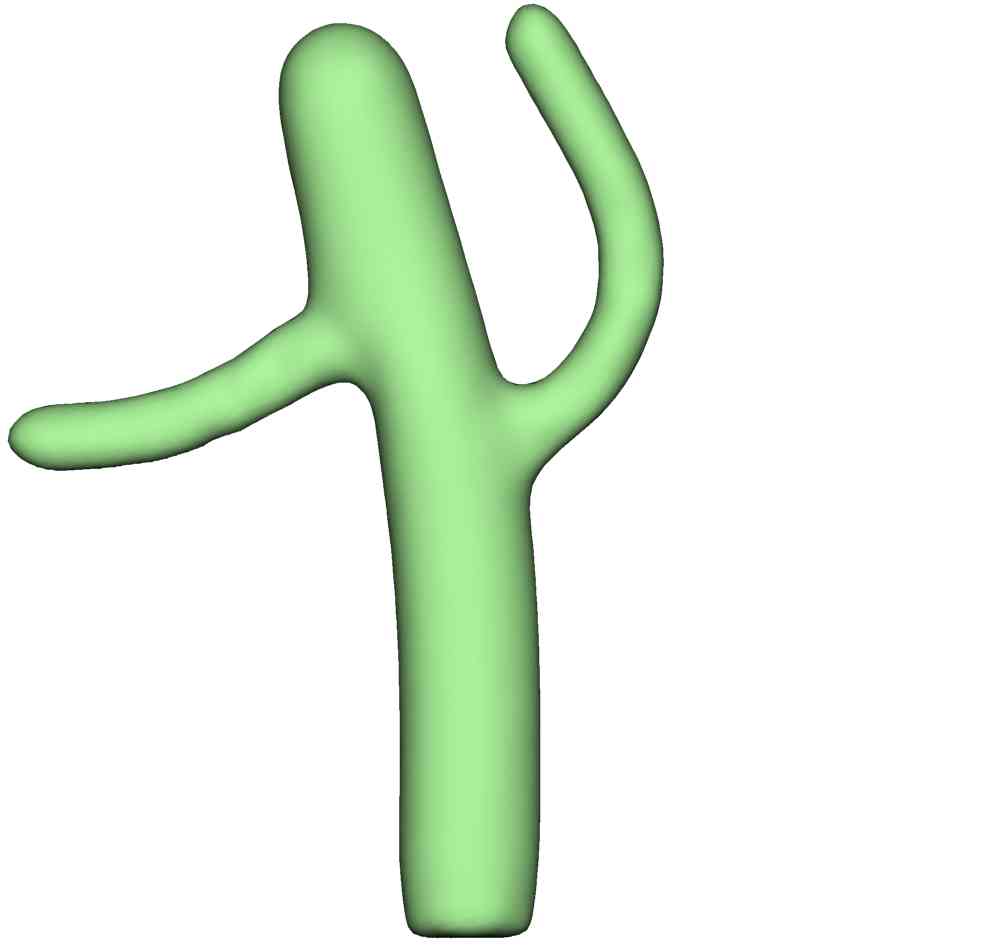}}}
\put(4,1.5){\resizebox{\unitlength}{!}{\includegraphics{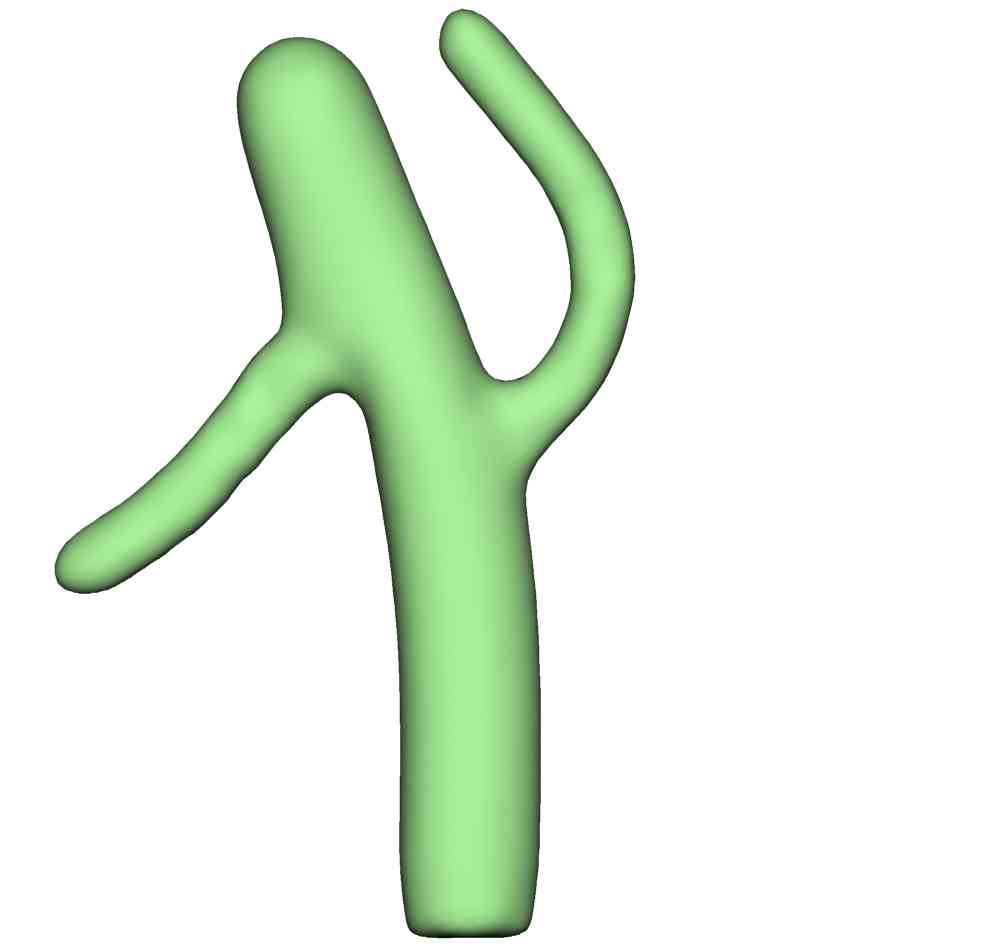}}}
\put(6,1.5){\resizebox{\unitlength}{!}{\includegraphics{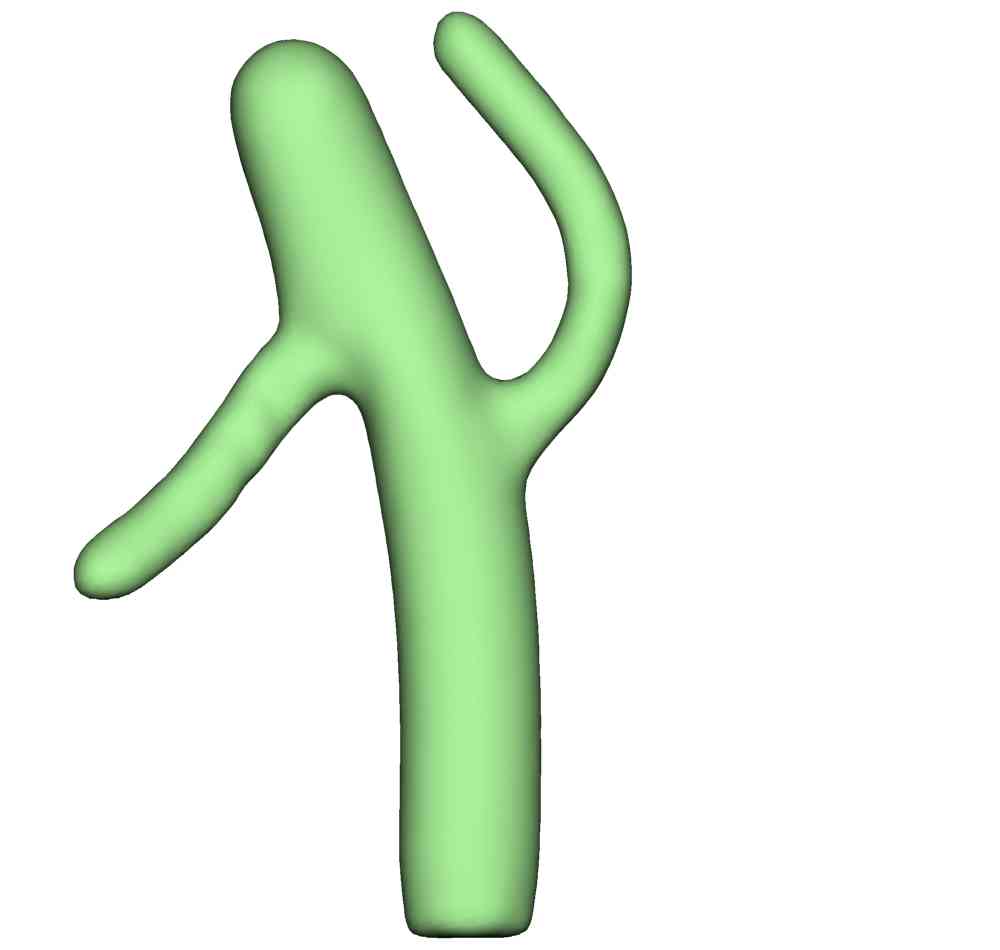}}}
\put(7,1.5){\resizebox{\unitlength}{!}{\includegraphics{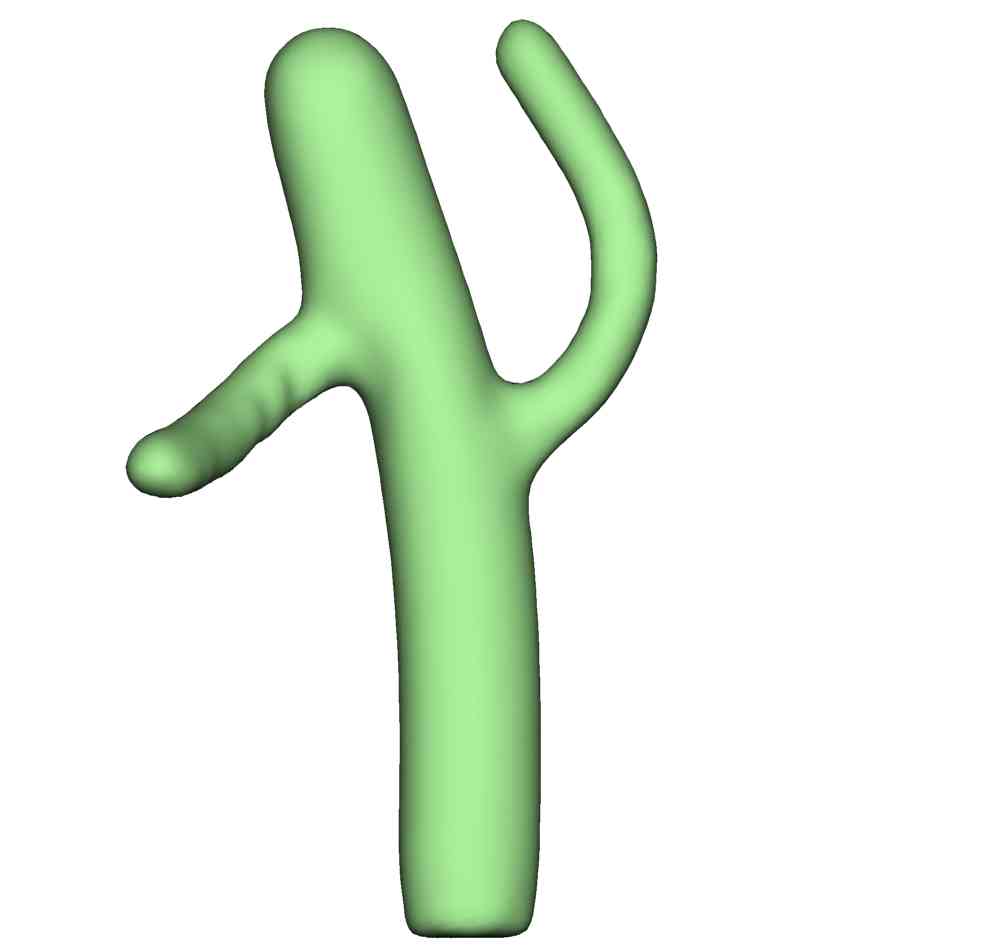}}}
\put(8,1.5){\resizebox{\unitlength}{!}{\includegraphics{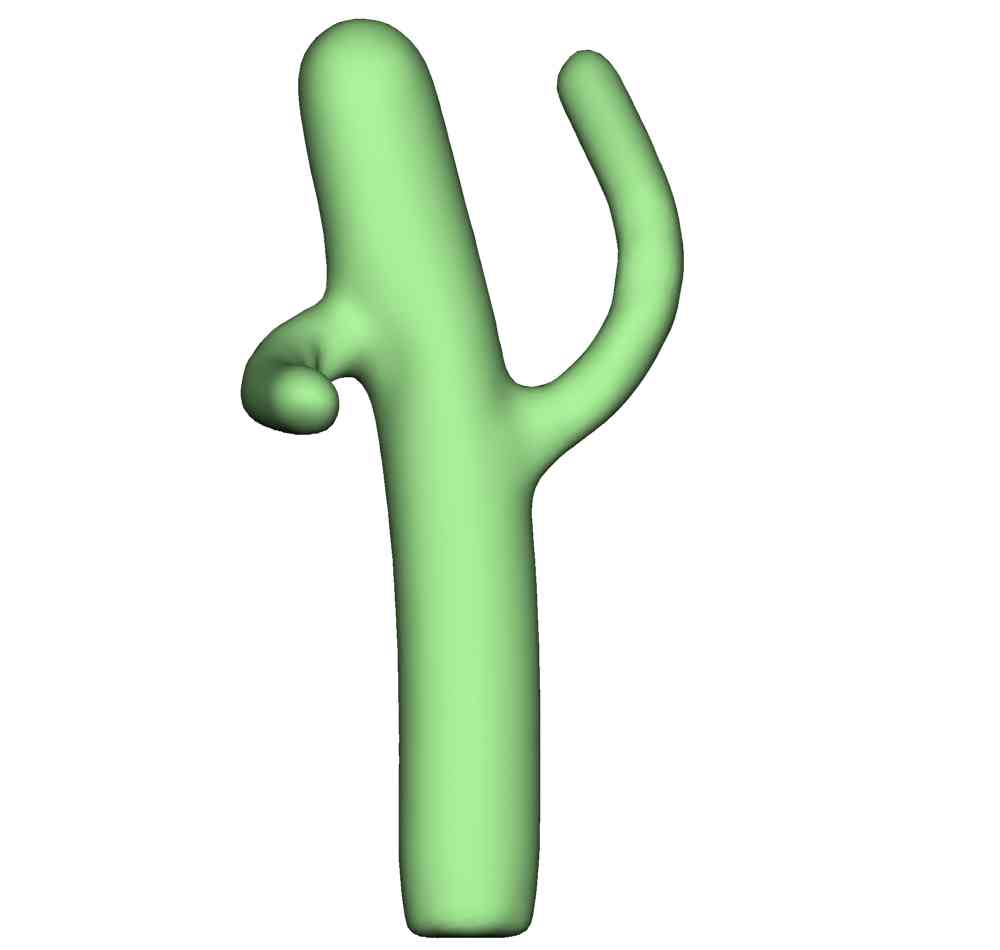}}}
\put(9,1.5){\resizebox{\unitlength}{!}{\includegraphics{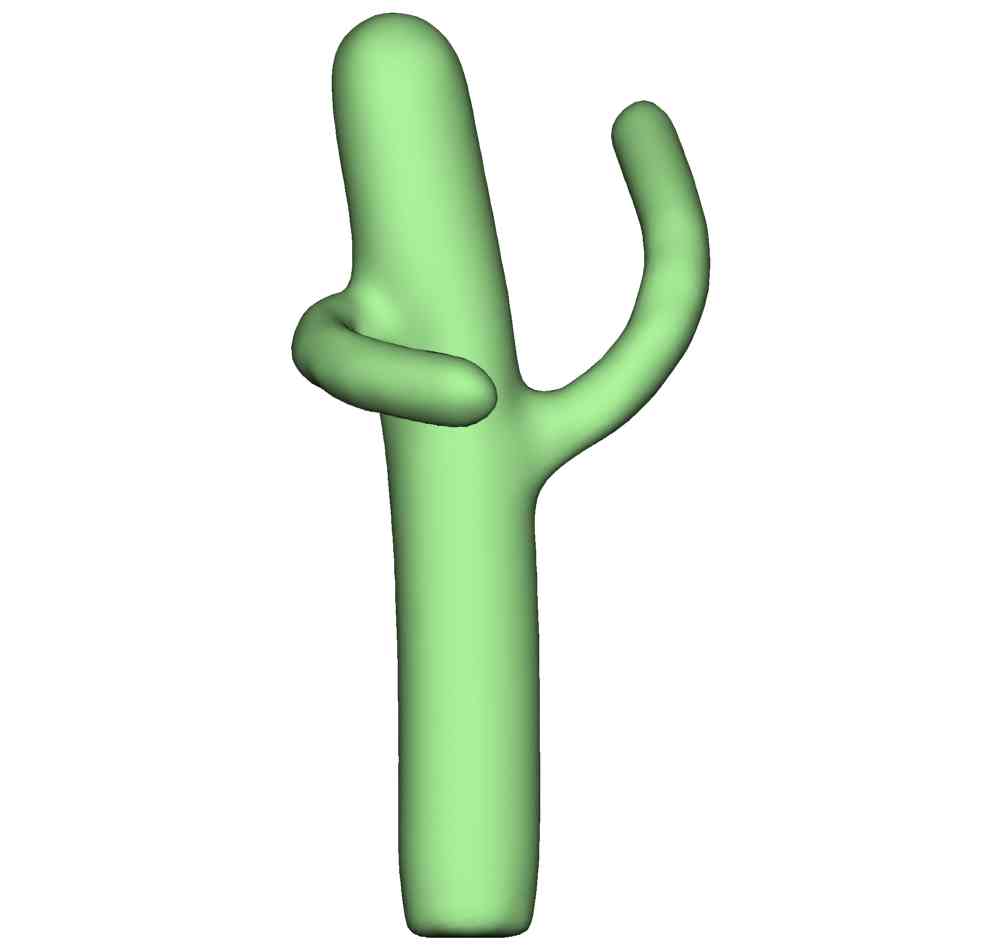}}}
\put(11,1.5){\resizebox{\unitlength}{!}{\includegraphics{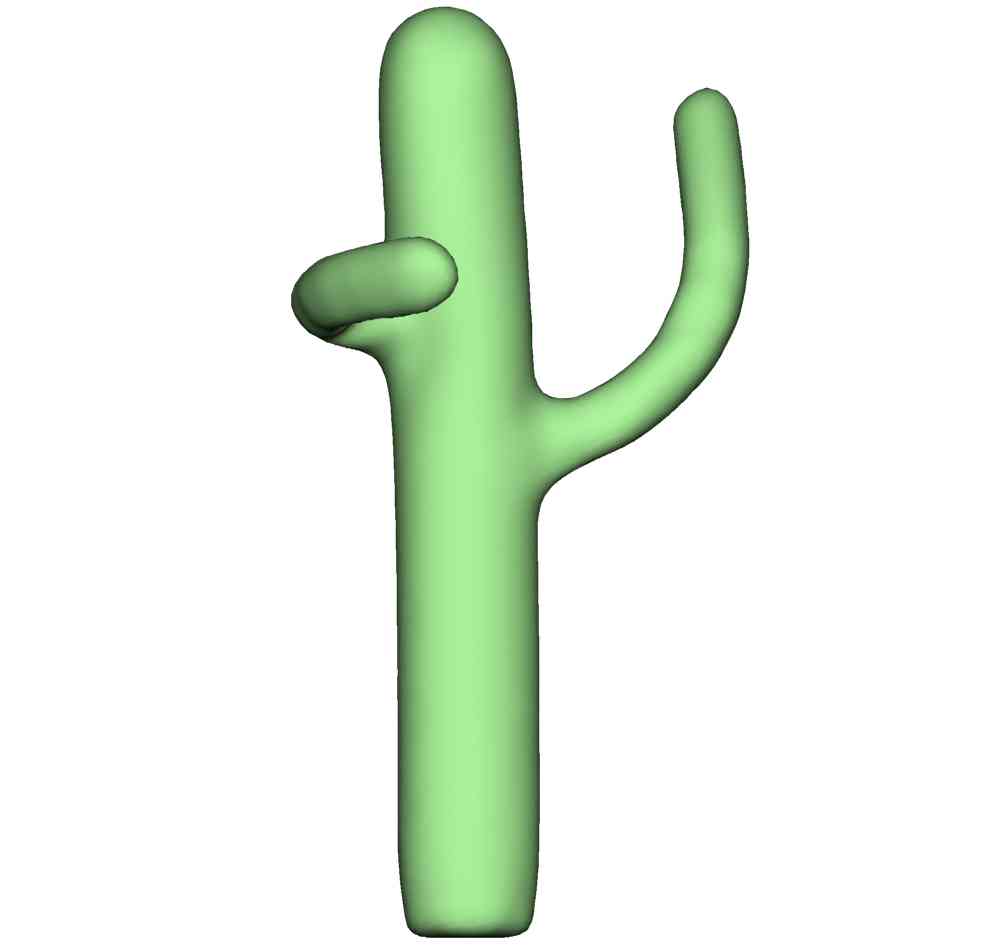}}}
\put(12,1.5){\resizebox{\unitlength}{!}{\includegraphics{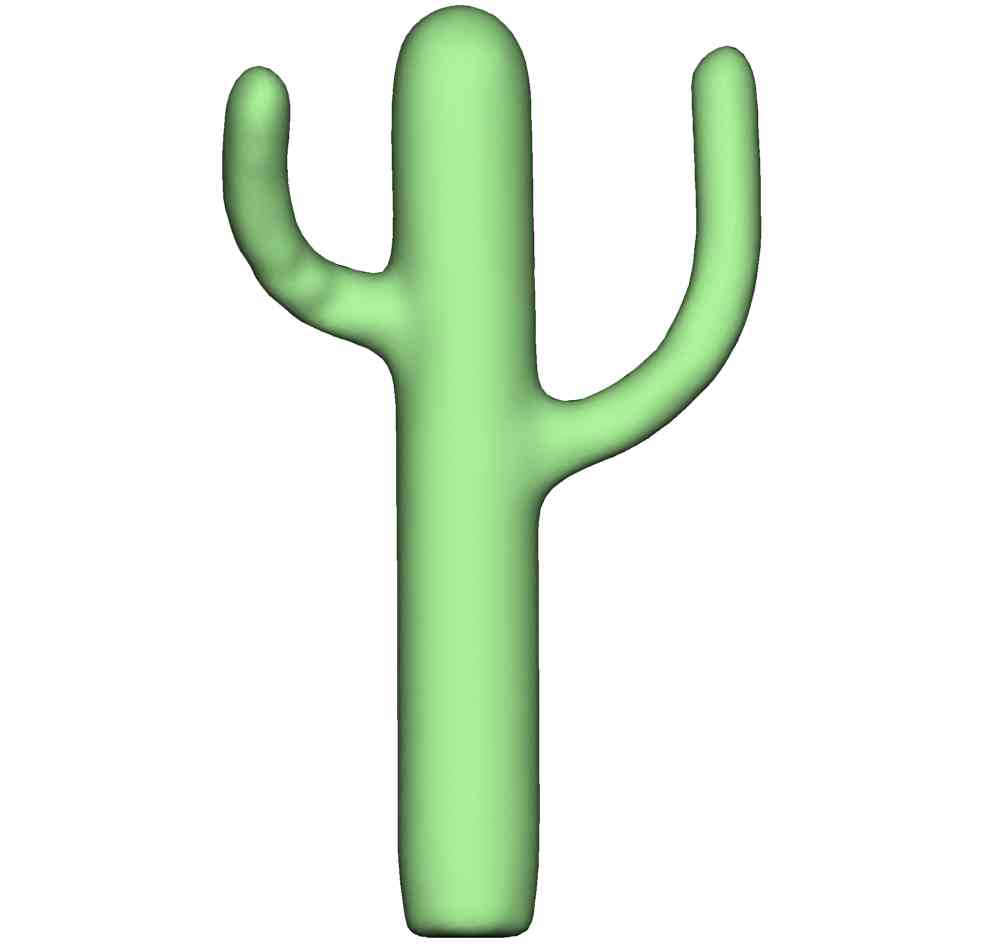}}}
\put(0,0.3){\resizebox{\unitlength}{!}{\includegraphics{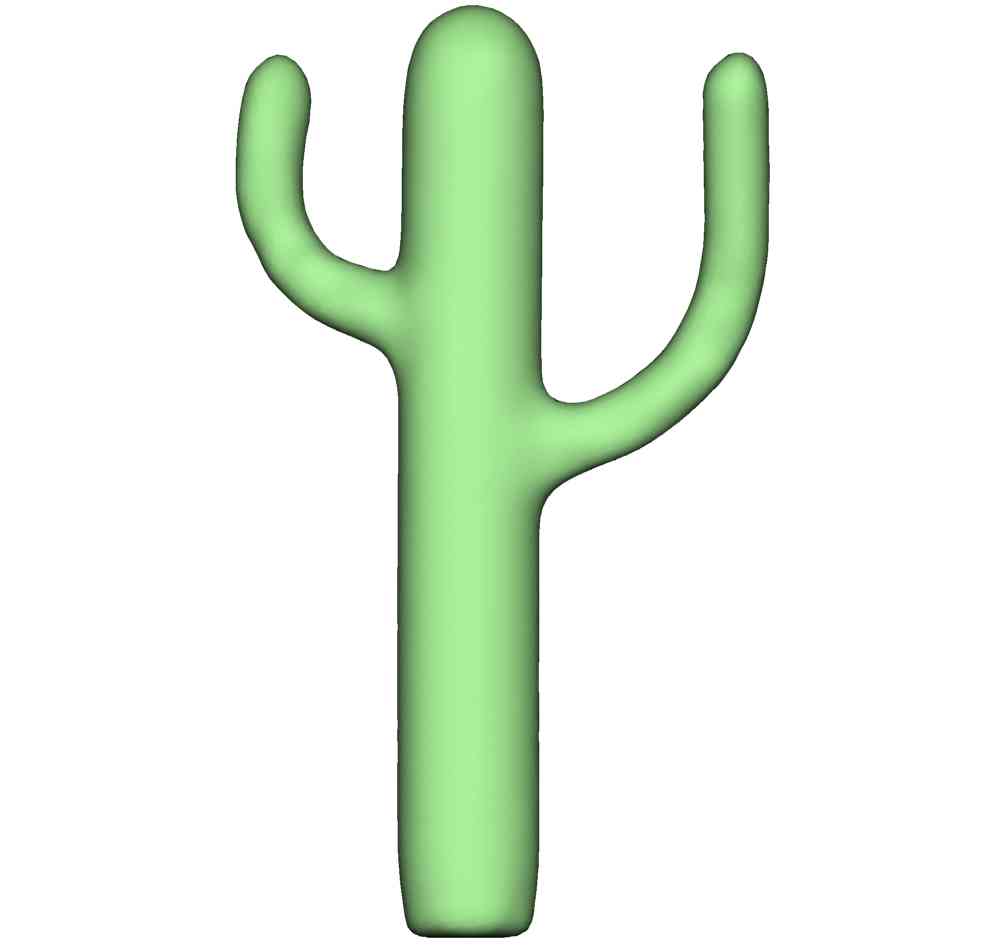}}}
\put(1,0.3){\resizebox{\unitlength}{!}{\includegraphics{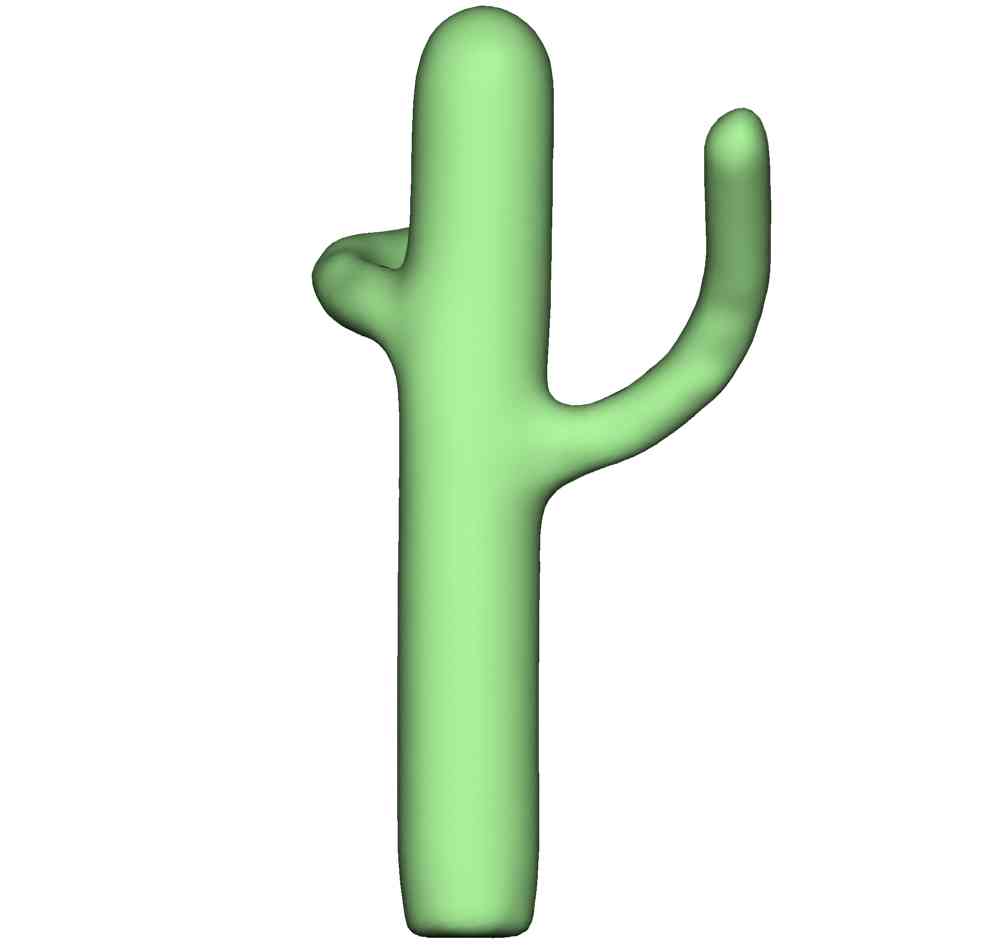}}}
\put(3,0.3){\resizebox{\unitlength}{!}{\includegraphics{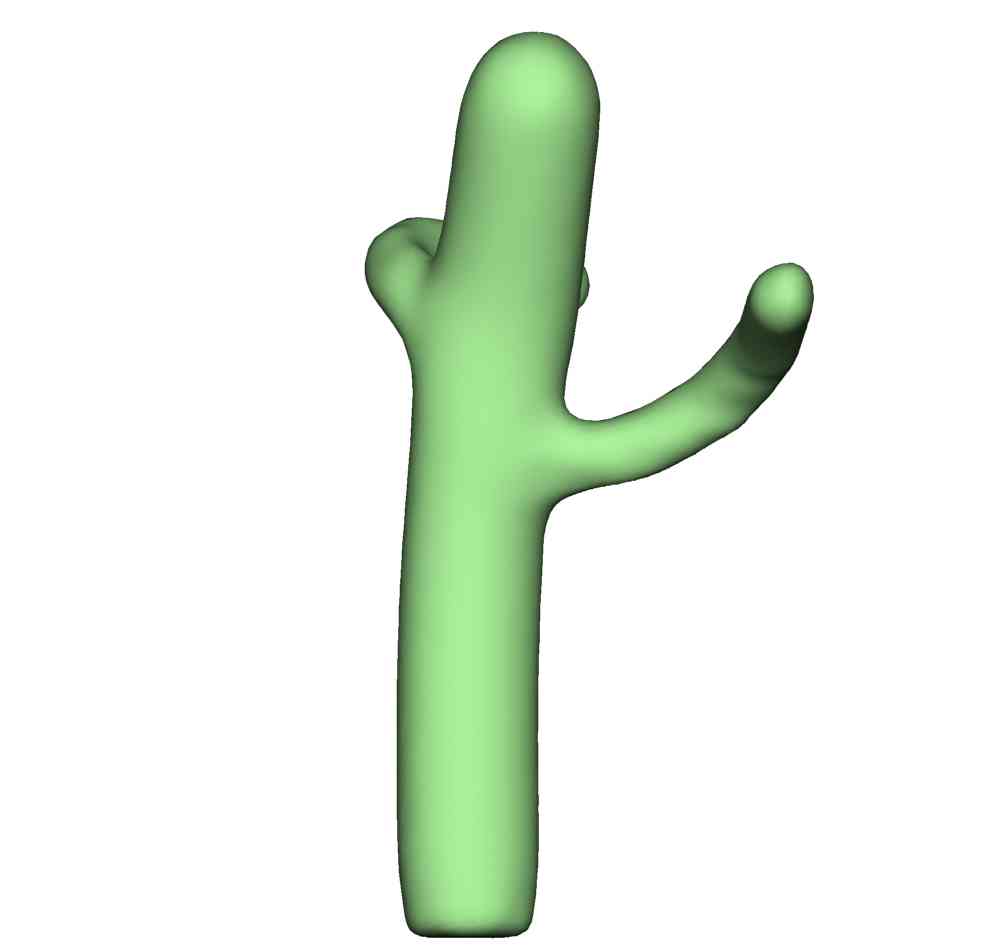}}}
\put(4,0.3){\resizebox{\unitlength}{!}{\includegraphics{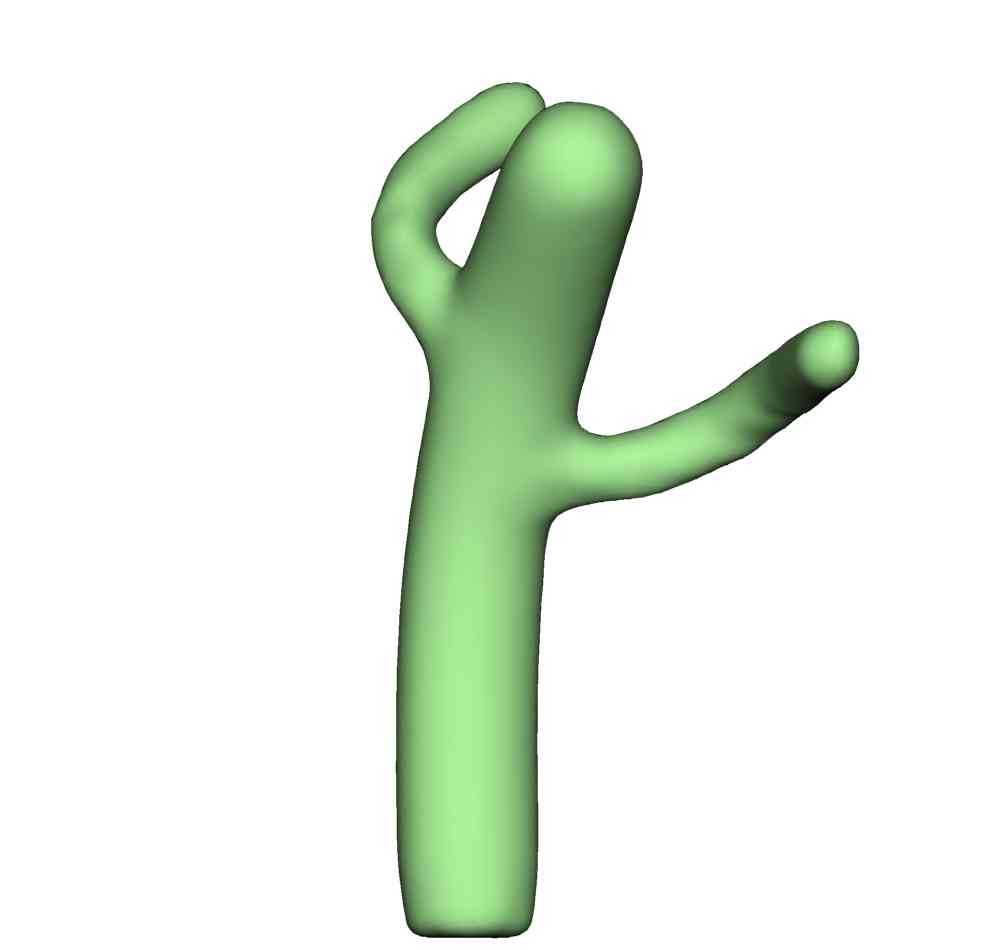}}}
\put(5,0.3){\resizebox{\unitlength}{!}{\includegraphics{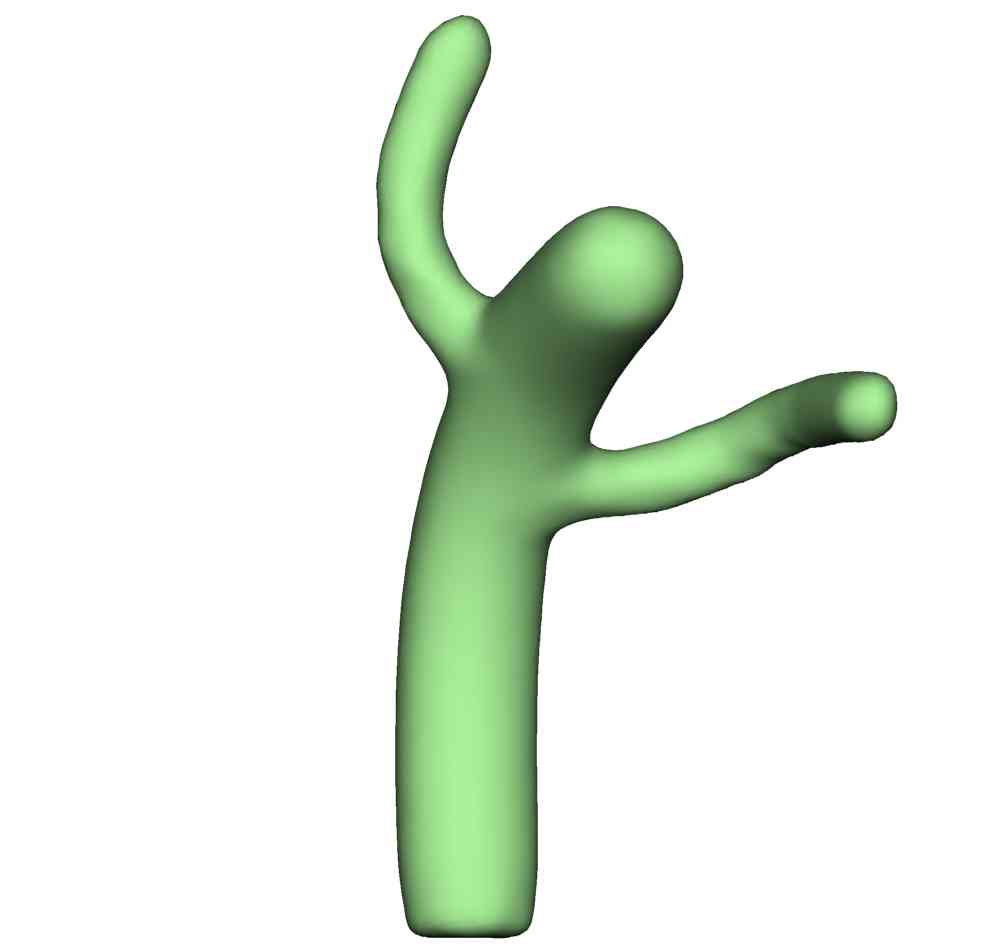}}}
\put(6,0.3){\resizebox{\unitlength}{!}{\includegraphics{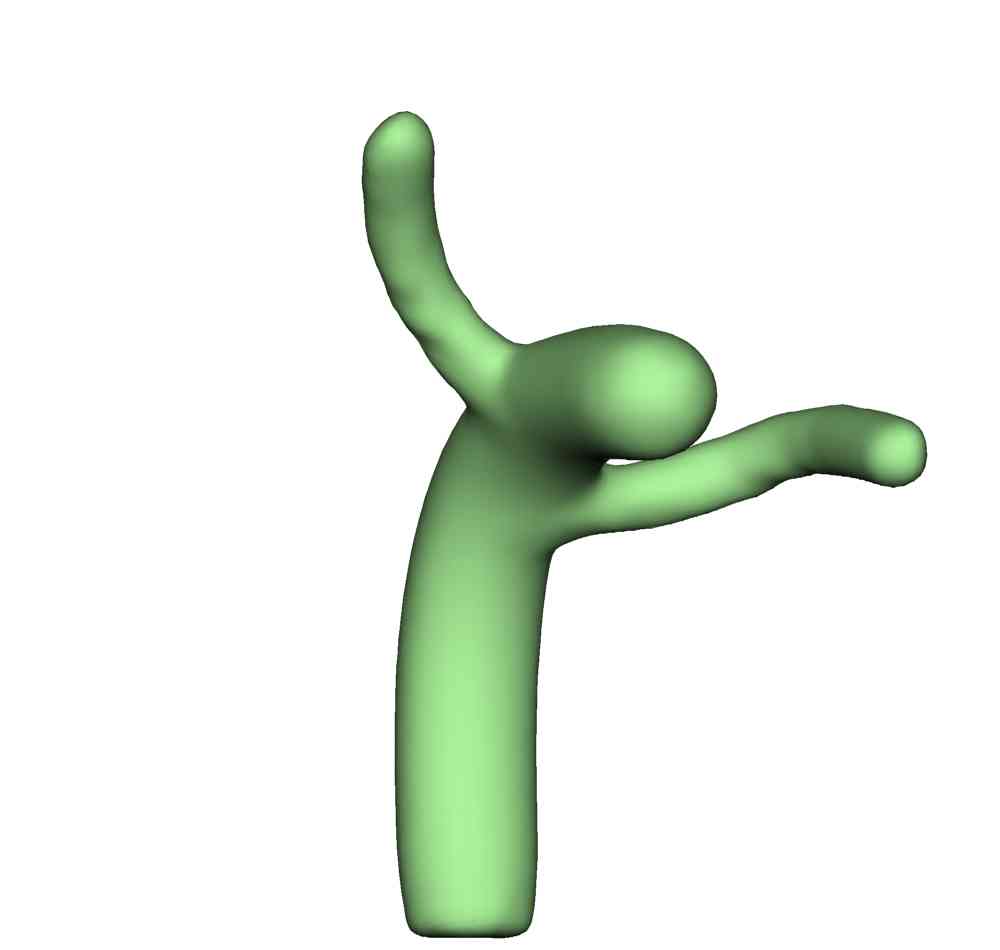}}}
\put(8,0.3){\resizebox{\unitlength}{!}{\includegraphics{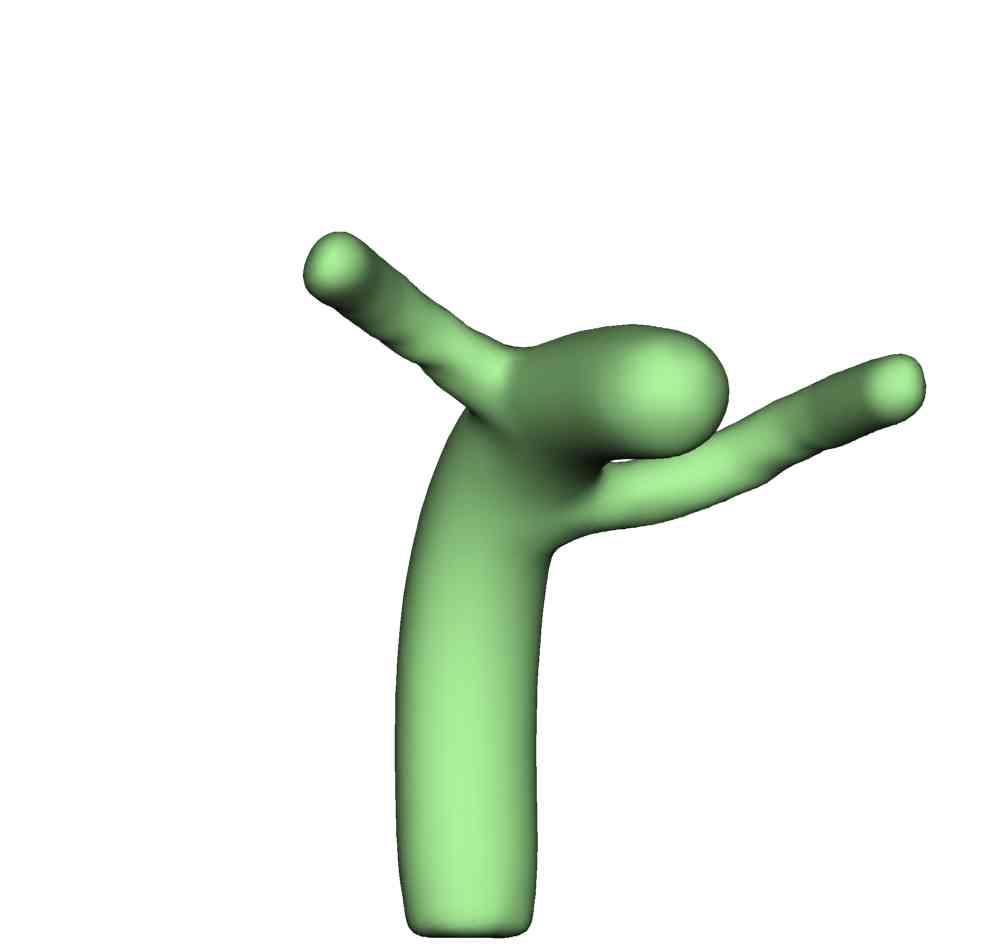}}}
\put(9,0.3){\resizebox{\unitlength}{!}{\includegraphics{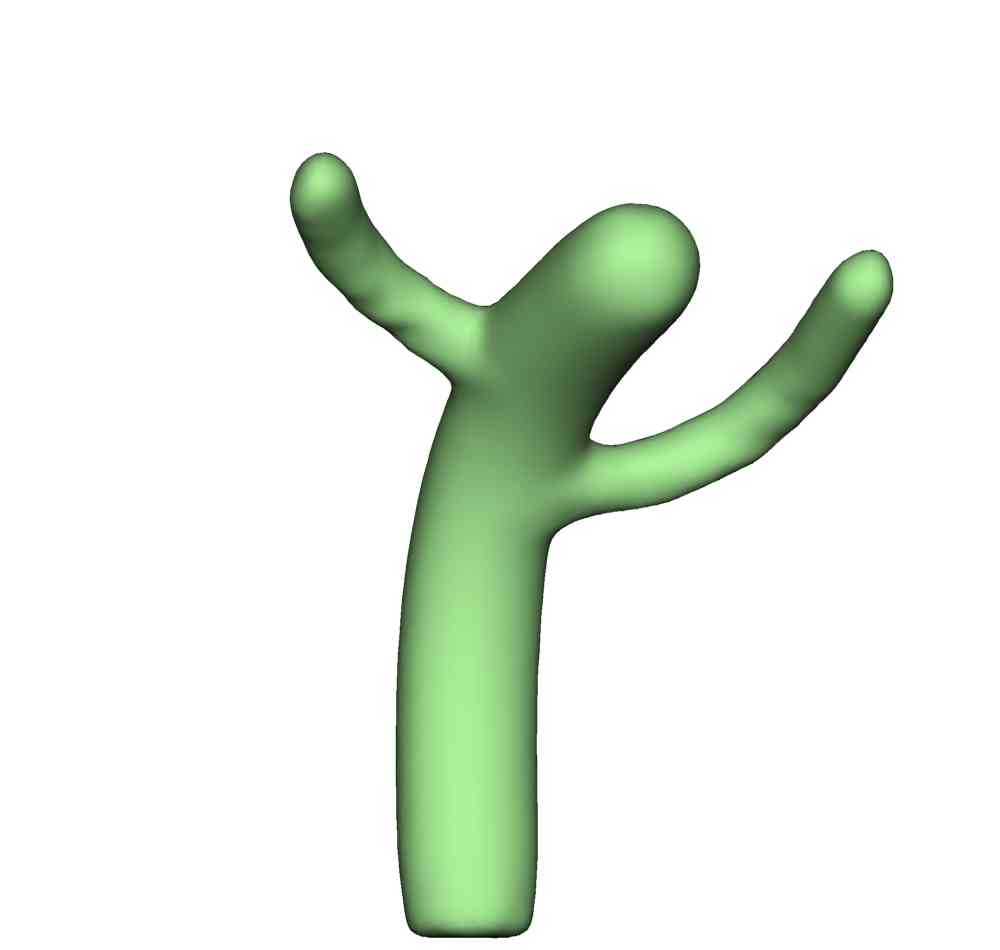}}}
\put(10,0.3){\resizebox{\unitlength}{!}{\includegraphics{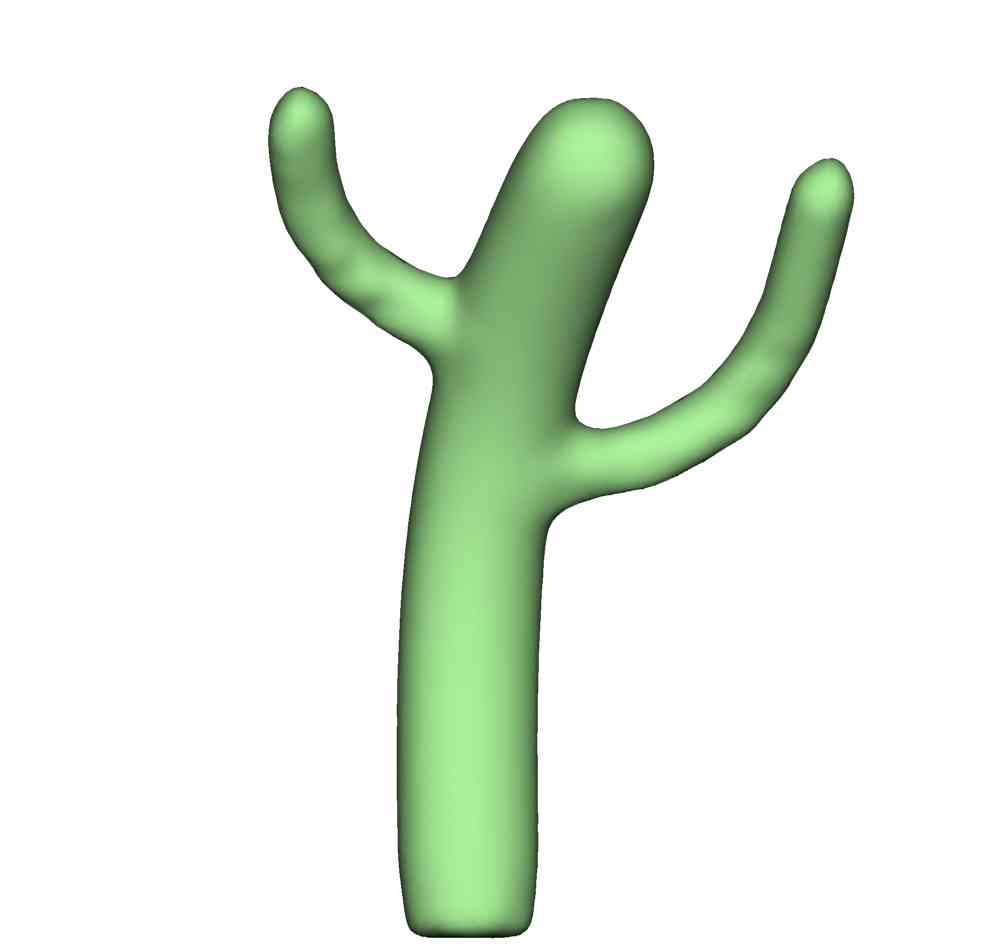}}}
\put(11,0.3){\resizebox{\unitlength}{!}{\includegraphics{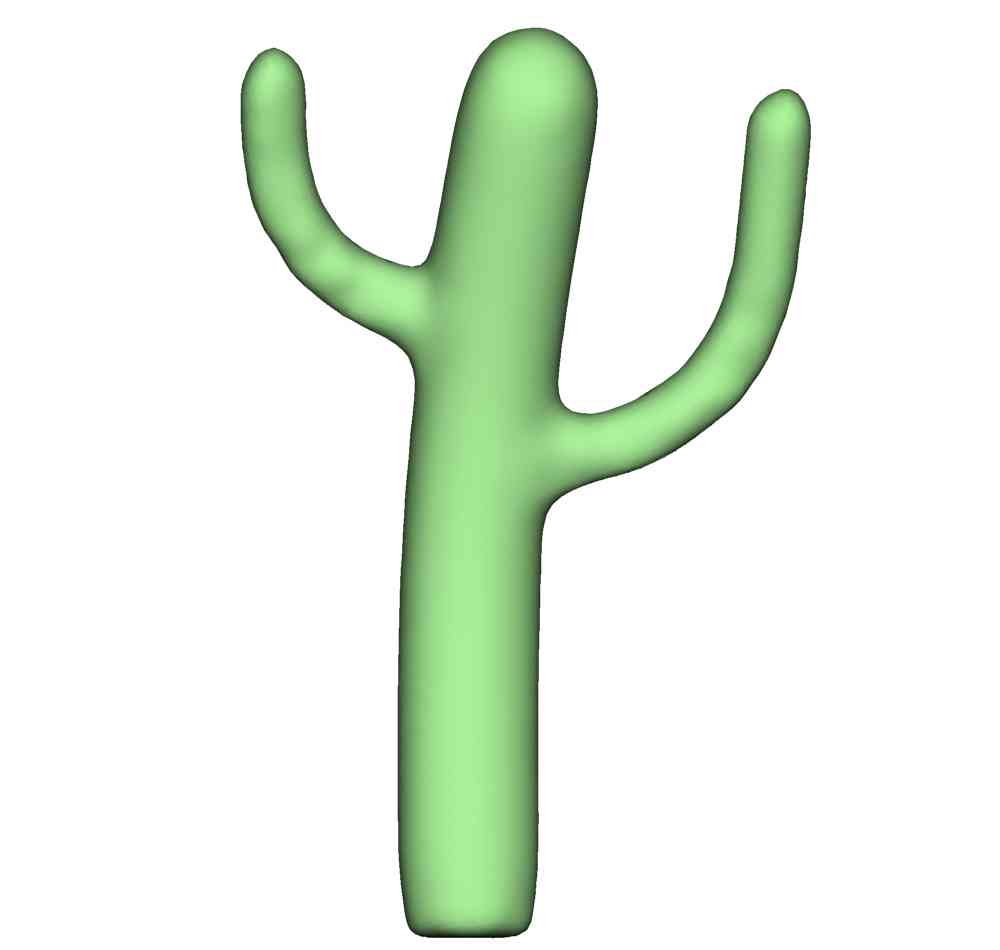}}}
\end{picture}

\centering
\setlength{\unitlength}{.13\linewidth}%
\begin{picture}(6,1.3)
 \put(0,0.2){\resizebox{.85\unitlength}{!}{\includegraphics{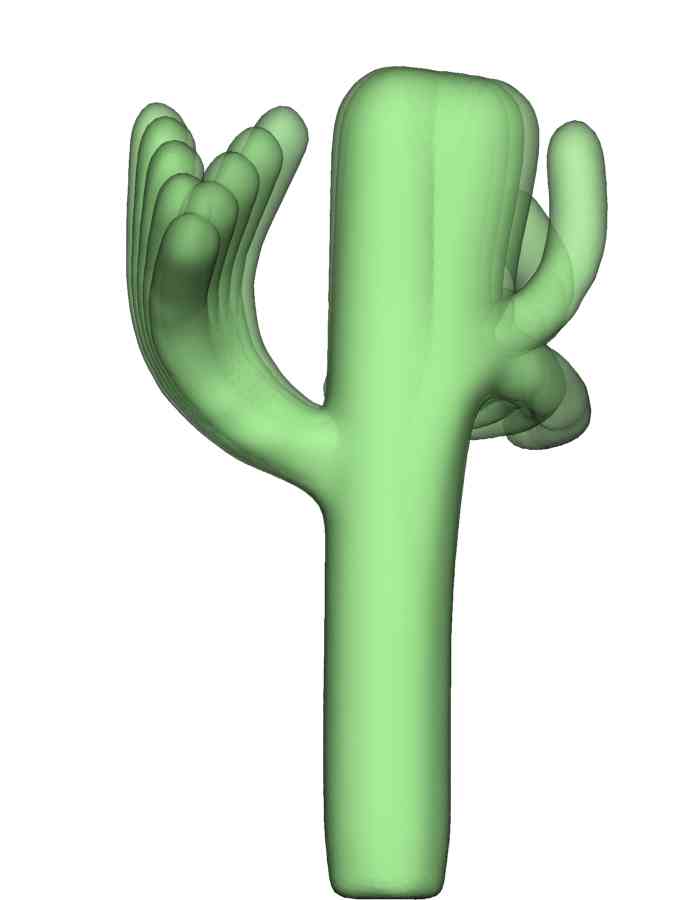}}}
  \put(0.8,0.2){\resizebox{.85\unitlength}{!}{\includegraphics{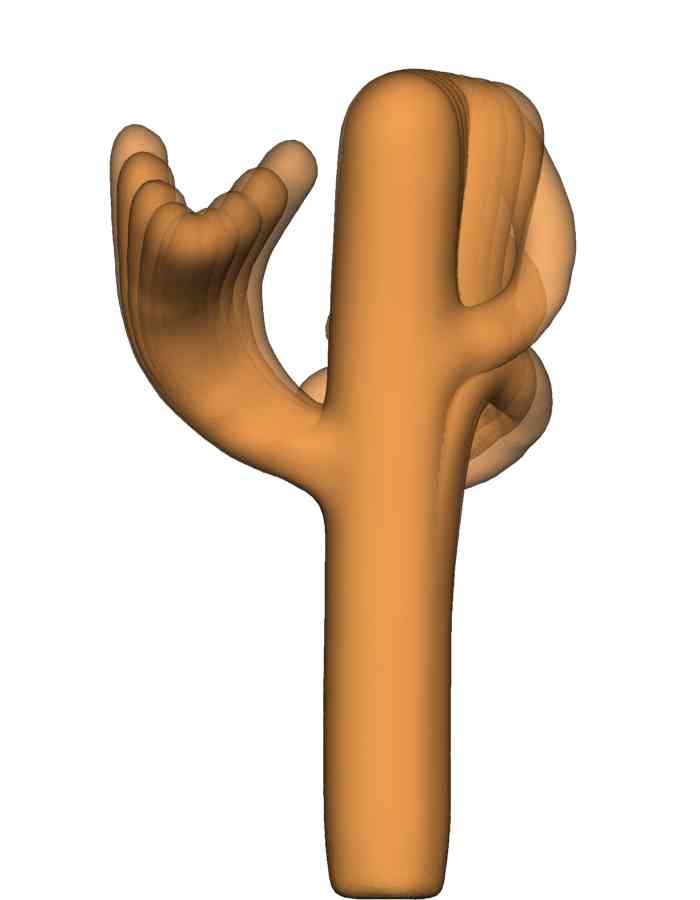}}}
  \put(1.8,0.1){\resizebox{\unitlength}{!}{\includegraphics{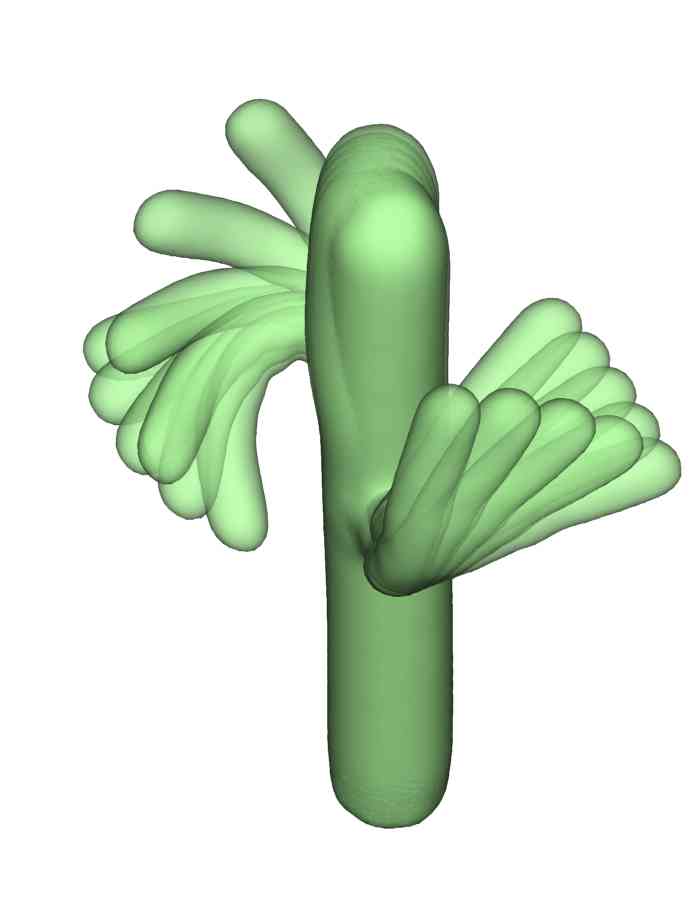}}}
  \put(2.8,0.1){\resizebox{\unitlength}{!}{\includegraphics{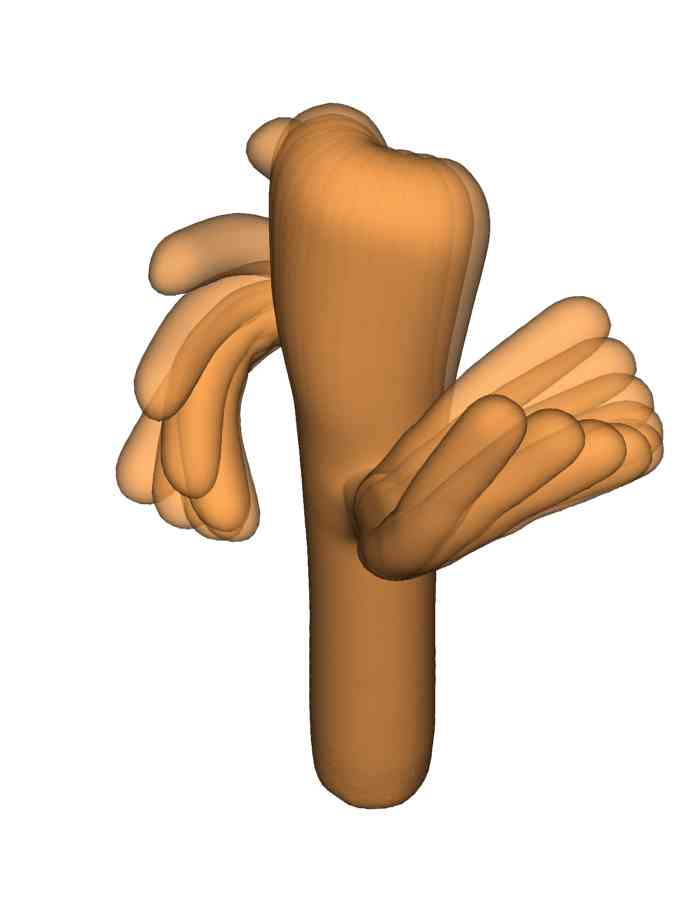}}}
   \put(4.05,0.3){\resizebox{\unitlength}{!}{\includegraphics{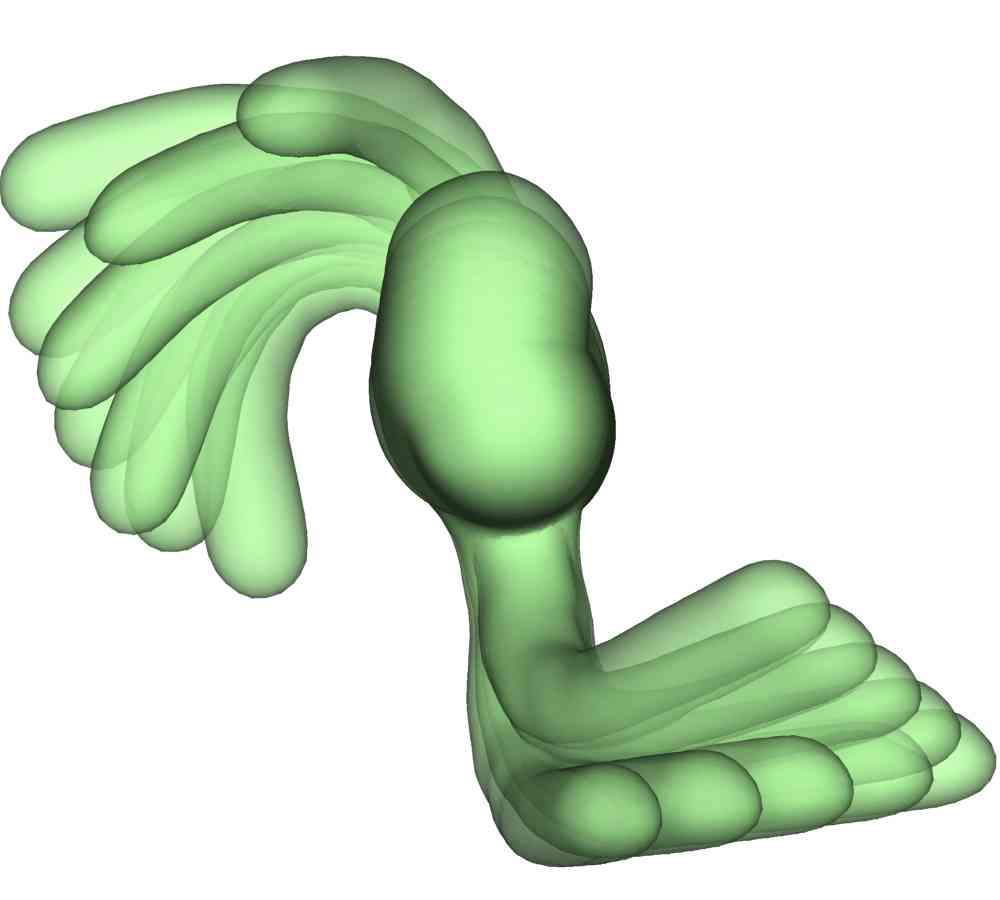}}}
  \put(5.1,0.3){\resizebox{\unitlength}{!}{\includegraphics{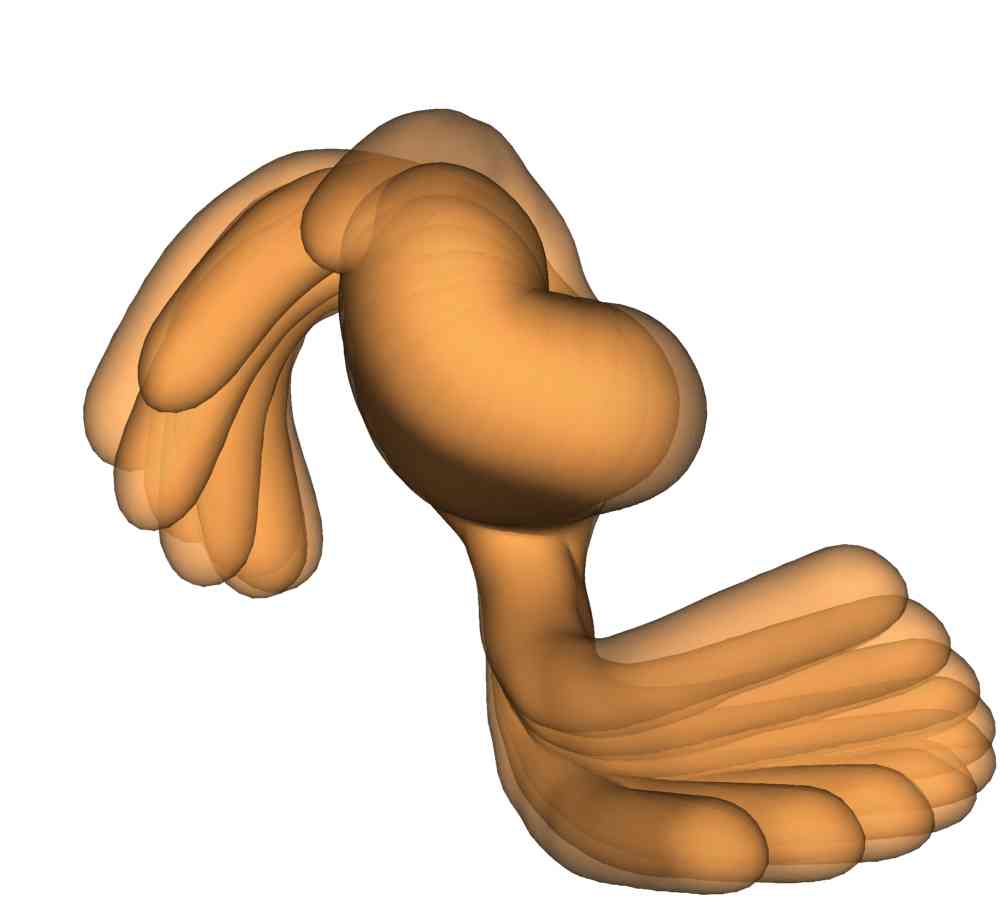}}}
\end{picture}

\centering
\begin{tikzpicture}[scale=0.8]
\begin{axis}[x=0.34cm,y=.03cm, xtickmin = 0, xtickmax = 50, enlarge x limits=false,ytick=\empty ]
\addplot[color=myOrange, very thick] plot coordinates {  (1, 1.03*0.1 )  ( 2,    2.27*0.1 )  ( 3,    3.31*0.1 )  ( 4,    4.85*0.1 )  ( 5,    7.13*0.1 )  ( 6,    1.02 )  ( 7,    1.42 )  ( 8,    1.92 )  ( 9,    2.53 )  ( 10,    3.32 )  ( 11,    1.98 )  ( 12,    1.18 )  ( 13,    7.69*0.1 )  ( 14,    6.86*0.1 )  ( 15,    9.43*0.1 )  ( 16,    1.55 )  ( 17,    2.59 )  ( 18,    4.13 )  ( 19,    6.17 )  ( 20,    8.97 )  ( 21,    5.67 )  ( 22,    3.83 )  ( 23,    2.75 )  ( 24,    2.03 )  ( 25,    1.75 )  ( 26,    2.01 )  ( 27,    2.72 )  ( 28,    4.02 )  ( 29,    6.09 )  ( 30,    9.32 )  ( 31,    6.29 )  ( 32,    4.13 )  ( 33,    2.58 )  ( 34,    1.55 )  ( 35,    9.45*0.1 )  ( 36,    7.19*0.1 )  ( 37,    8.99*0.1 )  ( 38,    1.55 )  ( 39,    2.78 )  ( 40,    4.86 )  ( 41,    3.65 )  ( 42,    2.79 )  ( 43,    2.10 )  ( 44,    1.56 )  ( 45,    1.20 )  ( 46,    9.98*0.1 )  ( 47,    8.50*0.1 )  ( 48,    6.88*0.1 )  ( 49,    5.93*0.1 )  };
\addplot[color=myGreen, very thick] plot coordinates {   ( 1, 0.015221 )  ( 2, 0.018111 )  ( 3, 0.0173745 )  ( 4, 0.0171919 )  ( 5, 0.0174105 )  ( 6, 0.0172703 )  ( 7, 0.017123 )  ( 8, 0.017094 )  ( 9, 0.0143724 )  ( 10, 37.3902 )  ( 11, 0.0206153 )  ( 12, 0.0234803 )  ( 13, 0.0228676 )  ( 14, 0.0225033 )  ( 15, 0.0225812 )  ( 16, 0.0232982 )  ( 17, 0.0234791 )  ( 18, 0.0244838 )  ( 19, 0.0225182 )  ( 20, 86.9195 )  ( 21, 0.066487 )  ( 22, 0.0638303 )  ( 23, 0.0603444 )  ( 24, 0.0582485 )  ( 25, 0.0576028 )  ( 26, 0.0584422 )  ( 27, 0.0608681 )  ( 28, 0.0634713 )  ( 29, 0.065535 )  ( 30, 86.8167 )  ( 31, 0.0216524 )  ( 32, 0.0240733 )  ( 33, 0.0231159 )  ( 34, 0.0232251 )  ( 35, 0.0229853 )  ( 36, 0.0234122 )  ( 37, 0.023376 )  ( 38, 0.0230098 )  ( 39, 0.0202336 )  ( 40, 59.0867 )  ( 41, 0.0150149 )  ( 42, 0.0176436 )  ( 43, 0.0179652 )  ( 44, 0.0179942 )  ( 45, 0.0180151 )  ( 46, 0.0183136 )  ( 47, 0.0179488 )  ( 48, 0.0174361 )  ( 49, 0.0146319 )   };
\end{axis}
\end{tikzpicture}
 \setlength{\unitlength}{.162\linewidth}%
\begin{picture}(5,0.18)
 \put(-0.1,-0.1){\resizebox{.3\unitlength}{!}{\includegraphics{cactus_spline_NL_0p0001mu_key_cropped0}}}
\put(0.8,-0.1){\resizebox{.3\unitlength}{!}{\includegraphics{cactus_spline_NL_0p0001mu_key_cropped1}}}
\put(1.8,-0.1){\resizebox{.3\unitlength}{!}{\includegraphics{cactus_spline_NL_0p0001mu_key_cropped2}}}
\put(2.8,-0.1){\resizebox{.3\unitlength}{!}{\includegraphics{cactus_spline_NL_0p0001mu_key_cropped3}}}
\put(3.8,-0.1){\resizebox{.3\unitlength}{!}{\includegraphics{cactus_spline_NL_0p0001mu_key_cropped4}}}
\put(4.8,-0.1){\resizebox{.3\unitlength}{!}{\includegraphics{cactus_spline_NL_0p0001mu_key_cropped5}}}
\end{picture}

\caption{Top: interpolation of six input shapes (grey) by a discrete spline (orange) and a piecewise discrete geodesic (green), respectively. 
Each discrete path consists of $51$ discrete shells in total where we have depicted every second shape here.
The chosen parameters are $\lambda =1$, $\mu =1$, $\zeta=1$, $\eta =10^{-4}$. 
Middle: Three different views of the subsequences from shape 16 to shape 24 (superposed) for the discrete spline (orange) and the piecewise discrete geodesic (green), respectively. 
Bottom: distribution of the energy $k \mapsto \energy[ \y_k, \widetilde \y_k]$ both for the discrete spline (orange) and the piecewise discrete geodesic (green).}
\label{fig:cactus}
\end{figure}
Figure \ref{fig:cactus} shows a spline curve for six given input poses of the discrete shell model of a cactus and compares the discrete spline interpolation and a piecewise discrete geodesic interpolation. 
The plotted quantity $\energy[ \y_k, \tilde \y_k]$ is considered as an approximation of the squared covariant derivative $\|\cov \dot y\|^2$ 
(cf.\ Remark~\ref{remark:motivation}).
It reflects the postulated regularity of spline curves. Indeed, we experimentally observe that $\|\cov \dot y\|^2$ appears to be bounded but not differentiable.
The profile is approximately parabolic as would be the profile for Euclidean spline interpolation, which exhibits a piecewise affine acceleration (see also Section~\ref{sec:embedded}).

\subsection{Viscous rods} \label{sec:viscousrods}
Unlike shells, which represent thin, macroscopically two-dimensional material layers,
the space of (linearized) viscous rods contains thin, macroscopically one-dimensional rods of material.
This shape space can be used to model curves, for instance contours of objects in 2D.
Again, the metric of the shape space is based on the energy dissipation due to rod stretching and bending.
Below we shall consider closed viscous rods, that is, shapes that have the same topology as $S^1$.
Let us pick the definition introduced in \cite{RuWi12b} and modify it slightly to ensure existence of 
spline curves.

\begin{definition}[Viscous rods]
Given a template shape $\y_A:S^1\to\R^2$ with Sobolev regularity $W^{2,2}(S^1;\R^2)$ and a (sufficiently small) $\eta>0$, the manifold of \emph{viscous rods} is given as
\begin{equation*}
\textstyle\manifold=\big\{\y:S^1\to\R^2\,|\,\|\y-\y_A\|_{W^{2,2}(S^1;\R^2)}^2\leq\eta,\,\int_{S^1}\y(s) \d s = 0\big\}\,,
\end{equation*}
thus it contains all closed curves in $\R^2$ which do not deviate too much from the template in the $W^{2,2}$ sense
and whose centre of mass is fixed to eliminate mere shape translations.
The \emph{dissipation} and \emph{metric} on $\manifold$ are given as
\begin{align*}
\energy[\y,\widetilde{\y}]
&=\int_{S^1} \frac{\delta }{2} \left(1-\frac{|\widetilde{\y}_{,s}|^2}{|\y_{,s}|^2}\right)^2 |\y_{,s}| + \delta^3  (\widetilde{\y}_{,ss}-\y_{,ss})^2 \d s\,,\\
g_{\y}(v,v)
&=\tfrac12\partial_2^2\energy[\y,\y](v,v)
= \int_{S^1} 2 \delta  \frac{|v^\tgl_{,s}|^2}{|\y_{,s}|^2} |\y_{,s}| + \delta^3 |v_{,ss}|^2 \d s\,,
\end{align*}
where $\delta>0$ has the interpretation of the material thickness, subscript $,s$ denotes the derivative along $S^1$, and $v^\tgl_{,s} = v_{,s} \cdot \frac{\y_{,s}}{|\y_{,s}|}$ reflects the tangential component of $v_{,s}$. Note that $|\y_{,s}| \d s$ means arclength integration along the curve $\y(S^1)$. 
Different from the model proposed in \cite{RuWi12b} we here drop the length element $|\y_{,s}|$ in the last term of the energy $\energy$ and the metric $g$.
This ensures that the highest order terms in the dissipation and the metric are a quadratic form $\Qenergy(v,v)=\delta^3\int_{S^1}|v_{,ss}|^2 \d s$.
\end{definition}
The first integrand in $\energy$ and $g_y$ represents energy dissipation due to stretching or compression of the rod
(indeed, it prefers $\widetilde{\y}_{,s}=\y_{,s}$ or zero tangential stretch $v^\tgl_{,s}=0$), while the second measures (linearized) dissipation due to bending.
With the choice
\begin{align*}
\Y &= C^{1,\alpha}(S^1;\R^2)\,,\quad\alpha\in[0,\tfrac12),\\
\V &= \textstyle\big\{ v \in W^{2,2}(S^1;\R^2)\,|\, \int_{S^1} v(s) \d s = 0 \big\}
\end{align*}
it is shown in \cite[Sec.\,7.2]{RuWi12b} that $g$ as well as $\energy$ satisfy the admissibility conditions from Definitions\,\ref{def:admissibleMetric} and \ref{def:admissibleEnergy} on $\manifold$
for $\eta$ small enough (such that there are constants $c,C>0$ with $C>|\y_{,s}|>c$ for all $\y\in\manifold$).
Indeed, if $\y_j\rightharpoonup\y$ and $\widetilde{\y}_j\rightharpoonup\widetilde{\y}$ in $\V$ as $n\to\infty$, then $\y_j\to\y$ and $\widetilde{\y}_j\to\widetilde{\y}$ strongly in $C^{1,\alpha}(S^1;\R^2)$
so that $\energy^c[\y_j,\widetilde{\y}_j]\to\energy^c[\y,\widetilde{\y}]$ with $\energy^c[\y,\tilde \y] = \energy[\y,\tilde \y] - \Qenergy(\y-\tilde \y,\y-\tilde \y)$.

\begin{remark}[Applicability of model analysis]
Again, $\manifold$ is not admissible in the sense of Definition\,\ref{def:admissibleManifold} since it has a boundary.
However, our analysis stays applicable.
In particular, the proof of existence of continuous spline interpolations from Theorem\,\ref{thm:ExistenceContinuous} still holds
(and also the proof of the regularity estimate as long as the minimizer does not touch the manifold boundary $\partial\manifold$), since $\manifold$ is weakly closed in $\V$.
The proof of existence of discrete spline interpolations from Theorem\,\ref{thm:ExistenceDiscrete} holds for the same reason.
Also the proof of $\Gamma$-convergence from Theorem\,\ref{thm:GammaConvergence} remains unchanged as long as one considers the $\Gamma$-limit in a curve with positive distance from $\partial\manifold$,
and for the convergence of discrete to continuous interpolations (as long as the continuous interpolation has positive distance to $\partial\manifold$) we can proceed as in Remark\,\ref{rem:proofChangesFiniteDimM},
first using the additional constraint of staying at least a distance $\varepsilon>0$ away from $\partial\manifold$ (which is weakly closed) and then letting $\varepsilon\to0$.
\end{remark}

\begin{figure}[h]
\includegraphics[width=\linewidth]{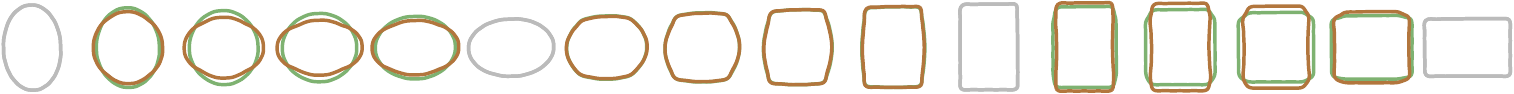}
\caption{Computed discrete geodesic (green) and spline (orange) interpolation between the grey viscous rods.
The influence of the ellipses' curvature on the spline segment between the rectangles is still visible, while for the piecewise geodesic interpolation any memory of the ellipses is lost between the rectangles.
Likewise, memory of the rectangles persists in the spline segment between the ellipses in form of a slight concavity in those places where the rectangle corners were flattened out.}
\label{fig:viscousRodsCircleSquare}
\end{figure}

For numerical computations, viscous rods were discretized using a spectral representation, and the energy $\energy$ was approximated via trapezoidal quadrature.
Figure\,\ref{fig:viscousRodsCircleSquare} shows an overlay of a discrete piecewise geodesic and a discrete spline interpolation with natural boundary conditions.
Here, the influence of the ellipses on the spline segment between the first two rectangles is still slightly visible as the sides of the intermediate rectangles are slightly bent inward (opposite to the curvature of the ellipses).
The piecewise geodesic interpolation does not show this feature.
Also the size change of the different contours is smoother in the spline than the piecewise geodesic.
Figure\,\ref{fig:viscousRodsDogs} shows the same comparison for interpolations between two silhouettes of a running dog.
While in the piecewise geodesic curve there is an abrupt change at each interpolation point between spreading and contracting the legs,
the legs stay spread out or pulled in for a longer time in the spline curve.
Finally, Figure\,\ref{fig:viscousRodsMPEG} computes a discrete periodic spline interpolation between a set of shapes from the MPEG-7 Core Experiment CE-Shape-1 (\url{http://www.cis.temple.edu/~latecki/TestData/mpeg7shapeB.tar.gz}).
It also shows $\energy[\y_k,\widetilde\y_k]$ as a function of $k$, which may be viewed as an approximation to the squared acceleration $g_\y(\cov\dot\y,\cov\dot\y)$.
The plot is consistent with the regularity result in Theorem\,\ref{thm:ExistenceContinuous}, which shows that the second derivatives of the spline interpolation are still H\"older continuous.
The graph is reminiscent of a continuous, piecewise quadratic function, which would be exactly the form of $\|\ddot\y\|^2$ for a cubic spline interpolation $\y$ in the Euclidean case.

\begin{figure}[h]
\setlength\unitlength\linewidth
\begin{picture}(1,.19)
\put(0,0){\includegraphics[width=\unitlength]{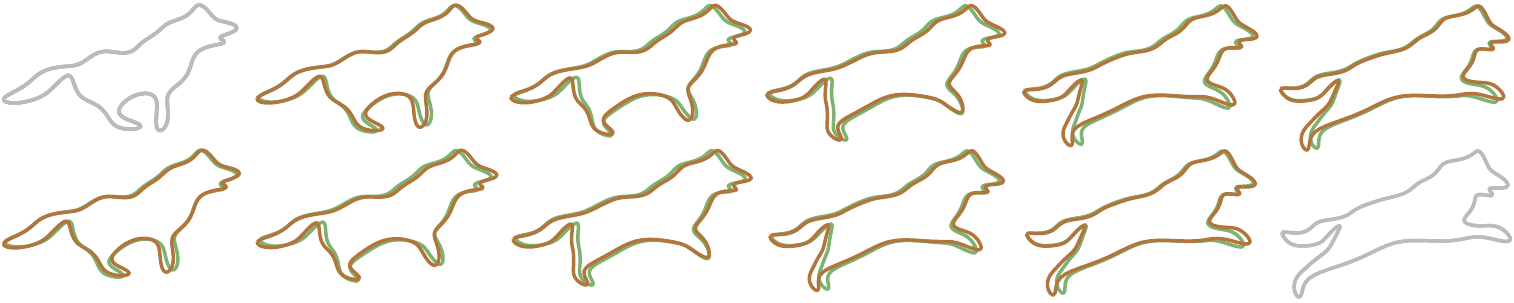}}
\multiput(.14,.14)(.167,0){5}{\includegraphics[width=.035\unitlength]{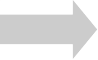}}
\multiput(.15,.05)(.167,0){5}{\includegraphics[width=.035\unitlength,angle=180,origin=c]{arrow}}
\put(.07,.08){\includegraphics[width=.035\unitlength,angle=90,origin=c]{arrow}}
\put(.91,.09){\includegraphics[width=.035\unitlength,angle=-90,origin=c]{arrow}}
\end{picture}
\caption{Computed discrete periodic piecewise geodesic (green) and spline (orange) interpolation between the grey viscous rods.}
\label{fig:viscousRodsDogs}
\end{figure}

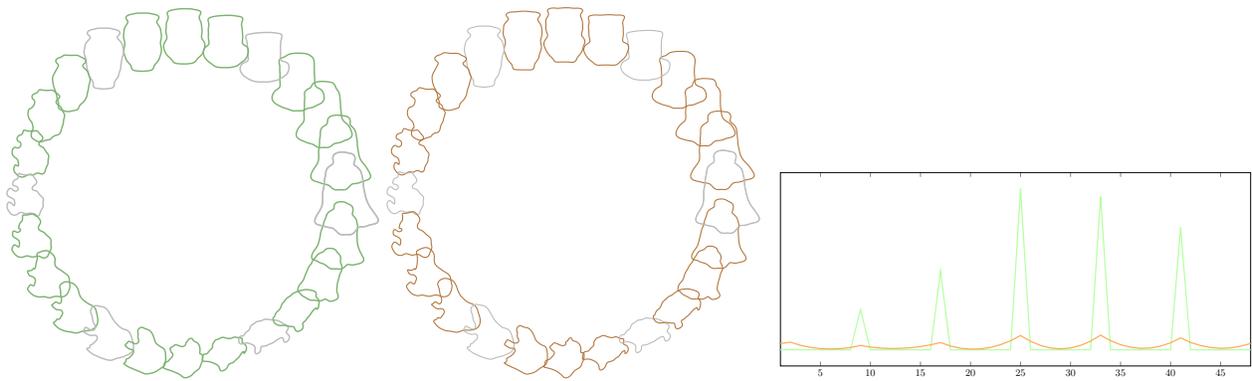
\begin{figure}[h]
\includegraphics[width=.3\linewidth]{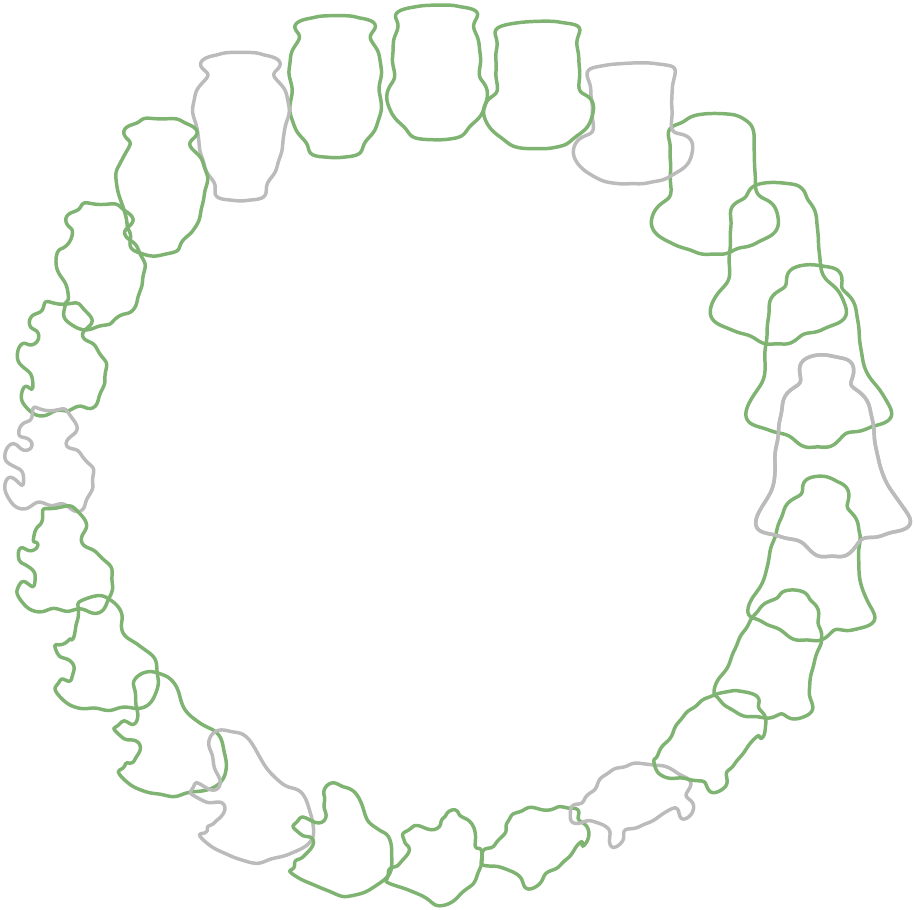}
\includegraphics[width=.3\linewidth]{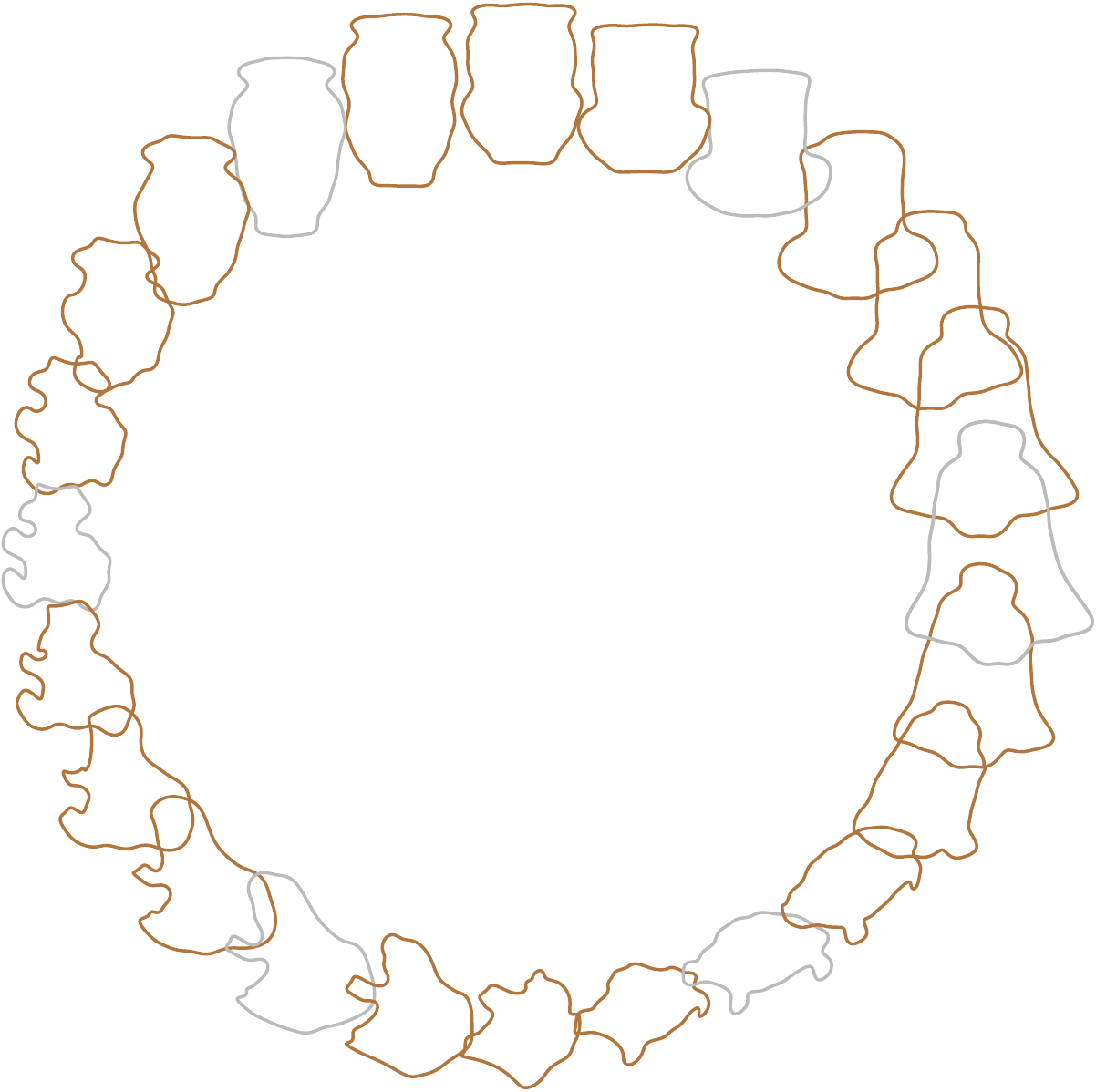}
\definecolor{myOrange}{RGB}{255,169,87}
\definecolor{myGreen}{RGB}{180,255,162}
\begin{tikzpicture}[scale=0.38]
\begin{axis}[x=.35cm,y=.0005cm, xtickmin = 0, xtickmax = 48, enlarge x limits=false,ytick=\empty]
\addplot[color=myGreen, very thick] plot coordinates { (1,5.14689793) (2,5.14761291) (3,0.00694378) (4,0.00211039) (5,0.00060565) (6,0.00031697) (7,0.00263394) (8,1.59483743) (9,2809.41073742) (10,1.59763123) (11,0.00196002) (12,0.00026815) (13,0.00051474) (14,0.00081095) (15,0.00290151) (16,3.11032109) (17,5477.42868347) (18,3.11557123) (19,0.00488327) (20,0.00682075) (21,0.00208920) (22,0.00718740) (23,0.00775329) (24,6.36225359) (25,11218.66031558) (26,6.38371572) (27,0.01375447) (28,0.00849266) (29,0.00652885) (30,0.00904631) (31,0.00819895) (32,6.08450502) (33,10673.38781041) (34,6.03818741) (35,0.00991369) (36,0.00533448) (37,0.00600604) (38,0.01519570) (39,0.00831694) (40,4.80397498) (41,8517.11944282) (42,4.82912068) (43,0.06632377) (44,0.00514074) (45,0.04640732) (46,0.03026716) (47,0.01625112) (48,5.14689793) };
\addplot[color=myOrange, very thick] plot coordinates { (1,440.64997446) (2,535.47294356) (3,314.41136366) (4,176.03284593) (5,86.48768901) (6,58.92976425) (7,85.20078304) (8,164.61080605) (9,306.30171432) (10,193.06784916) (11,125.06711995) (12,94.64734358) (13,102.03101454) (14,146.59121123) (15,227.77600518) (16,348.12875503) (17,507.27513393) (18,267.25851502) (19,115.01044478) (20,48.04822059) (21,66.55814093) (22,170.40377980) (23,358.77728620) (24,633.15524264) (25,996.33119659) (26,594.80966406) (27,304.20481197) (28,130.08182355) (29,71.62679032) (30,129.88341768) (31,309.94626355) (32,605.78372952) (33,1021.22889770) (34,635.95755857) (35,357.96615781) (36,180.27164768) (37,103.59448036) (38,129.55956470) (39,260.04561693) (40,490.39260271) (41,836.66240024) (42,515.74295552) (43,296.97629761) (44,163.69471502) (45,105.63933558) (46,135.05664290) (47,257.03914638) (48,440.64997446) };
\end{axis}
\end{tikzpicture}
\caption{Computed discrete periodic piecewise geodesic (green) and spline (orange) interpolation between the grey viscous rods (left, only half of all intermediate shapes are shown, the grey shapes were taken from the MPEG-7 Core Experiment CE-Shape-1)
and $\energy[\y_k,\widetilde\y_k]$ as a function of $k$ (right).}
\label{fig:viscousRodsMPEG}
\end{figure}

\paragraph{Acknowledgements.} B.\,Heeren and M.\,Rumpf acknowledge support of the Hausdorff Center for
Mathematics and the Collaborative Research Center 1060 funded by the German Research Foundation.
B.\,Wirth's research was supported by the Alfried Krupp Prize for Young University Teachers awarded by the Alfried Krupp von Bohlen und Halbach-Stiftung.
The work was also supported by the German Research Foundation, Cells-in-Motion Cluster of Excellence (EXC1003 -- CiM), University of M\"unster.


\bibliographystyle{alpha}
\bibliography{arxiv}

\newcommand{\etalchar}[1]{$^{#1}$}
\def\polhk#1{\setbox0=\hbox{#1}{\ooalign{\hidewidth
  \lower1.5ex\hbox{`}\hidewidth\crcr\unhbox0}}}
\begin{thebibliography}{HRWW12}

\bibitem[BB11]{BaBr11}
Martin Bauer and Martins Bruveris.
\newblock A new {R}iemannian setting for surface registration.
\newblock In {\em Proc. of MICCAI Workshop on Mathematical Foundations of
  Computational Anatomy}, pages 182--194, 2011.
\newblock arXiv:1106.0620.

\bibitem[Bra02]{Br02}
Andrea Braides.
\newblock {\em {$\Gamma$}-convergence for beginners}, volume~22 of {\em Oxford
  Lecture Series in Mathematics and its Applications}.
\newblock Oxford University Press, Oxford, 2002.

\bibitem[BvTH16]{BrTyHi16}
Christopher Brandt, Christoph von Tycowicz, and Klaus Hildebrandt.
\newblock Geometric flows of curves in shape space for processing motion of
  deformable objects.
\newblock {\em Comput. Graph. Forum}, 35(2):295--305, 2016.

\bibitem[CKPF05]{ChKePo05}
Guillaume Charpiat, Renaud Keriven, Jean-Philippe Pons, and Olivier Faugeras.
\newblock Designing spatially coherent minimizing flows for variational
  problems based on active contours.
\newblock In {\em Computer Vision, ICCV 2005.}, 2005.

\bibitem[CSLC01]{CaSiCr01}
Margarida Camarinha, Fatima Silva~Leite, and Peter Crouch.
\newblock On the geometry of {R}iemannian cubic polynomials.
\newblock {\em Diff. Geom. Appl.}, 15:107--135, 2001.

\bibitem[dB63]{Bo63}
Carl de~Boor.
\newblock Best approximation properties of spline functions of odd degree.
\newblock {\em J. Math. Mech.}, 12:747--749, 1963.

\bibitem[DGM98]{DuGrMi98}
Paul Dupuis, Ulf Grenander, and Michael~I. Miller.
\newblock Variational problems on flows of diffeomorphisms for image matching.
\newblock {\em Quart. Appl. Math.}, 56:587--600, 1998.

\bibitem[Dyn02]{Dy02}
Nira Dyn.
\newblock Interpolatory subdivision schemes.
\newblock In Armin Iske, Ewald Quak, and Michael~S. Floater, editors, {\em
  Tutorials on Multiresolution in Geometric Modelling}, pages 25--50. Springer,
  Berlin, 2002.

\bibitem[GG02]{GiGi02}
Roberto Giambo and Fabio Giannoni.
\newblock An analytical theory for {R}iemmanian cubic polynomials.
\newblock {\em IMA Journal of Mathematical Control and Information},
  19:445--460, 2002.

\bibitem[GHDS03]{GrHiDeSc03}
Eitan Grinspun, Anil~N. Hirani, Mathieu Desbrun, and Peter Schr{\"{o}}der.
\newblock Discrete shells.
\newblock In {\em Proc. of ACM SIGGRAPH/Eurographics Symposium on Computer
  animation}, pages 62--67, 2003.

\bibitem[HMFJ12]{HiMuFl12}
Jacob Hinkle, Prasanna Muralidharan, P.~Thomas Fletcher, and Sarang Joshi.
\newblock Polynomial regression on {R}iemannian manifolds.
\newblock In {\em Proc. of European Conference on Computer Vision}, volume 7574
  of {\em Lecture Notes in Computer Science}, pages 1--14, 2012.

\bibitem[HPR17]{HuPeRu17}
Pascal Huber, Ricardo Perl, and Martin Rumpf.
\newblock Smooth interpolation of key frames in a {R}iemannian shell space.
\newblock {\em Comput. Aided Geom. Design}, 52 - 53:313 -- 328, 2017.

\bibitem[HRS{\etalchar{+}}14]{HeRuSc14}
Behrend Heeren, Martin Rumpf, Peter Schr{\"{o}}der, Max Wardetzky, and Benedikt
  Wirth.
\newblock Exploring the geometry of the space of shells.
\newblock {\em Comput. Graph. Forum}, 33(5):247--256, 2014.

\bibitem[HRS{\etalchar{+}}16]{HeRuSc16}
Behrend Heeren, Martin Rumpf, Peter Schr{\"{o}}der, Max Wardetzky, and Benedikt
  Wirth.
\newblock Splines in the space of shells.
\newblock {\em Comput. Graph. Forum}, 35(5):111--120, 2016.

\bibitem[HRWW12]{HeRuWa12}
Behrend Heeren, Martin Rumpf, Max Wardetzky, and Benedikt Wirth.
\newblock Time-discrete geodesics in the space of shells.
\newblock {\em Comput. Graph. Forum}, 31(5):1755--1764, 2012.

\bibitem[HZN09]{HaZaNi09}
Gabriel~L. Hart, Christopher Zach, and Marc Niethammer.
\newblock An optimal control approach for deformable registration.
\newblock In {\em IEEE Computer Society Conference on Computer Vision and
  Pattern Recognition}, 2009.

\bibitem[KKDS10]{KuKlDi10}
Sebastian Kurtek, Eric Klassen, Zhaohua Ding, and Anuj Srivastava.
\newblock A novel {R}iemannian framework for shape analysis of {3D} objects.
\newblock In {\em Proc. of IEEE Conference on Computer Vision and Pattern
  Recognition}, pages 1625--1632, 2010.

\bibitem[KMP07]{KiMiPo07}
Martin Kilian, Niloy~J. Mitra, and Helmut Pottmann.
\newblock Geometric modeling in shape space.
\newblock {\em ACM Trans. Graph.}, 26(64):1--8, 2007.

\bibitem[LSDM10]{LiShDi10}
Xiuwen Liu, Yonggang Shi, Ivo Dinov, and Washington Mio.
\newblock A computational model of multidimensional shape.
\newblock {\em Int. J. Comput. Vis.}, 89(1):69--83, 2010.

\bibitem[MM06]{MiMu04}
Peter~W. Michor and David Mumford.
\newblock {R}iemannian geometries on spaces of plane curves.
\newblock {\em J. Eur. Math. Soc.}, 8(1):1--48, 2006.

\bibitem[MM07]{MiMu07}
Peter~W. Michor and David Mumford.
\newblock An overview of the {R}iemannian metrics on spaces of curves using the
  {H}amiltonian approach.
\newblock {\em Appl. Comput. Harmon. Anal.}, 23(1):74--113, 2007.

\bibitem[NHP89]{NoHePa89}
Lyle Noakes, Greg Heinzinger, and Brad Paden.
\newblock Cubic splines on curved spaces.
\newblock {\em IMA J. Math. Control Inform.}, 6(4):465--473, 1989.

\bibitem[PH05]{PoHo05}
Helmut Pottmann and Michael Hofer.
\newblock A variational approach to spline curves on surface.
\newblock {\em Comput. Aided Geom. Design}, 22(7):693--709, 2005.

\bibitem[RW15]{RuWi12b}
Martin Rumpf and Benedikt Wirth.
\newblock Variational time discretization of geodesic calculus.
\newblock {\em IMA J. Numer. Anal.}, 35(3):1011--1046, 2015.

\bibitem[SJJK06]{SrJaJo06}
Anuj Srivastava, Aastha Jain, Shantanu~H. Joshi, and David Kaziska.
\newblock Statistical shape models using elastic-string representations.
\newblock In {\em Proc. of Asian Conference on Computer Vision}, volume 3851 of
  {\em Lecture Notes in Computer Science}, pages 612--621, 2006.

\bibitem[SMSY11]{SuMeSo11}
Ganesh Sundaramoorthi, Andrea Mennucci, Stefano Soatto, and Anthony Yezzi.
\newblock A new geometric metric in the space of curves, and applications to
  tracking deforming objects by prediction and filtering.
\newblock {\em SIAM J. Imaging Sci.}, 4(1):109--145, 2011.

\bibitem[SYM07]{SuYeMe07}
Ganesh Sundaramoorthi, Anthony Yezzi, and Andrea Mennucci.
\newblock {S}obolev active contours.
\newblock {\em Int. J. Comput. Vis.}, 73(3):345--366, 2007.

\bibitem[TV12]{TrVi12}
Alain Trouv{\'e} and Fran{\c{c}}ois-Xavier Vialard.
\newblock Shape splines and stochastic shape evolutions : A second order point
  of view.
\newblock {\em Quart. Appl. Math.}, 70(2):219--251, 2012.

\bibitem[VRRC12]{ViRiRu12a}
Fran\c{c}ois-Xavier Vialard, Laurent Risser, Daniel Rueckert, and Colin~J.
  Cotter.
\newblock Diffeomorphic {3D} image registration via geodesic shooting using an
  efficient adjoint calculation.
\newblock {\em International Journal of Computer Vision}, 97:229--241, 2012.

\bibitem[VRRH12]{ViRiRu12}
Fran\c{c}ois-Xavier Vialard, Laurent Risser, Daniel Rueckert, and Darryl~D.
  Holm.
\newblock Diffeomorphic atlas estimation using geodesic shooting on volumetric
  images.
\newblock {\em Annals of the BMVA}, 2012(5):1--12, 2012.

\bibitem[Wal04]{Wa04}
Johannes Wallner.
\newblock Existence of set-interpolating and energy-minimizing curves.
\newblock {\em Comput. Aided Geom. Design}, 21(9):883--892, 2004.

\bibitem[Wal14]{Wa14}
Johannes Wallner.
\newblock On convergent interpolatory subdivision schemes in {R}iemannian
  geometry.
\newblock {\em Constr. Approx.}, 40(3):473--486, 2014.

\bibitem[WD05]{WaDy05}
Johannes Wallner and Nira Dyn.
\newblock Convergence and {$C^1$} analysis of subdivision schemes on manifolds
  by proximity.
\newblock {\em Comput. Aided Geom. Design}, 22(7):593--622, 2005.

\bibitem[YMSM08]{YoMiSh08}
Laurent Younes, Peter~W. Michor, Jayant Shah, and David Mumford.
\newblock A metric on shape space with explicit geodesics.
\newblock {\em Atti Accad. Naz. Lincei Cl. Sci. Fis. Mat. Natur. Rend. Lincei
  (9) Mat. Appl.}, 19(1):25--57, 2008.

\end{thebibliography}

\end{document}